\pdfoutput=1

\documentclass{amsart}

\usepackage{amsthm,amssymb,amsmath,braket,fullpage,mathtools,setspace}
\usepackage[utf8]{inputenc}
\usepackage{tikz}
\usetikzlibrary{arrows,decorations.markings}
\title{Special Functions for Hyperoctahedral Groups Using \\ Bosonic, Trigonometric Six-Vertex Models}
\author{Ben Brubaker}
\author{Will Grodzicki}
\author{Andrew Schultz}
\newcommand{\comment}[1]{}

\newtheorem*{theorem*}{Theorem}
\newtheorem*{maintheorem}{Main Theorem}
\newtheorem*{lemma*}{Lemma}
\newtheorem{proposition}{Proposition}[section]
\newtheorem{theorem}[proposition]{Theorem}
\newtheorem{corollary}[proposition]{Corollary}
\newtheorem{lemma}[proposition]{Lemma}

\theoremstyle{definition}
\newtheorem{definition}{Definition}[section]
\newtheorem*{definition*}{Definition}
\newtheorem*{remark*}{Remark}
\newtheorem*{question*}{Motivating Question}

\numberwithin{equation}{section}

\DeclareRobustCommand{\stasisthree}{\ \ 
\begin{tikzpicture}[scale=.5,baseline={([yshift=-\the\dimexpr\fontdimen22\textfont2\relax]current bounding box.center)}]
\filldraw[white] (0,2.35) circle (3pt);
\filldraw[red] (0,2) circle (3pt);
\filldraw[red] (0,1.5) circle (3pt);
\filldraw[red] (0,1) circle (3pt);
\draw[fill=white] (0,.5) circle (3pt);
\draw[fill=white] (0,0) circle (3pt);
\draw[fill=white] (0,-.5) circle (3pt);
\filldraw[white] (0,-.85) circle (3pt);
\end{tikzpicture}\ \ 
}

\DeclareRobustCommand{\stasistwo}{\ \ 
\begin{tikzpicture}[scale=.5,baseline={([yshift=-\the\dimexpr\fontdimen22\textfont2\relax]current bounding box.center)}]
\filldraw[red] (0,1.5) circle (3pt);
\filldraw[red] (0,1) circle (3pt);
\draw[fill=white] (0,.5) circle (3pt);
\draw[fill=white] (0,0) circle (3pt);
\end{tikzpicture}\ \ 
}

\DeclareRobustCommand{\stasisone}{\ \ 
\begin{tikzpicture}[scale=.5,baseline={([yshift=-\the\dimexpr\fontdimen22\textfont2\relax]current bounding box.center)}]
\filldraw[red] (0,1) circle (3pt);
\draw[fill=white] (0,.5) circle (3pt);
\end{tikzpicture}\ \ 
}
\begin{document}

\begin{abstract}
Recent works have sought to realize certain families of orthogonal, symmetric polynomials as partition functions of well-chosen classes of solvable lattice models. Many of these use Boltzmann weights arising from the trigonometric six-vertex model $R$-matrix (or generalizations or specializations of these weights).  In this paper, we seek new variants of bosonic models on lattices designed for type B/C root systems, whose partition functions match the zonal spherical function in type C. Under general assumptions, we find that this is possible for all highest weights in rank $2$ and $3$, but not for higher rank.
\end{abstract}

\maketitle

\section{Introduction}\label{sec:intro}

A number of recent papers have studied symmetric function theory and related special functions using solvable lattice models, including \cite{BorodinWheelerColored, Borodin-Wheeler-nsmac, hkice, BrubakerSchultz, Buciumas-Scrimshaw, Gorbounov-Korff, Wheeler-Zinn-Justin}. The word ``solvable'' here indicates that the models satisfy certain local relations, principally the ``Yang-Baxter equations'' or ``RTT relations'' (to be described in more detail in the subsequent sections). Such relations are known to arise naturally from modules for quantum groups and it is an interesting problem to associate families of symmetric functions to quantum group modules via the partition functions of solvable lattice models in this way. In the present paper, we explore this connection when the resulting $R$-matrix is, up to scaling, the trigonometric $R$-matrix for the six-vertex model coming from the standard module for the quantum group $U_q(\hat{sl}_2)$ and the lattices are constructed to reflect symmetries of the Weyl group in types B/C -- the so-called hyperoctahedral groups.

The two-dimensional lattice models we consider will consist primarily of finite rectangular arrays of vertices on a square lattice, with four edges (two vertical and two horizontal) adjacent to each vertex. Models for classical Cartan types outside of type $A$ may have additional boundary vertices which roughly reflect the embeddings of classical groups into the general linear group and are inspired by Kuperberg's models for symmetry classes of alternating sign matrices \cite{KuperbergOneRoof}. These will be defined precisely in Section~\ref{sec:type.B}.

Admissible states of these models for given boundary data will consist of paths of particles beginning from the bottom boundary, moving upward and leftward, and exiting out the left boundary. There's an interesting dichotomy here in whether we allow superposition of particles along {\it vertical edges} - such models are called ``bosonic'' models when superposition is allowed, and ``fermionic'' when not allowed. Superposition of particles on horizontal edges will always be forbidden in this context, in keeping with our requirement that the $R$-matrix in the RTT-relation involving a pair of adjacent rows is a six-vertex model. 

We may keep track of initial particle locations along the bottom boundary using an integer partition $\lambda$ whose parts $\lambda_i$ are associated to column indices where particles appear. In particular the parts of $\lambda$ are weakly decreasing in the bosonic case and strictly decreasing in the fermionic case. All other boundaries will be regular, so $\lambda$ suffices to determine the boundary conditions. An example of an admissible state in a type $A$ bosonic model is pictured in Figure~\ref{fig:type.A.example} below. Note in particular the superposition of two particles on a lone edge in the column labeled `2' at the bottom, as depicted by the particle number in the box. Let $\mathcal{S}_\lambda$ denote the set of all admissible states corresponding to the bottom boundary in the model. Full details on these models will be provided in Sections~\ref{sec:type.A} and~\ref{sec:type.B}.

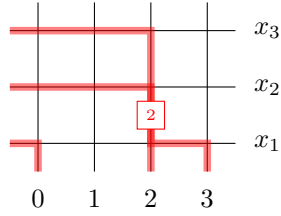
\begin{figure}[ht]
\begin{tikzpicture}[scale=.75, baseline=.5ex]
\node [label=right:$x_3$] at (4,1) {};
\node [label=right:$x_2$] at (4,0) {};
\node [label=right:$x_1$] at (4,-1) {};

\node [label=below:$0$] at (0.5,-1.5) {};
\node [label=below:$1$] at (1.5,-1.5) {};
\node [label=below:$2$] at (2.5,-1.5) {};
\node [label=below:$3$] at (3.5,-1.5) {};

\draw [-] (0,1) -- (4,1);
\draw [-] (0,0) -- (4,0);
\draw [-] (0,-1) -- (4,-1);

\draw [-] (0.5,1.5) -- (0.5,-1.5);
\draw [-] (1.5,1.5) -- (1.5,-1.5);
\draw [-] (2.5,1.5) -- (2.5,-1.5);
\draw [-] (3.5,1.5) -- (3.5,-1.5);

\draw [-,line width=1mm,red,opacity=.5] (3.5,-1.5) -- (3.5,-1) -- (2.5,-1) -- (2.5,1) -- (0,1);
%\draw [-,line width=1mm,red,opacity=.5] (3.5,1) -- (3.5,-1.5);

\draw [-,line width=1mm,red,opacity=.5] (0,0)  -- (2.5,0) -- (2.5,-1.5);
%\draw [-,line width=1mm,red,opacity=.5] (1.5,0) -- (1.5,-1);
%\draw [-,line width=1mm,red,opacity=.5] (1.5,-1) -- (2.5,-1);
%\draw [-,line width=1mm,red,opacity=.5] (2.5,-1) -- (2.5,-1.5);

\draw [-,line width=1mm,red,opacity=.5] (0.5,-1.5) -- (0.5,-1) -- (0,-1);
%\draw [-,line width=1mm,red,opacity=.5] (3.5,1) -- (3.5,-1.5);
\node [rectangle,draw,red,fill=white] at (2.5,-.5) {\tiny{$2$}};
%\node [rectangle,draw,red,fill=white] at (2.5,.5) {\tiny{$1$}};

\end{tikzpicture}
\caption{A sample state for the type $A$ model when $\lambda = (3,2,0) $}\label{fig:type.A.example}
\end{figure}

Broadly speaking, the goal is to represent families of symmetric functions indexed by partitions $\lambda$ as generating functions on the admissible states $\mathcal{S}_\lambda$. These generating functions come from choices of local weights, the ``Boltzmann weights'' at each vertex depending on occupancy in adjacent edges. Taking the product of weights at all vertices in the admissible state, and summing these products over all admissible states in $\mathcal{S}_\lambda$, results in a generating function over states called the ``partition function,'' denoted here as $\mathcal{Z}(\mathcal{S}_\lambda)$. As indicated in the example above, each row of the model is associated to an indeterminate $x_i$ (the so-called ``spectral parameters") and the Boltzmann weights of vertices in the row will be allowed to depend on these $x_i$. In this way, we will see that $\mathcal{Z}(\mathcal{S}_\lambda)$ will be a polynomial in the parameters $x_i$ according to the choice of Boltzmann weights.  

The symmetric functions we seek to represent as partition functions of lattice models are a particular family of orthogonal polynomials in several variables $\boldsymbol{x}:=(x_1, \ldots, x_r)$ indexed by $\lambda \in \Lambda_r$ associated to affine root systems. More precisely, they are one parameter specializations of the two parameter Macdonald polynomials $P_\lambda^{(\epsilon)}(\boldsymbol{x})$ defined in Section 5.7 of \cite{Macdonald}. If the affine Weyl group associated to the root system is $W = W_0 \ltimes \Lambda_r$ with $W_0$ the finite Weyl group and $\Lambda_r$ the weight lattice of rank $r$, then the family of polynomials in $r$ variables $\boldsymbol{x}$ and a linear character $\epsilon$ of $W_0$. In our specialization, the $P_\lambda^{(\epsilon)}$ are constructed from certain Hecke algebra symmetrizers (cf.~(5.5.6) in \cite{Macdonald}) built from $\epsilon$ as sums over $W_0$, and acting on monomials $\boldsymbol{x}^\lambda = x_1^{\lambda_1} \cdots x_r^{\lambda_r}$. Our primary example will be the case where $\epsilon$ is the trivial character, where $P_\lambda^{(\epsilon)}$ is the usual two parameter Macdonald polynomial and in our specialization to one parameter $q$ is Macdonald's zonal spherical function.\footnote{Macdonald polynomials are commonly written in terms of parameters $q$ and $t$, though conventions differ about their roles. Our conventions match those of $p$-adic representation theory. In particular, the usual role of $t$ in the literature is played by $q$ for us, owing to its connection with the cardinality of the residue field of the $p$-adic field which is typically notated with the same letter. No confusion will arise since our polynomials have a single deformation parameter and its role will always appear through explicit formulas like~(\ref{eq:hlpoly}).} In type $A$, this is also known as the Hall-Littlewood polynomial. For any root system and associated finite Weyl group $W_0$, it takes the form:
\begin{equation} P_\lambda(x_1, \ldots, x_r; q) = \frac{1}{Q_\lambda(q)} \sum_{\sigma \in W_0} \sigma \left( \boldsymbol{x}^\lambda \prod_{\alpha \in \Phi^+} \frac{1 - q \boldsymbol{x}^{-\alpha}}{1 - \boldsymbol{x}^{-\alpha}} \right) \label{eq:hlpoly}, \end{equation}
for some polynomial $Q_\lambda$. For other choices of linear character $\epsilon$, we achieve similar averaging formulas over $W_0$; see for example Section 8 of~\cite{BBFMatrix} for a discussion of this. Attempting to use lattice models derived from quantum group modules to represent Macdonald polynomials made from Hecke algebras is perhaps natural in light of Jimbo's generalization of Schur-Weyl duality \cite{Jimbo}. Making this connection precise in the language of partition functions, viewed as matrix coefficients of quantum group modules, is an interesting open question. At present, we don't even know whether solvable lattice models exist for all such polynomials in one deformation parameter (nevermind two). This paper is an attempt to understand this question better for classical groups, and we begin by reviewing what is known.

In type $A$, the connections between such orthogonal polynomials and solvable lattice models are well developed. The Weyl groups $W_0$ of simply laced types have just two linear characters - the sign and trivial characters - which produce spherical Whittaker functions and Hall-Littlewood polynomials, respectively. Families of solvable lattice models whose partition functions are spherical Whittaker functions and the aforemnetioned Hall-Littlewood polynomials appeared previously in the literature in \cite{hkice} and \cite{Wheeler-Zinn-Justin}, respectively. The Whittaker functions use a fermionic lattice model while those for Hall-Littlewood polynomials are bosonic. In fact, this connection between symmetric functions and lattice models extends to the nonsymmetric analogues. Indeed if we distinguish each path in an admissible state of the model by recording its color, then the resulting lattice models with a fixed set of colors along the boundary realize the non-symmetric analogues above. These are the non-symmetric Hall-Littlewood polynomials and Iwahori-Whittaker functions, which were studied in \cite{BorodinWheelerColored} and \cite{BBBGIwahori}, respectively. These latter Whittaker functions arise in representation theory of $p$-adic algebraic groups; further specializing the one parameter in the Boltzmann weights, they degenerate to so-called ``Demazure atoms'' as in \cite{MasonAtoms}, originally called ``standard bases'' by Lascoux and Sch\"{u}tzenberger in \cite{LascouxSchutzenbergerKeys}, and their connection to lattice models is explored in \cite{demice}.

In types $B$ and $C$, less is known. There are results for spherical Whittaker functions -- that is, $P_\lambda^{(\epsilon)}$ with $\epsilon$ the sign character -- in~\cite{Ivanov} using fermionic lattice models with bends along one side and bivalent vertices at each bend (Figure~\ref{fig:rank.2.example} gives one such example). In \cite{Wheeler-Zinn-Justin}, bosonic models with the same underlying lattice shape are shown to give certain two-parameter generalizations of polynomials $P_\lambda^{(\epsilon)}$ where $\epsilon$ is the character of $W_0$ sending long roots to 1 and short roots to $-1$. The two parameters, called $\gamma$ and $\delta$, arise from a pair of ``bonus'' columns introduced into the lattice models next to the bends. The Boltzmann weights for vertices in each of these two extra columns depend on one of these parameters, but setting $\gamma=\delta=0$ is equivalent to omitting these two bonus columns entirely from the model. In this case, according to (64) in~\cite{Wheeler-Zinn-Justin}, one obtains $P_\lambda^{(\epsilon)}$ for this character $\epsilon$ (see Theorem \ref{thm:WZJ}). If instead one chooses $\gamma$ and $\delta$ with $\gamma + \delta = 0$ and $\gamma \delta = q$, this results in the zonal spherical function in type $C$. So it is natural to ask which special functions are achievable from a solvable lattice model with bends in which there are no ``bonus'' columns, and for which the rectangular (tetravalent) vertices satisfy a Yang-Baxter equation with the  trigonometric six-vertex model. This is the main question we try to answer in the present paper. In particular, we address whether one can obtain the zonal spherical function in these non-simply-laced types via such solvable lattice models.

As we alluded to above, the lattice models in types $B$ and $C$ consist almost entirely of tetravalent vertices, together with a set of bivalent vertices appearing as bends connecting pairs of adjacent rows along one side of the model. The example in Figure~\ref{fig:rank.2.example} shows one such admissible state in rank two. Thus our question becomes one of exploring choices of these Boltzmann weights at the bend vertices (or ``bend weights,'' for short) which simultaneously preserve the solvability of the model and achieve our desired special functions as partition functions of the model. In doing so, we are forced to consider bend weights which are {\bf not uniformly defined} in terms of spectral parameters. For example in rank two models, the function defining the Boltzmann weight for the top bend may differ from the one for the bottom bend. Details are given in Section~\ref{sec:type.B}.

In later sections of the paper, we provide conditions on the bend weights to guarantee the solvability of the model. This leads to the following result.

\begin{maintheorem} Let $\mathcal{B}_\lambda$ be a solvable lattice model of type $B/C$ with boundary conditions corresponding to a dominant weight $\lambda$, and having Boltzmann weights from the bosonic, trigonometric six-vertex model at all tetravalent vertices. Then the existence of bend weights such that $\mathcal{Z}(\mathcal{B}_\lambda) = P_\lambda(\boldsymbol{x}; q)$, the zonal spherical function in type $C$, for all dominant weights $\lambda$ depends on the rank as follows:
\begin{itemize}
    \item For rank $r \leq 3$, there exists a choice of Boltzmann weights for the bend vertices that realizes the zonal spherical function at dominant weights $\lambda$. These bend weights depend on the row, and are not uniformly dependent on the (spectral) parameters $x_i$.
    \item In rank $r \geq 4$, no such choice of Boltzmann weights for bend vertices exists.
\end{itemize}
\end{maintheorem}

We briefly outline the contents of each section and their role in leading up to this main result. In Section~\ref{sec:type.A}, we fix notation, recall the definition of triangular weights and review the concept of solvability (the existence of a Yang-Baxter equation for tetravalent vertices) for bosonic models in type $A$. In Section~\ref{sec:type.B}, we extend this notion of solvability to type $B/C$ models. See Definition~\ref{def:solvable} for a precise definition of solvability for lattice models of types $B/C$. The definition ensures that the row-to-row transfer matrices (taken in pairs of rows connected by a bend vertex) commute, as shown in Lemmas~\ref{le:inversion} and~\ref{le:symmetric}. Solvability in these types requires certain additional relations, as has been long understood in lattice models of this geometry \cite{BrubakerSchultz, Ivanov, KuperbergOneRoof, Tsuchiya, Wheeler-Zinn-Justin}. Following Kuperberg, these relations tend to be named after animals -- fish, caduceus, etc. We prove necessary conditions on bend weights to satisfy the required relations in a series of lemmas in Section~\ref{sec:type.B}. In particular, we find new solutions to the caduceus relation (e.g., Lemma~\ref{le:caduceus}) using non-uniform choices of bend weights. Section~\ref{sec:genmethods} provides methods for explicit evaluations of solvable lattice models in type $B/C$, in the style of \cite{Wheeler-Zinn-Justin}. The later sections of the paper then apply these methods to lattice models of various ranks, as the analysis is rather different in each case, and results in the main theorem above. In the process, we evaluate all solvable lattice models in rank at most $3$ with trigonometric six-vertex model weights at tetravalent vertices.

Our main theorem handles the most general set of bend weights that is consistent with the paradigm of partition functions as matrix coefficients for quantum group modules -- in particular, we require the caduceus equation for bend weights to be consistent with Cherednik and Sklyanin's reflection equation and we require monomial dependence in the spectral parameters $x_i$, consistent with formal group laws for parametrized Yang-Baxter equations from families of affine quantum group representations. Thus any solution for bend weights outside of our assumptions would represent a break from this point of view. It is possible that our analysis could be used as a starting point for developing combinatorial solutions outside of this paradigm, and one might hope the resulting algebraic structures would be of fundamental interest.
 This work was partially supported by NSF grant DMS-2101392 (Brubaker).

\section{Bosonic Models for Cartan Type $A$}\label{sec:type.A}

We begin by describing the general features of six-vertex lattice models in type $A$, all of which will be needed in subsequent sections on other classical Cartan types. Each lattice model is indexed by an integer partition $\lambda = (\lambda_1 \geqslant \lambda_2 \geqslant \cdots \geqslant \lambda_r)$ (with $\lambda_r \geq 0$) which determines its size, shape, and boundary conditions. Given such a $\lambda$, each lattice model will be built from two-dimensional grids of tetravalent vertices with $n$ rows, numbered $1$ to $n$ from bottom to top, and $\lambda_1+1$ columns, numbered $0$ to $\lambda_1$ from left to right. An admissible state of such a model is a configuration of $r$ paths along the lattice edges, with the $i$-th path starting from the $\lambda_i$-th column along the bottom boundary edges and ending on an edge at the far left side of the diagram. At each vertex, paths are allowed to either move up or left. Paths are allowed to overlap without restriction along vertical edges, but any given horizontal edge can allow at most one path. Alternatively, we could view columns between each row as recording the position of a family of particles (or a finite window of interest among a semi-infinite collection of particles) and the admissible states record a discrete time evolution of these particles with each new row. The fact that columns may contain multiple paths then translates as the superposition of these particles, so we refer to the associated models as ``bosonic'' (as opposed to ``fermionic'' where no superposition is permitted). We will denote the set of all such (admissible) states as $\mathcal{A}_\lambda$. 

Ultimately we will attach a polynomial to each particular model.  To do this, each vertex $v$ is assigned a Boltzmann weight $B(v)$; that weight will be a function depending on the edge configurations adjacent to that vertex $v$ and on a transcendental parameter $x_i$ corresponding to the row $i$ in which the vertex occurs.  For an admissible state $S$ of the system, its weight $\textrm{wt}(S)$ is then defined as the product of weights over all its vertices: $\textrm{wt}(S)= \prod_{v \in S} B(v)$.  The partition function $\mathcal{Z}(\mathcal{A}_\lambda)$ for the model is then recovered by summing these weights over all states of the model: $$\mathcal{Z}(\mathcal{A}_\lambda) = \sum_{S \in \mathcal{A}_\lambda} \textrm{wt}(S).$$
This is a slight abuse of notation, as we've continued to use $\mathcal{A}_\lambda$ rather than introduce a new notation for the set of admissible states {\it together} with an associated collection of Boltzmann weights. We repeat this often in what follows when the underlying Boltzmann weights are clear from context.

Thus to specify the partition function $\mathcal{Z}(\mathcal{A}_\lambda)$ it remains to define the Boltzmann weights. The Boltzmann weights at each vertex are given in the table in Figure~\ref{fig:six.vertex.model}. They agree with the weights in Equation (35) of \cite{Wheeler-Zinn-Justin} and arise naturally from evolution operators on bosonic Fock space made by creation and annihilation of particles (see Section 2 of \cite{Wheeler-Zinn-Justin} for details). For us, the key point is that these weights satisfy certain relations known as ``Yang-Baxter equations'' or ``$RTT$ relations,'' which we address shortly below.

\begin{figure}[!ht]
\begin{tabular}
{|c|c|c|c|c|c|c|c|}
\hline
%&&&&&\\
$\begin{tikzpicture}
\coordinate (left) at (0.25,1);
\coordinate (right) at (1.75,1);
\coordinate (top) at (1,1.75);
\coordinate (bottom) at (1,0.25);
\node [label=left:$x_j$] (L) at (left) {};
\node [label=right:$ $] (R) at (right) {};
\draw [-] (L.east) -- (R.west);
\draw [-] (top) -- (bottom);
\draw [-,line width=1mm,red,opacity=.5] (top) -- (bottom);
\node [rectangle,draw,red,fill=white] (T) at (1,1.35) {\tiny{$m$}};
\node [rectangle,draw,red,fill=white] (T) at (1,.65) {\tiny{$m$}};
\end{tikzpicture}$
&
$\begin{tikzpicture}
\coordinate (left) at (0.25,1);
\coordinate (right) at (1.75,1);
\coordinate (top) at (1,1.75);
\coordinate (bottom) at (1,0.25);
\node [label=left:$x_j$] (L) at (left) {};
\node [label=right:$ $] (R) at (right) {};
\draw [-] (L.east) -- (R.west);
\draw [-] (top) -- (bottom);
\draw [-,line width=1mm,red,opacity=.5] (top) -- (bottom);
\draw [-,line width=1mm,red,opacity=.5] (L) -- (R);
\node [rectangle,draw,red,fill=white] (T) at (1,1.35) {\tiny{$m$}};
\node [rectangle,draw,red,fill=white] (B) at (1,.65) {\tiny{$m$}};
\end{tikzpicture}$
&
$\begin{tikzpicture}
\coordinate (left) at (0.25,1);
\coordinate (right) at (1.75,1);
\coordinate (top) at (1,1.75);
\coordinate (bottom) at (1,0.25);
\node [label=left:$x_j$] (L) at (left) {};
\node [label=right:$ $] (R) at (right) {};
\node [] at (bottom) {};
\draw [-] (L.east) -- (R.west);
\draw [-] (top) -- (bottom);
\draw [-,line width=1mm,red,opacity=.5] (top) -- (bottom);
\draw [-,line width=1mm,red,opacity=.5] (1,1) -- (R);
\node [rectangle,draw,red,fill=white] (T) at (1,1.35) {\tiny{$m$}};
\node [rectangle,draw,red,fill=white] (B) at (1,.65) {\tiny{$m-1$}};
\end{tikzpicture}$
&
$\begin{tikzpicture}
\coordinate (left) at (0.25,1);
\coordinate (right) at (1.75,1);
\coordinate (top) at (1,1.75);
\coordinate (bottom) at (1,0.25);
\node [label=left:$x_j$] (L) at (left) {};
\node [label=right:$ $] (R) at (right) {};
\node []  at (bottom) {};
\node []  at (top) {};
\draw [-] (L.east) -- (R.west);
\draw [-] (top) -- (bottom);
\draw [-,line width=1mm,red,opacity=.5] (top) -- (bottom);
\draw [-,line width=1mm,red,opacity=.5] (L) -- (1,1);
\node [rectangle,draw,red,fill=white] (T) at (1,1.35) {\tiny{$m$}};
\node [rectangle,draw,red,fill=white] (B) at (1,.65) {\tiny{$m+1$}};
\end{tikzpicture}$
\\
%&&&&&\\ 
\hline
$1$&$x_j$&$1-q^m$&$x_j$\\ \hline
\end{tabular}
\caption{Rectangular lattice Boltzmann weights - the boxed integer $m$ denotes the multiplicity of paths along the associated vertical edge.}\label{fig:six.vertex.model}
\end{figure}
This choice of Boltzmann weights gives the following explicit evaluations of the partition function:  

\begin{theorem}[Wheeler-Zinn-Justin \cite{Wheeler-Zinn-Justin}, Thm. 2] For any partition $\lambda$ and weights as in Figure~\ref{fig:six.vertex.model}, then  $$ \mathcal{Z}(\mathcal{A}_\lambda) = \left( x_1 \cdots x_r \right) P_\lambda(\boldsymbol{x};q) $$ where $P_\lambda$ denotes the Hall-Littlewood polynomial as in (\ref{eq:hlpoly}).
\end{theorem}

To prove this, one must demonstrate that the partition function $\mathcal{Z}(\mathcal{A}_\lambda)$ satisfies certain symmetries.  In this case, we want to ensure that $\mathcal{Z}(\mathcal{A}_\lambda)$ is left unchanged after an exchange of variables $x_i \leftrightarrow x_{i+1}$. Recalling that the weights depend on $x_i$ in row $i$, this is equivalent to saying the partition function is symmetric upon switching the roles of the rows $i$ and $i+1$, a property known as ``commuting transfer matrices'' for the row-to-row transfer matrix. 

Following Baxter \cite{Baxter}, one may do this by taking advantage of certain ``local symmetries" enjoyed by these Boltzmann weights. Consider an auxiliary family of Boltzmann weights for the six so-called ``$R$-vertices,'' or $R$-matrix weights depicted in Figure \ref{fig:R.weights} below. Their graphical depiction reflects the fact that each such vertex is the intersection of two horizontal edges, and hence each adjacent edge may admit at most one path. Their Boltzmann weights are a scaled version of the entries in the trigonometric $R$-matrix of the six-vertex model, now depending on the pair of parameters $x_j$ and $x_k$ on the respective rows. In addition to providing the explicit description of each weight in the bottom row of the table, the middle row gives an abstract name for each vertex, to be used in later computations.

\begin{figure}[!ht]
\begin{tabular}
{|c|c|c|c|c|c|}
\hline
$\begin{tikzpicture}[scale=0.5]
\node [label=left:$x_j$] at (0,1) {};
\node [label=left:$x_k$] at (0,0) {};
\node [label=right:$ $] at (2,1) {};
\node [label=right:$ $] at (2,0) {};
\node [label=above:$ $] at (0,1) {};
\node [label=below:$ $] at (0,0) {};
\draw [-] (0,1) .. controls (1,1) and (1,0) .. (2,0);
\draw [-] (0,0) .. controls (1,0) and (1,1) .. (2,1);
\draw [-,line width=1mm,red,opacity=0.5] (0,1) .. controls (1,1) and (1,0) .. (2,0);
\draw [-,line width=1mm,red,opacity=0.5] (0,0) .. controls (1,0) and (1,1) .. (2,1);
\end{tikzpicture}$
&
$\begin{tikzpicture}[scale=0.5]
\node [label=left:$x_j$] at (0,1) {};
\node [label=left:$x_k$] at (0,0) {};
\node [label=right:$ $] at (2,1) {};
\node [label=right:$ $] at (2,0) {};
\node [label=above:$ $] at (0,1) {};
\node [label=below:$ $] at (0,0) {};
\draw [-] (0,1) .. controls (1,1) and (1,0) .. (2,0);
\draw [-] (0,0) .. controls (1,0) and (1,1) .. (2,1);
%\draw [-,line width=1mm,red,opacity=0.5] (0,1) .. controls (1,1) and (1,0) .. (2,0);
%\draw [-,line width=1mm,red,opacity=0.5] (0,0) .. controls (1,0) and (1,1) .. (2,1);
\end{tikzpicture}$
&
$\begin{tikzpicture}[scale=0.5]
\node [label=left:$x_j$] at (0,1) {};
\node [label=left:$x_k$] at (0,0) {};
\node [label=right:$ $] at (2,1) {};
\node [label=right:$ $] at (2,0) {};
\node [label=above:$ $] at (0,1) {};
\node [label=below:$ $] at (0,0) {};
\draw [-] (0,1) .. controls (1,1) and (1,0) .. (2,0);
\draw [-] (0,0) .. controls (1,0) and (1,1) .. (2,1);
\draw [-,line width=1mm,red,opacity=0.5] (0,1) .. controls (1,1) and (1,0) .. (2,0);
\end{tikzpicture}$
&
$\begin{tikzpicture}[scale=0.5]
\node [label=left:$x_j$] at (0,1) {};
\node [label=left:$x_k$] at (0,0) {};
\node [label=right:$ $] at (2,1) {};
\node [label=right:$ $] at (2,0) {};
\node [label=above:$ $] at (0,1) {};
\node [label=below:$ $] at (0,0) {};
\draw [-] (0,1) .. controls (1,1) and (1,0) .. (2,0);
\draw [-] (0,0) .. controls (1,0) and (1,1) .. (2,1);
\draw [-,line width=1mm,red,opacity=0.5] (0,0) .. controls (1,0) and (1,1) .. (2,1);
\end{tikzpicture}$
&
$\begin{tikzpicture}[scale=0.5]
\node [label=left:$x_j$] at (0,1) {};
\node [label=left:$x_k$] at (0,0) {};
\node [label=right:$ $] at (2,1) {};
\node [label=right:$ $] at (2,0) {};
\node [label=above:$ $] at (0,1) {};
\node [label=below:$ $] at (0,0) {};
\draw [-] (0,1) .. controls (1,1) and (1,0) .. (2,0);
\draw [-] (0,0) .. controls (1,0) and (1,1) .. (2,1);
\draw [-,line width=1mm,red,opacity=0.5] (0,1) .. controls (1,1) and (1,0) .. (1,0.5);
\draw [-,line width=1mm,red,opacity=0.5] (1,0.5) .. controls (1,0) and (1,1) .. (2,1);
\end{tikzpicture}$
&
$\begin{tikzpicture}[scale=0.5]
\node [label=left:$x_j$] at (0,1) {};
\node [label=left:$x_k$] at (0,0) {};
\node [label=right:$ $] at (2,1) {};
\node [label=right:$ $] at (2,0) {};
\node [label=above:$ $] at (0,1) {};
\node [label=below:$ $] at (0,0) {};
\draw [-] (0,1) .. controls (1,1) and (1,0) .. (2,0);
\draw [-] (0,0) .. controls (1,0) and (1,1) .. (2,1);
\draw [-,line width=1mm,red,opacity=0.5] (1,0.5) .. controls (1,1) and (1,0) .. (2,0);
\draw [-,line width=1mm,red,opacity=0.5] (0,0) .. controls (1,0) and (1,1) .. (1,0.5);
\end{tikzpicture}$
\\ \hline
$a_2(k,j)$ & $a_1(k,j)$ & $b_2(k,j)$ & $b_1(k,j)$ & $c_2(k,j)$ & $c_1(k,j)$
\\ \hline
$1$ & $1$ & $\frac{q(x_k-x_j)}{x_k-qx_j}$ & $\frac{x_k-x_j}{x_k-qx_j}$ & $\frac{(1-q)x_j}{x_k-qx_j}$&$\frac{(1-q)x_k}{x_k-qx_j}$ \\ 
\hline
\end{tabular}
\caption{$R$-matrix weights from the trigonometric six-vertex model}\label{fig:R.weights}
\end{figure}

\begin{proposition}[Yang-Baxter equation, RTT relation]\label{le:YBE}
For any choice of edge labels $\alpha,\beta,\gamma,\delta,\epsilon,\phi$ and Boltzmann weights as in Figures~\ref{fig:six.vertex.model} and~\ref{fig:R.weights}, one has the equality of partition functions
\begin{equation*}\label{eq:YBE}
\mathcal{Z}
\left(
\begin{minipage}{1.1in}
\begin{tikzpicture}
\node [label=left:$x_j$] at (0,1) {};
\node [label=left:$x_k$] at (0,0) {};
%Boundary condition
\node [circle,draw,scale=.6] at (1.5,1.75) {$\phi$};
\node [circle,draw,scale=0.6] at (1.5,-.75) {$\gamma$};
\node [circle,draw,scale=.6] at (.25,1) {$\alpha$};
\node [circle,draw,scale=0.55] at (.25,0) {$\beta$};
\node [circle,draw,scale=.6] at (2.25,1) {$\epsilon$};
\node [circle,draw,scale=0.6] at (2.25,0) {$\delta$};

\draw [-] (1.5,1.5) -- (1.5,-.5);
\draw [-] (1.5,1) -- (2,1);
\draw [-] (1.5,0) -- (2,0);

\draw [-] (1.5,0) .. controls (1,0) and (1,1) .. (.5,1);
\draw [-] (1.5,1) .. controls (1,1) and (1,0) .. (.5,0);

\end{tikzpicture}
\end{minipage}
\quad \right) = 
\mathcal{Z}
\left(
\begin{minipage}{1.1in}
\begin{tikzpicture}
\node [label=left:$x_j$] at (0,1) {};
\node [label=left:$x_k$] at (0,0) {};
%Boundary condition
\node [circle,draw,scale=.6] at (1,1.75) {$\phi$};
\node [circle,draw,scale=0.6] at (1,-.75) {$\gamma$};
\node [circle,draw,scale=.6] at (.25,1) {$\alpha$};
\node [circle,draw,scale=0.55] at (.25,0) {$\beta$};
\node [circle,draw,scale=.6] at (2.25,1) {$\epsilon$};
\node [circle,draw,scale=0.6] at (2.25,0) {$\delta$};

\draw [-] (1,1.5) -- (1,-.5);
\draw [-] (.5,1) -- (1,1);
\draw [-] (.5,0) -- (1,0);

\draw [-] (2,0) .. controls (1.5,0) and (1.5,1) .. (1,1);
\draw [-] (2,1) .. controls (1.5,1) and (1.5,0) .. (1,0);
\end{tikzpicture}
\end{minipage}
\quad \right).
\end{equation*}
\end{proposition}
\noindent This may be checked explicitly, as the weights in Figure~\ref{fig:six.vertex.model} are uniformly expressed in terms of the number of particles $m$. Such an identity was asserted in Equation~3.3 of~\cite{BogBull} and in Equation~(14) of~\cite{Wheeler-Zinn-Justin}. 

In addition to the previous result, it will be useful to know the following result; compare Equation (11) in~\cite{Wheeler-Zinn-Justin}.  Its proof is a direct computation and left to the reader.
\begin{proposition}[Unitarity relation]\label{thm:unitarity}
For any choice of edge labels $\alpha$ and $\beta$ and using Boltzmann weights as in Figure \ref{fig:R.weights}, one has
\[
\mathcal{Z}\left(
\begin{tikzpicture}[scale=1,baseline={([yshift=-\the\dimexpr\fontdimen22\textfont2\relax]current bounding box.center)}]
\node [label=left:$x_j$] at (-.25,1) {};
\node [label=left:$x_k$] at (-.25,0) {};
%Boundary condition
\node [circle,draw,scale=.6] (Lalpha) at (0,1) {$\alpha$};
\node [circle,draw,scale=0.55] (Lbeta) at (0,0) {$\beta$};
\node [circle,draw,scale=.6] (Ralpha) at (2,1) {$\alpha$};
\node [circle,draw,scale=0.55] (Rbeta) at (2,0) {$\beta$};

\draw [-] (Lalpha) .. controls (.5,1) and (.5,0) .. (1,0) .. controls (1.5,0) and (1.5,1) .. (Ralpha);
\draw [-] (Lbeta) .. controls (.5,0) and (.5,1) .. (1,1) .. controls (1.5,1) and (1.5,0) .. (Rbeta);
\end{tikzpicture}
\right)
=1.
\]
\end{proposition}

We say that a type $A$ model is ``solvable'' if its weights possess a solution $R$ to the Yang-Baxter equation above. It implies that the row-to-row transfer matrices commute, via the now familiar ``train argument,'' and hence the partition function $\mathcal{Z}(\mathcal{A}_\lambda)$ is symmetric under the exchange of variables $x_i \leftrightarrow x_{i+1}$. A graphical depiction of this argument is shown in Figure~\ref{fig:train.type.A}, which highlights both the repeated use of the Yang-Baxter equation and that the symmetry of $\mathcal{Z}$ hinges on the fact that the $R$ matrix weights $a_1(k,k+1)=a_2(k,k+1)$ for all $k$. Note that the figure doesn't attempt to depict ``generic'' boundary conditions (in-coming paths) according to $\lambda$ along the bottom boundary.

\begin{figure}[h]
\begin{align*}
&\mathcal{Z}\left(\ 
\begin{minipage}{1.15in}
\begin{centering}
\begin{tikzpicture}[scale=.5, baseline=.5ex, every node/.style={transform shape}]
\node [label=right:\normalsize{$x_r$}] at (4,2) {};
\node [label=right:\normalsize{$\vdots$}] at (4,1) {};
\node [label=right:\normalsize{$x_2$}] at (4,0) {};
\node [label=right:\normalsize{$x_1$}] at (4,-1) {};
\node [label=below:\normalsize{$0$}] at (-0.5,-1.5) {};
\node [label=below:\normalsize{$1$}] at (0.5,-1.5) {};
\node [label=below:\normalsize{$\cdots$}] at (1.5,-1.5) {};
\node [label=below:\normalsize{$\lambda_1-1$}] at (2.5,-1.5) {};
\node [label=below:\normalsize{$\lambda_1$}] at (3.5,-1.5) {};
\draw [-] (-1,2) -- (1,2);
\draw [densely dotted] (1,2) -- (2,2);
\draw [-] (2,2) -- (4,2);
\draw [-] (-1,0) -- (1,0);
\draw [densely dotted] (1,0) -- (2,0);
\draw [-] (2,0) -- (4,0);
\draw [-] (-1,-1) -- (1,-1);
\draw [densely dotted] (1,-1) -- (2,-1);
\draw [-] (2,-1) -- (4,-1);
\draw [-] (-0.5,2.5) -- (-.5,1.5);
\draw[densely dotted] (-.5,1.5) -- (-.5,0.5);
\draw[-] (-.5,0.5) -- (-0.5,-1.5);
\draw [-] (0.5,2.5) -- (.5,1.5);
\draw[densely dotted] (.5,1.5) -- (.5,0.5);
\draw[-] (.5,0.5) -- (0.5,-1.5);
\draw [-] (2.5,2.5) -- (2.5,1.5);
\draw[densely dotted] (2.5,1.5) -- (2.5,0.5);
\draw[-] (2.5,0.5) -- (2.5,-1.5);
\draw [-] (3.5,2.5) -- (3.5,1.5);
\draw[densely dotted] (3.5,1.5) -- (3.5,0.5);
\draw[-] (3.5,0.5) -- (3.5,-1.5);
\filldraw[white] (-1,2) circle (4pt);
\filldraw[white] (-1,0) circle (4pt);
\filldraw[white] (-1,-1) circle (4pt);
\filldraw[red,opacity=.5] (-1,2) circle (4pt);
\filldraw[red,opacity=.5] (-1,0) circle (4pt);
\filldraw[red,opacity=.5] (-1,-1) circle (4pt);
\end{tikzpicture}
\end{centering}
\end{minipage}
\right)
\textrm{wt}\left(\ 
\begin{minipage}{.55in}
\begin{centering}
\begin{tikzpicture}[scale=.5,baseline=1.5ex, every node/.style={transform shape}]
%\node [label=left:$$] at (0,1) {};
%\node [label=left:$$] at (0,0) {};
\node [label=right:\normalsize{$x_1$}] at (2,1) {};
\node [label=right:\normalsize{$x_2$}] at (2,0) {};
\draw [-] (0,1) .. controls (1,1) and (1,0) .. (2,0);
\draw [-] (0,0) .. controls (1,0) and (1,1) .. (2,1);
\end{tikzpicture}
\end{centering}
\end{minipage}
\right)\!=\mathcal{Z}\left(
\begin{minipage}{1.3in}
\begin{centering}
\begin{tikzpicture}[scale=.5, baseline=.5ex, every node/.style={transform shape}]
\node [label=right:\normalsize{$x_r$}] at (4,2) {};
\node [label=right:\normalsize{$\vdots$}] at (4,1) {};
\node [label=right:\normalsize{$x_1$}] at (5,0) {};
\node [label=right:\normalsize{$x_2$}] at (5,-1) {};
\node [label=below:\normalsize{$0$}] at (-0.5,-1.5) {};
\node [label=below:\normalsize{$1$}] at (0.5,-1.5) {};
\node [label=below:\normalsize{$\cdots$}] at (1.5,-1.5) {};
\node [label=below:\normalsize{$\lambda_1-1$}] at (2.5,-1.5) {};
\node [label=below:\normalsize{$\lambda_1$}] at (3.5,-1.5) {};
\draw [-] (-1,2) -- (1,2);
\draw [densely dotted] (1,2) -- (2,2);
\draw [-] (2,2) -- (4,2);
\draw [-] (-1,0) -- (1,0);
\draw [densely dotted] (1,0) -- (2,0);
\draw [-] (2,0) -- (4,0);
\draw [-] (-1,-1) -- (1,-1);
\draw [densely dotted] (1,-1) -- (2,-1);
\draw [-] (2,-1) -- (4,-1);
\draw [-] (4,0) .. controls (4.5,0) and (4.5,-1) .. (5,-1);
\draw [-] (4,-1) .. controls (4.5,-1) and (4.5,0) .. (5,0);
\draw [-] (-0.5,2.5) -- (-.5,1.5);
\draw[densely dotted] (-.5,1.5) -- (-.5,0.5);
\draw[-] (-.5,0.5) -- (-0.5,-1.5);
\draw [-] (0.5,2.5) -- (.5,1.5);
\draw[densely dotted] (.5,1.5) -- (.5,0.5);
\draw[-] (.5,0.5) -- (0.5,-1.5);
\draw [-] (2.5,2.5) -- (2.5,1.5);
\draw[densely dotted] (2.5,1.5) -- (2.5,0.5);
\draw[-] (2.5,0.5) -- (2.5,-1.5);
\draw [-] (3.5,2.5) -- (3.5,1.5);
\draw[densely dotted] (3.5,1.5) -- (3.5,0.5);
\draw[-] (3.5,0.5) -- (3.5,-1.5);
\filldraw[white] (-1,2) circle (4pt);
\filldraw[white] (-1,0) circle (4pt);
\filldraw[white] (-1,-1) circle (4pt);
\filldraw[red,opacity=.5] (-1,2) circle (4pt);
\filldraw[red,opacity=.5] (-1,0) circle (4pt);
\filldraw[red,opacity=.5] (-1,-1) circle (4pt);
\end{tikzpicture}
\end{centering}
\end{minipage}
\right)\!=\mathcal{Z}\left(
\begin{minipage}{1.15in}
\begin{centering}
\begin{tikzpicture}[scale=.5, baseline=.5ex, every node/.style={transform shape}]
\node [label=right:\normalsize{$x_r$}] at (4,2) {};
\node [label=right:\normalsize{$\vdots$}] at (4,1) {};
\node [label=right:\normalsize{$x_1$}] at (4,0) {};
\node [label=right:\normalsize{$x_2$}] at (4,-1) {};
\node [label=below:\normalsize{$0$}] at (-0.5,-1.5) {};
\node [label=below:\normalsize{$1$}] at (0.5,-1.5) {};
\node [label=below:\normalsize{$\cdots$}] at (1.5,-1.5) {};
\node [label=below:\normalsize{$\lambda_1-1$}] at (2.5,-1.5) {};
\node [label=below:\normalsize{$\lambda_1$}] at (3.5,-1.5) {};
\draw [-] (-1,2) -- (4,2);
%\draw [-] (0,1) -- (4,1);
\draw [-] (-1,0) -- (2.5,0);
\draw [-] (-1,-1) -- (2.5,-1);
\draw [-] (-1,2) -- ++(.5,0);
\draw [-] (-1,0) -- ++(.5,0);
\draw [-] (-1,-1) -- ++(.5,0);
\draw [-] (2.5,0) .. controls (3,0) and (3,-1) .. (3.5,-1);
\draw [-] (2.5,-1) .. controls (3,-1) and (3,0) .. (3.5,0);
\draw [-] (3.5,0) -- (4,0);
\draw [-] (3.5,-1) -- (4,-1);
\draw [-] (-0.5,2.5) -- (-0.5,-1.5);
\draw [-] (0.5,2.5) -- (0.5,-1.5);
%\draw [-] (1.5,1.5) -- (1.5,-1.5);
\draw [-] (2.5,2.5) -- (2.5,-1.5);
\draw [-] (3.5,2.5) -- (3.5,-1.5);
%\draw [-,line width=1mm,red,opacity=.5] (3.5,-1.5) -- (3.5,-1) -- (2.5,-1) -- (2.5,1) -- (0,1);
%\draw [-,line width=1mm,red,opacity=.5] (3.5,1) -- (3.5,-1.5);
%\draw [-,line width=1mm,red,opacity=.5] (0,0)  -- (2.5,0) -- (2.5,-1.5);
%\draw [-,line width=1mm,red,opacity=.5] (1.5,0) -- (1.5,-1);
%\draw [-,line width=1mm,red,opacity=.5] (1.5,-1) -- (2.5,-1);
%\draw [-,line width=1mm,red,opacity=.5] (2.5,-1) -- (2.5,-1.5);
%\draw [-,line width=1mm,red,opacity=.5] (0.5,-1.5) -- (0.5,-1) -- (0,-1);
%\draw [-,line width=1mm,red,opacity=.5] (3.5,1) -- (3.5,-1.5);
\filldraw[white] (-1,2) circle (4pt);
\filldraw[white] (-1,0) circle (4pt);
\filldraw[white] (-1,-1) circle (4pt);
\filldraw[red,opacity=.5] (-1,2) circle (4pt);
\filldraw[red,opacity=.5] (-1,0) circle (4pt);
\filldraw[red,opacity=.5] (-1,-1) circle (4pt);
\end{tikzpicture}
\end{centering}
\end{minipage}
\right)\!=\!\cdots
\\[5pt]&=
\mathcal{Z}\left(
\begin{minipage}{1.15in}
\begin{centering}
\begin{tikzpicture}[scale=.5, baseline=.5ex, every node/.style={transform shape}]
\node [label=right:\normalsize{$x_r$}] at (4,2) {};
\node [label=right:\normalsize{$\vdots$}] at (4,1) {};
\node [label=right:\normalsize{$x_1$}] at (4,0) {};
\node [label=right:\normalsize{$x_2$}] at (4,-1) {};
\node [label=below:\normalsize{$0$}] at (-0.5,-1.5) {};
\node [label=below:\normalsize{$1$}] at (0.5,-1.5) {};
\node [label=below:\normalsize{$\cdots$}] at (1.5,-1.5) {};
\node [label=below:\normalsize{$\lambda_1-1$}] at (2.5,-1.5) {};
\node [label=below:\normalsize{$\lambda_1$}] at (3.5,-1.5) {};
\draw [-] (-1,2) -- (4,2);
%\draw [-] (0,1) -- (4,1);
\draw [-] (-1,0) -- (-0.5,0);
\draw [-] (-1,-1) -- (-0.5,-1);
\draw [-] (-1,2) -- ++(.5,0);
\draw [-] (-1,0) -- ++(.5,0);
\draw [-] (-1,-1) -- ++(.5,0);
\draw [-] (-0.5,0) .. controls (0,0) and (0,-1) .. (0.5,-1);
\draw [-] (-0.5,-1) .. controls (0,-1) and (0,0) .. (0.5,0);
\draw [-] (0.5,0) -- (4,0);
\draw [-] (0.5,-1) -- (4,-1);
\draw [-] (-0.5,2.5) -- (-0.5,-1.5);
\draw [-] (0.5,2.5) -- (0.5,-1.5);
%\draw [-] (1.5,1.5) -- (1.5,-1.5);
\draw [-] (2.5,2.5) -- (2.5,-1.5);
\draw [-] (3.5,2.5) -- (3.5,-1.5);
\filldraw[white] (-1,2) circle (4pt);
\filldraw[white] (-1,0) circle (4pt);
\filldraw[white] (-1,-1) circle (4pt);
\filldraw[red,opacity=.5] (-1,2) circle (4pt);
\filldraw[red,opacity=.5] (-1,0) circle (4pt);
\filldraw[red,opacity=.5] (-1,-1) circle (4pt);
\end{tikzpicture}
\end{centering}
\end{minipage}
\right)
=
\mathcal{Z}\left(\ 
\begin{minipage}{1.3in}
\begin{centering}
\begin{tikzpicture}[scale=.5, baseline=.5ex, every node/.style={transform shape}]
\node [label=right:\normalsize{$x_r$}] at (4,2) {};
\node [label=right:\normalsize{$\vdots$}] at (4,1) {};
\node [label=right:\normalsize{$x_1$}] at (4,0) {};
\node [label=right:\normalsize{$x_2$}] at (4,-1) {};
\node [label=below:\normalsize{$0$}] at (-0.5,-1.5) {};
\node [label=below:\normalsize{$1$}] at (0.5,-1.5) {};
\node [label=below:\normalsize{$\cdots$}] at (1.5,-1.5) {};
\node [label=below:\normalsize{$\lambda_1-1$}] at (2.5,-1.5) {};
\node [label=below:\normalsize{$\lambda_1$}] at (3.5,-1.5) {};
\draw [-] (-2,0) .. controls (-1.5,0) and (-1.5,-1) .. (-1,-1);
\draw [-] (-2,-1) .. controls (-1.5,-1) and (-1.5,0) .. (-1,0);
\draw [-,line width=1mm,red,opacity=.5] (-2,0) .. controls (-1.5,0) and (-1.5,-1) .. (-1,-1);
\draw [-,line width=1mm,red,opacity=.5] (-2,-1) .. controls (-1.5,-1) and (-1.5,0) .. (-1,0);
\draw [-] (-1,2) -- (4,2);
\draw [-] (-1,0) -- (4,0);
\draw [-] (-1,-1) -- (4,-1);
\draw [-] (-1,2) -- ++(.5,0);
\draw [-,line width=1mm,red,opacity=.5] (-1,0) -- ++(.5,0);
\draw [-,line width=1mm,red,opacity=.5] (-1,-1) -- ++(.5,0);
\draw [-] (-0.5,2.5) -- (-0.5,-1.5);
\draw [-] (0.5,2.5) -- (0.5,-1.5);
\draw [-] (2.5,2.5) -- (2.5,-1.5);
\draw [-] (3.5,2.5) -- (3.5,-1.5);
\filldraw[white] (-1,2) circle (4pt);
\filldraw[red,opacity=.5] (-1,2) circle (4pt);
\end{tikzpicture}
\end{centering}
\end{minipage}
\right)
=
\textrm{wt}\left(\ 
\begin{minipage}{.55in}
\begin{centering}
\begin{tikzpicture}[scale=.5, baseline=1.5ex, every node/.style={transform shape}]
%\node [label=left:$$] at (0,1) {};
%\node [label=left:$$] at (0,0) {};
\node [label=right:\normalsize{$x_1$}] at (2,1) {};
\node [label=right:\normalsize{$x_2$}] at (2,0) {};
\draw [-,line width=1mm,red,opacity=.5] (0,1) .. controls (1,1) and (1,0) .. (2,0);
\draw [-,line width=1mm,red,opacity=.5] (0,0) .. controls (1,0) and (1,1) .. (2,1);
\draw [-] (0,1) .. controls (1,1) and (1,0) .. (2,0);
\draw [-] (0,0) .. controls (1,0) and (1,1) .. (2,1);
\end{tikzpicture}
\end{centering}
\end{minipage}
\right)\mathcal{Z}\left(\ 
\begin{minipage}{1.15in}
\begin{centering}
\begin{tikzpicture}[scale=.5, baseline=.5ex, every node/.style={transform shape}]
\node [label=right:\normalsize{$x_r$}] at (4,2) {};
\node [label=right:\normalsize{$\vdots$}] at (4,1) {};
\node [label=right:\normalsize{$x_1$}] at (4,0) {};
\node [label=right:\normalsize{$x_2$}] at (4,-1) {};
\node [label=below:\normalsize{$0$}] at (-0.5,-1.5) {};
\node [label=below:\normalsize{$1$}] at (0.5,-1.5) {};
\node [label=below:\normalsize{$\cdots$}] at (1.5,-1.5) {};
\node [label=below:\normalsize{$\lambda_1-1$}] at (2.5,-1.5) {};
\node [label=below:\normalsize{$\lambda_1$}] at (3.5,-1.5) {};
\draw [-] (-1,2) -- (4,2);
\draw [-] (-1,0) -- (4,0);
\draw [-] (-1,-1) -- (4,-1);
\draw [-] (-1,2) -- ++(.5,0);
\draw [-] (-1,0) -- ++(.5,0);
\draw [-] (-1,-1) -- ++(.5,0);
\draw [-] (-0.5,2.5) -- (-0.5,-1.5);
\draw [-] (0.5,2.5) -- (0.5,-1.5);
\draw [-] (2.5,2.5) -- (2.5,-1.5);
\draw [-] (3.5,2.5) -- (3.5,-1.5);
\filldraw[white] (-1,2) circle (4pt);
\filldraw[white] (-1,0) circle (4pt);
\filldraw[white] (-1,-1) circle (4pt);
\filldraw[red,opacity=.5] (-1,2) circle (4pt);
\filldraw[red,opacity=.5] (-1,0) circle (4pt);
\filldraw[red,opacity=.5] (-1,-1) circle (4pt);
\end{tikzpicture}
\end{centering}
\end{minipage}
\right)
\end{align*}
\caption{The Train Argument}
\label{fig:train.type.A}
\end{figure}
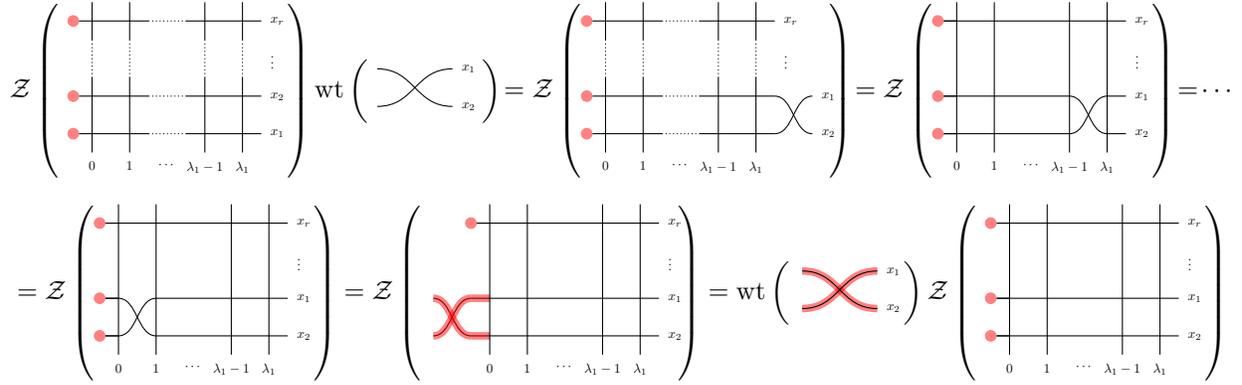

In the above figure, and throughout the remainder of the paper, we write ``\textrm{wt}(S)'' when evaluating the Boltzmann weight of a single admissible state or lattice configuration and we write $\mathcal{Z}(\mathcal{S})$ when evaluating the partition function of the set of admissible states $\mathcal{S}$. 

\section{Models for Cartan Type $B$ and $C$\label{sec:type.B}}

% In this paper, we will be working with an integrable model of deformed bosons. In particular,

Our results pertain to solvable lattice models for root systems of Cartan type $B$ and $C$, so named because solvability implies an action of the Weyl group of type $B/C$ on the partition function. These models are similar to the Type $A$ models presented in Section \ref{sec:type.A} in many ways. Our Type $B/C$ model has vertices in a rectangular lattice dictated by a partition $\lambda = (\lambda_1 \geqslant \lambda_2 \geqslant \cdots \geqslant \lambda_r)$ -- with $\lambda_1+1$ columns numbered in ascending order from left to right, and with $2r$ rows whose associated spectral parameters are labeled in pairs $x_i^{-1}=: \overline{x}_i$ and $x_i$ in ascending order from bottom to top. The lattice is then augmented with a set of $r$ bivalent vertices connecting the rows corresponding to the spectral parameters $x_i$ and $\overline{x}_i$ for each $i$, forming u-turn bends along the left-hand boundary of the model.

\begin{figure}[h]
\begin{tikzpicture}[scale=.75, baseline=.5ex]
\node [label=right:$x_2$] at (4,2) {};
\node [label=right:$\overline{x}_2$] at (4,1) {};
\node [label=right:$x_1$] at (4,0) {};
\node [label=right:$\overline{x}_1$] at (4,-1) {};
\node [label=right:0] at (0,-2) {};
\node [label=right:1] at (1,-2) {};
\node [label=right:2] at (2,-2) {};
\node [label=right:3] at (3,-2) {};
\node [label=left:{$\mathbf{2}$}] at (-.75,1.5) {};
\node [label=left:{$\mathbf{1}$}] at (-.75,-.5) {};
\coordinate (top) at (-0.5,1.5);
\coordinate (bot) at (-0.5,-.5);

\draw [-]
	(top) arc (180:90:.5);
\draw [-]
	(top) arc (180:270:.5);

\draw [-]
	(bot) arc (180:90:.5);
\draw [-]
	(bot) arc (180:270:.5);

\draw [-] (0,2) -- (4,2);
\draw [-] (0,1) -- (4,1);
\draw [-] (0,0) -- (4,0);
\draw [-] (0,-1) -- (4,-1);

\draw [-] (0.5,2.5) -- (0.5,-1.5);
\draw [-] (1.5,2.5) -- (1.5,-1.5);
\draw [-] (2.5,2.5) -- (2.5,-1.5);
\draw [-] (3.5,2.5) -- (3.5,-1.5);

\draw [-,line width=1mm,red,opacity=.5]
    (top) arc (180:270:.5);
\draw [-,line width=1mm,red,opacity=.5, rounded corners=1pt] (0,1) -- (3.5,1) -- (3.5, -1.5);

\draw [-,line width=1mm,red,opacity=.5]
    (bot) arc (180:90:.5);
\draw [-,line width=1mm,red,opacity=.5,rounded corners=1pt] (0,0) -- (1.5,0) -- (1.5,-1) -- (2.5, -1) -- (2.5, -1.5);

\node [circle,draw=black, fill=black, inner sep=0pt,minimum size=5pt] at (top)  {};
\node [circle,draw=black, fill=black, inner sep=0pt,minimum size=5pt] at (bot) {};
\end{tikzpicture}
\caption{A sample state for the type $B/C$ model when $\lambda = (3,2)$.}\label{fig:rank.2.example}
\end{figure}
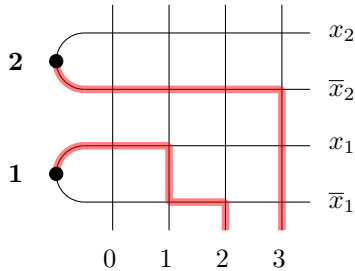

An admissible state in the rank $r$ type $B/C$ model is a configuration of $r$ paths, with the $i$th path starting from the $\lambda_i$-th edge along the bottom and moving upward and leftward through the lattice, ending at a bivalent vertex along the u-turn bend. An example of an admissible state in rank 2 is shown in Figure~\ref{fig:rank.2.example}. The set of all admissible configurations for the lattice with boundary conditions determined by $\lambda$ will be denoted $\mathcal{B}_\lambda$. It is natural to use these admissible states to represent special functions for type $B/C$ root systems. Indeed, there is a natural bijection between admissible states in $\mathcal{B}_\lambda$ and certain arrays of interleaving integers generalizing Gelfand-Tsetlin patterns in that they arise naturally from branching rules for symplectic and odd orthogonal groups (see~\cite{Proctor} for details and references).

There is further motivation for these geometries from quantum groups, where Boltzmann weights of tetravalent vertices represent matrix coefficients in $R$-matrices in $\textrm{End}(V(x) \otimes W)$ for a pair of quantum group modules $V(x),W$ representing the horizontal and vertical edges, respectively, while the bivalent vertices at the u-turn bends represent so-called $K$-matrices in $\textrm{Hom}(V(x_i) \otimes V(\bar{x}_i), \mathbb{C}).$ These have been featured in solvable lattice models for symmetric functions previously, for example in \cite{KuperbergOneRoof, Tsuchiya, Betea-Wheeler, Wheeler-Zinn-Justin}. We require the bare minimum set of conditions on the bend weights to guarantee that the resulting partition function has hyperoctahedral symmetry - these are the so-called ``fish'' and ``caduceus'' relations to be presented in this section after fully introducing the lattice model. Our ``caduceus'' relation will turn out to be equivalent to the ``reflection equation'' of Sklyanin \cite{Sklyanin}; see also Section~2.7 of~\cite{Wheeler-Zinn-Justin} where the $K$-matrix is referred to as the boundary covector.

At each of the tetravalent vertices in the rectangular part of the lattice model, we use the same Boltzmann weights from Figure~\ref{fig:six.vertex.model}, just as in the Type $A$ model. The remaining Boltzmann weights for the bivalent vertices at each bend remain to be determined. We have recorded the possible admissible configurations in Figure~\ref{fig:gen.bends}, but have only represented those weights with very generic labels. In particular, for any row index $j \in [1,r]$ we allow four possible configurations with respective bend weights $A_j, B_j, C_j, D_j$ as in Figure~\ref{fig:gen.bends}, which are each functions (potentially distinct for each $j$) taking values in the polynomial ring $\mathbb{Z}[q](x_k, \overline{x}_k)$ where $x_k$ is the associated spectral parameter for the pair of rows. We will momentarily determine conditions on these bend weights in order that the type $B/C$ models are solvable.

Just as in type $A$, when the choice of Boltzmann weights is understood for a set of admissible configurations $\mathcal{B}_\lambda$, we refer to the set of states with associated weights as the ``model'' associated to $\lambda$ and continue to refer to the model as $\mathcal{B}_\lambda$.
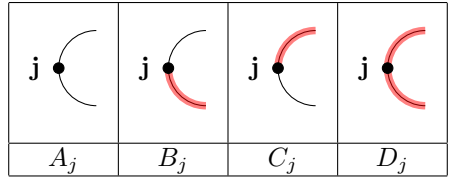
\begin{figure}[!ht]
\begin{tabular}
{|c|c|c|c|}
\hline
%&&&&&\\
$\begin{tikzpicture}[scale=0.5]
\node [label=above:$ $] at (0,1) {};
\node [label=below:$ $] at (0,-1) {};
\node [label=left: {$\mathbf{j}$}] at (-1,0) {};
\draw [-]
	(0,1) arc (90:270:1);
\filldraw[black] (-1,0) circle (4pt);
\end{tikzpicture}$
&
$\begin{tikzpicture}[scale=0.5]
\node [label=above:$ $] at (0,1) {};
\node [label=below:$ $] at (0,-1) {};
\node [label=left: {$\mathbf{j}$}] at (-1,0) {};
\draw [-]
	(0,1) arc (90:270:1);
\draw [-,line width=1mm,red,opacity=0.5]
	(-1,0) arc (180:270:1);
\filldraw[black] (-1,0) circle (4pt);
\end{tikzpicture}$
&
$\begin{tikzpicture}[scale=0.5]
\node [label=above:$ $] at (0,1) {};
\node [label=below:$ $] at (0,-1) {};
\node [label=left: {$\mathbf{j}$}] at (-1,0) {};
\draw [-]
	(0,1) arc (90:270:1);
\draw [-,line width=1mm,red,opacity=0.5]
	(0,1) arc (90:180:1);
\filldraw[black] (-1,0) circle (4pt);
\end{tikzpicture}$
&
$\begin{tikzpicture}[scale=0.5]
\node [label=above:$ $] at (0,1) {};
\node [label=below:$ $] at (0,-1) {};
\node [label=left: {$\mathbf{j}$}] at (-1,0) {};
\draw [-]
	(0,1) arc (90:270:1);
\draw [-,line width=1mm,red,opacity=0.5]
	(0,1) arc (90:270:1);
\filldraw[black] (-1,0) circle (4pt);
\end{tikzpicture}$
\\
%&&&&&\\ 
\hline
$A_j$&$B_j$&$C_j$&$D_j$\\ \hline
\end{tabular}
\caption{Generic Bend Weights}\label{fig:gen.bends}
\end{figure}

\noindent With these definitions in place, we can again create a partition function by summing the weights of all admissible configurations: $$\mathcal{Z}(\mathcal{B}_\lambda) = \sum_{S \in \mathcal{B}_\lambda} \textrm{wt}(S).$$

As in the case of $\mathcal{Z}(\mathcal{A}_\lambda)$, we may evaluate $\mathcal{Z}(\mathcal{B}_\lambda)$ by understanding its behavior under action of the Weyl group. In order to exhibit this action, we need both the Yang-Baxter equation (Proposition \ref{le:YBE}) as well as two additional identities: the Fish relation (Lemma \ref{le:fish}) and the Caduceus relation (Lemmas \ref{le:caduceus}, \ref{le:caduceus3}, and~\ref{le:caduceus1}). These lemmas establish all non-trivial bend weight choices among monomials in the $x_i^{\pm1}$ in the Laurent polynomial ring $\mathbb{Z}[q](x_1^{\pm 1},\ldots,x_r^{\pm 1})$ which achieve these relations. The degree $deg(M)$ of a monomial $M$ in $\mathbb{Z}[q](x^{\pm1})$ is the integer exponent of $x$. In each of the subsequent lemmas, it is sometimes convenient to identify the set of edge labels (thus far described by paths) with the non-negative integers. In particular, horizontal edges can only have $\{0,1\}$ (no path, one path, respectively) as potential labels.

\begin{lemma}[Fish relation]\label{le:fish}
Let the $R$-matrix Boltzmann weights be given as in Figure~\ref{fig:R.weights} with spectral parameters $x$ and $\overline{x}$. We further assume that the bend weights are monomial in $x, \bar{x}$ with coefficients in $\mathbb{C}[q]$. Then the quantities

\begin{equation}\label{eq:fish}
\mathcal{Z}
\left(
%\begin{minipage}{1.2in}
\begin{tikzpicture}[scale=1,baseline={([yshift=-\the\dimexpr\fontdimen22\textfont2\relax]current bounding box.center)}]
\node [label=right:$\overline{x}$] at (0.5,1) {};
\node [label=right:$x$] at (0.5,0) {};
%Boundary condition
\node [circle,draw,scale=.6] at (.25,1) {$\alpha$};
\node [circle,draw,scale=0.55] at (.25,0) {$\beta$};
\node [label=left: {$\mathbf{j}$}] at (-1.5,0.5) {};

\draw [-] (0,0) .. controls (-0.5,0) and (-.5,1) .. (-1,1);
\draw [-] (0,1) .. controls (-0.5,1) and (-.5,0) .. (-1,0);

\draw [-]
	(-1,1) arc (90:180:.5);
\filldraw[black] (-1.5,.5) circle (2pt);
\draw [-]
	(-1.5,.5) arc (180:270:.5);
\end{tikzpicture}
%\end{minipage}
\right) \quad \textrm{and} \quad \emph{wt}\left(
%\begin{minipage}{.8in}
\begin{tikzpicture}[scale=1,baseline={([yshift=-\the\dimexpr\fontdimen22\textfont2\relax]current bounding box.center)}]
\node [label=right:$\overline{x}$] at (.5,1) {};
\node [label=right:$x$] at (.5,0) {};
\node [circle,draw,scale=.6] at (.25,1) {$\alpha$};
\node [circle,draw,scale=0.55] at (.25,0) {$\beta$};
\node [label=left: {$\mathbf{j}$}] at (-0.5,0.5) {};

\draw [-]
	(0,1) arc (90:180:.5);
\filldraw[black] (-.5,.5) circle (2pt);
\draw [-]
	(-.5,0.5) arc (180:270:.5);
\end{tikzpicture}
%\end{minipage}
\right)
\end{equation}
are proportional with proportionality constant $F$ in $\mathbb{Z}[q](x^{\pm1})$ independent of the choice of boundary labels $\alpha$ and $\beta$ in $\{0,1\}$ if and only if one of the following conditions on the Boltzmann weights at bivalent vertices in Figure~\ref{fig:gen.bends} hold:
\begin{enumerate} \item Either $\textrm{deg}(B_j(x,\bar{x})) = \textrm{deg}(C_j(x,\bar{x}))$ and either:
\begin{enumerate}
    \item $B_j(x,\bar{x})= C_j(x,\bar{x})$, in which case the constant of proportionality $F=1$. In this case, $deg(A_j(x,\bar{x}))=deg(D_j(x,\bar{x}))=0$; or
    \item  $B_j(x,\bar{x})=-q C_j(x,\bar{x})$ with $F =(x^2-q)/(1-qx^2)$. In this case, $A_j(x,\bar{x})=D_j(x,\bar{x})=0$.
\end{enumerate}
   \item Or $\textrm{deg}(B_j(x,\bar{x})) = \textrm{deg}(C_j(x,\bar{x})) + 2$ and either:
   \begin{enumerate}
    \item $B_j(x,\bar{x}) = -x^2 C_j(x,\bar{x})$, with $F=(x^2-q)/(1-qx^2)$, and then $A_j(x,\bar{x})=D_j(x,\bar{x})=0$; or
    \item  $B_j(x,\bar{x}) = q x^2 C_j(x,\bar{x})$ with $F=1$ and then $deg(A_j(x,\bar{x}))=deg(D_j(x,\bar{x}))=0$. 
\end{enumerate}
\end{enumerate}
\end{lemma}

\begin{proof} In what follows, we abbreviate $A_j(x,\bar{x})$ by $A_j(x)$ and $A_j(\bar{x},x)$ by $A_j(\bar{x})$, with similar abbreviations for the functions $B_j, C_j$ and $D_j$. Let's begin by considering the case where $B_j, C_j \neq 0$. In the case $(\alpha,\beta) = (1,0)$, we have 
\begin{equation}\label{eq:ratio1} 
\mathcal{Z} \left(
%\begin{minipage}{1.2in}
\begin{tikzpicture}[scale=1,baseline={([yshift=-\the\dimexpr\fontdimen22\textfont2\relax]current bounding box.center)}]
\node [label=right:$\overline{x}$] at (0.5,1) {};
\node [label=right:$x$] at (0.5,0) {};
%Boundary condition
\node [circle,draw,scale=.6] at (.25,1) {$1$};
\node [circle,draw,scale=0.55] at (.25,0) {$0$};
\node [label=left: {$\mathbf{j}$}] at (-1.5,0.5) {};

\draw [-] (0,0) .. controls (-0.5,0) and (-.5,1) .. (-1,1);
\draw [-] (0,1) .. controls (-0.5,1) and (-.5,0) .. (-1,0);

\draw [-]
	(-1,1) arc (90:180:.5);
\filldraw[black] (-1.5,.5) circle (2pt);
\draw [-]
	(-1.5,.5) arc (180:270:.5);
\end{tikzpicture}
%\end{minipage}
\right) \left/ \textrm{wt}
\left(
%\begin{minipage}{.8in}
\begin{tikzpicture}[scale=1,baseline={([yshift=-\the\dimexpr\fontdimen22\textfont2\relax]current bounding box.center)}]
\node [label=right:$\overline{x}$] at (.5,1) {};
\node [label=right:$x$] at (.5,0) {};
\node [circle,draw,scale=.6] at (.25,1) {$1$};
\node [circle,draw,scale=0.55] at (.25,0) {$0$};
\node [label=left: {$\mathbf{j}$}] at (-0.5,0.5) {};

\draw [-]
	(0,1) arc (90:180:.5);
\filldraw[black] (-.5,.5) circle (2pt);
\draw [-]
	(-.5,0.5) arc (180:270:.5);
\end{tikzpicture}
%\end{minipage}
\right)
\right. = \left(C_j(x) \cdot \frac{(1-q)x^2}{1-qx^2}+ B_j(x) \cdot \frac{1-x^2}{1-qx^2}\right)/ C_j(\bar{x}),
\end{equation}
where the Boltzmann weights $\frac{(1-q)x^2}{1-qx^2}$ and $\frac{1-x^2}{1-qx^2}$ come from the twisted weights in Figure \ref{fig:R.weights}. Similarly, in the case $(\alpha,\beta) = (0,1)$, the ratio is
\begin{equation}\label{eq:ratio2}
\left(C_j(x) \cdot \frac{q(1-x^2)}{1-qx^2} + B_j(x) \cdot \frac{1-q}{1-qx^2}\right) / B_j(\bar{x}).
\end{equation}
Equating these two expressions and clearing denominators, we require 
\begin{equation} \label{eq:cleanedfish} \left[ C_j(x)(1-q)x^2 + B_j(x)(1-x^2) \right] B_j(\bar{x}) = C_j(\bar{x}) \left[ q (1-x^2) C_j(x) + B_j(x)(1-q) \right]. \end{equation}
Since we have assumed that $B_j(x)$ and $C_j(x)$ are monomial in $x$, then a simple degree comparison of the two sides shows that the possibilities above (equal degree, or degree differing by 2 as stated) are the only possibilities.

Suppose first that $\textrm{deg}(C_j(x,\bar{x})) = \textrm{deg}(B_j(x,\bar{x}))$. Since our required identity is a homogeneous quadratic in $B_j$ and $C_j$ we may scale and assume, without loss of generality, that the common degree is 0. Equating the coefficients of $x^2$ and equating the constant terms on either side of~(\ref{eq:cleanedfish}), both turn out to produce the same condition:
\[ ( C_j(x) q + B_j(x) ) ( C_j(x) - B_j(x)) = 0. \]
Plugging in any solution to the above to~(\ref{eq:ratio1}) or~(\ref{eq:ratio2}) gives the respective ratios $F$ in Cases (1a) and (1b) of the Lemma.

If instead $\textrm{deg}(C_j(x,\bar{x})) = \textrm{deg}(B_j(x,\bar{x}))-2$ then we may again scale the monomial weights so that $C_j(x) = c \bar{x}$ for some $c \in \mathbb{C}[q]$ and $B_j(x)= b x$ for some $b \in \mathbb{C}[q]$. Substituting these expressions into~(\ref{eq:cleanedfish}) we again find that both the $x^2$ coefficient and the constant term produce the same required identity on $b,c$:
\[ (b + c) (b - qc) = 0. \]
If $b=-c$, then substituting into (\ref{eq:ratio2}) and simplifying, the resulting constant of proportionality $F=(x^2-q)/(1-qx^2)$. The case of $b = qc$ gives $F=1$.

Observe that if $B_j=0$, then \eqref{eq:ratio2} implies that $C_j$ must also be 0 for the quantities in \eqref{eq:fish} to be proportional. Similarly, if $C_j=0$, then $B_j$ must be 0 in order for the ratio of terms $F$ in \eqref{eq:fish} to be a constant independent of $\alpha$ and $\beta$.

Lastly, we consider the cases where $(\alpha,\beta) = (1,1)$ or $(0,0)$. Only bends of type $A_j$ and $D_j$ arise here. If $(\alpha,\beta) = (1,1)$, then $F \cdot D_j(\bar{x})= (D_j(x) \cdot 1)$ and one may easily see the conditions on $D(x)$ in the statement of the Lemma follow for the two cases for $F$. In particular, if $F=(x^2-q)/(1-qx^2)$, then no such monomial $D(x)$ can satisfy $F = D_j(x) / D_j(\bar{x})$ which forces $D(x)=0.$ Similar conditions hold for $A_j(x)$ following from the $(0,0)$ case. \end{proof}

The Fish relation is needed to prove that $\mathcal{Z}(\mathcal{B}_{\lambda})$ is symmetric under $x_i\leftrightarrow \overline{x}_i$. In order to prove that $\mathcal{Z}(\mathcal{B}_{\lambda})$ is symmetric under the remaining relations $x_i \leftrightarrow x_{i+1}$, however, we will need to use an additional relation that we call the Caduceus relation (to be described below).  That relation involves interactions \emph{between} bends, and for this reason one is forced to consider how the fish relations for various bends are related to each other (or not).  For this reason, we introduce the following

\begin{definition}
A weighting scheme for which each bend satisfies the same subcase of the Fish relation (Lemma \ref{le:fish}) is said to be of uniform regime.  If different bends satisfy different subcases of the Fish relation, we say the model is mixed regime.
\end{definition}

It is sometimes convenient to be explicit about which condition of the Fish relation is satisfied for a model of uniform regime, in which case we will say that the model is uniform in regime $R$ (where $R$ is some condition from $\{1(a),1(b),2(a),2(b)\}$ as enumerated in Lemma~\ref{le:fish}). 

The lure of mixed regimes is that they offer more flexibility; unfortunately, this also makes them more complex to study.  For the duration of this paper, we will focus on analyzing the behavior of uniform regime models.

The Caduceus relation gives algebraic conditions that are described via a case analysis of the number of particles allowed on the boundary (always between 0 and 4, as the caduceus has 4 boundary edges). We will treat the cases where there are one, two, or three particles allowed on the boundary in separate lemmas below. Note that our diagrams show only configurations in rank two, but the methods apply generally to any pair of adjacent bends in a lattice of arbitrary rank.

%\item  If $B_1 = -qC_1$ and $B_2 = C_2$, or if $B_1 = C_1$ and $B_2 = -qC_2$, then the terms from \eqref{eq:caduceus} are proportional with constant $F\in \mathbb{Z}[q,q^{-1}](x_1^{\pm 1},\cdots,x_r^{\pm 1})$ independent of the choice of labels $\alpha, \beta, \gamma, \delta$ if and only if $B_1B_2=0$ and $q^2A_1D_2=A_2D_1$. {\color{red} (Insert constant of proportionality eventually)}

%\begin{remark*}
%In what follows, we assume that if $\mathcal{Z}(\mathcal{B}_{\lambda})$ satisfies the Fish relation in all bends, then, in particular, each bend of $\mathcal{Z}(\mathcal{B}_{\lambda})$ satisfies the same case of Lemma 3.1.
%\end{remark*}

\begin{lemma}[Caduceus relation, 2 particles]\label{le:caduceus}
Let the $R$-matrix Boltzmann weights be given as in Figure 3, and assume that the bend weights are monomial with coefficients in $\mathbb{C}[q]$ of uniform regime.  Then the quantities
\begin{equation}\label{eq:caduceus}
\mathcal{Z}\left(
\begin{tikzpicture}[scale=.5,baseline={([yshift=-\the\dimexpr\fontdimen22\textfont2\relax]current bounding box.center)}]
\node [label=right:$x_1$] at (4,2) {};
\node [label=right:$\overline{x}_1$] at (4,1) {};
\node [label=right:$x_2$] at (4,0) {};
\node [label=right:$\overline{x}_2$] at (4,-1) {};
\node [label=left:$\mathbf{1}$] at (-.5,-.5) {};
\node [label=left:$\mathbf{2}$] at (-.5,1.5) {};

\node [circle,draw,scale=.6] at (3.5,2) {$\alpha$};
\node [circle,draw,scale=0.55] at (3.5,1) {$\beta$};
\node [circle,draw,scale=.6] at (3.5,0) {$\gamma$};
\node [circle,draw,scale=0.6] at (3.5,-1) {$\delta$};

\draw [-]
	(0,2) arc (90:270:.5);
\draw [-]
	(0,0) arc (90:270:.5);

\draw [-] (0,2) -- (1,2);
\draw [-] (2,2) -- (3,2);
\draw [-] (0,-1) -- (1,-1);
\draw [-] (2,-1) -- (3,-1);

\draw [-] (1,0) .. controls (.5,0) and (.5,1) .. (0,1);
\draw [-] (1,1) .. controls (.5,1) and (.5,0) .. (0,0);
\draw [-] (2,-1) .. controls (1.5,-1) and (1.5,0) .. (1,0);
\draw [-] (2,0) .. controls (1.5,0) and (1.5,-1) .. (1,-1);
\draw [-] (2,1) .. controls (1.5,1) and (1.5,2) .. (1,2);
\draw [-] (2,2) .. controls (1.5,2) and (1.5,1) .. (1,1);
\draw [-] (3,0) .. controls (2.5,0) and (2.5,1) .. (2,1);
\draw [-] (3,1) .. controls (2.5,1) and (2.5,0) .. (2,0);
\filldraw[black] (-.5,1.5) circle (4pt);
\filldraw[black] (-.5,-.5) circle (4pt);

\end{tikzpicture}\right)
\quad \text{ and } \quad  \emph{wt}\left(
\begin{tikzpicture}[scale=0.4,baseline={([yshift=-\the\dimexpr\fontdimen22\textfont2\relax]current bounding box.center)}]
\node [label=above:$ $] at (0,1) {};
\node [label=below:$ $] at (0,-1) {};
\node [label=left:$\mathbf{2}$] at (-1,0) {};

\node [circle,draw,scale=.6] at (.5,1) {$\alpha$};
\node [circle,draw,scale=0.55] at (.5,-1) {$\beta$};
\node [label=center:$x_1$] at (2,1) {};
\node [label=center:$\overline{x}_1$] at (2,-1) {};
\draw [-]
	(0,1) arc (90:180:1);
\filldraw[black] (-1,0) circle (4pt);
\draw [-]
	(-1,0) arc (180:270:1);
\draw [-]
	(0,1) arc (90:180:1);
\end{tikzpicture}\right)\cdot \emph{wt}\left(
\begin{tikzpicture}[scale=0.4,baseline={([yshift=-\the\dimexpr\fontdimen22\textfont2\relax]current bounding box.center)}]
\node [label=above:$ $] at (0,1) {};
\node [label=below:$ $] at (0,-1) {};
\node [label=left:$\mathbf{1}$] at (-1,0) {};

\node [circle,draw,scale=.6] at (.5,1) {$\gamma$};
\node [circle,draw,scale=0.6] at (.5,-1) {$\delta$};
\node [label=center:$x_2$] at (2,1) {};
\node [label=center:$\overline{x}_2$] at (2,-1) {};
\draw [-]
	(0,1) arc (90:180:1);
\filldraw[black] (-1,0) circle (4pt);
\draw [-]
	(-1,0) arc (180:270:1);
\draw [-]
	(-1,0) arc (180:270:1);
\end{tikzpicture}\right)
\end{equation}
are proportional with proportionality constant $F \in \mathbb{Z}[q,q^{-1}](x_1^{\pm 1}, x_2^{\pm 2})$ independent of the choice of boundary labels $\alpha,\beta,\gamma, \delta$ in $\{0,0,1,1\}$ if and only if one of the following regime-dependent conditions hold:
\begin{enumerate}
\item when the weights are regime 1(a), we have $C_j \in \mathbb{C}[q]$ for $j \in \{1,2\}$, and \begin{equation}\label{eq:cad.condition}q^2(A_1D_2-C_1 C_2) - (A_2D_1-C_1 C_2) = 0,\end{equation} in which case $F=1$;
\item when the weights are regime 2(b), we have $B_j(x) = qm_jx$, $C_j(x) = m_j\bar{x}$, where $m_j \in \mathbb{C}[q]$ for $j \in \{1,2\}$, and \begin{equation}\label{eq:cad.condition2} q^3m_2m_1 + q^2D_2A_1 - qm_2m_1 - A_2D_1 = 0, \end{equation}
in which case $F=1$;
\item when the weights are regime 1(b) the Caduceus relation automatically holds, in which case $F=(C_1(x_1)C_2(x_2))/(C_1(x_2)C_2(x_1))$; and
\item when the weights are regime 2(a), the Caduceus relation automatically holds, in which case $F=(C_1(x_1)C_2(x_2))/(C_1(x_2)C_2(x_1))$.
\end{enumerate}
\end{lemma}

\begin{proof}

As an example, consider the case $(\alpha,\beta,\gamma,\delta) = (1,0,1,0)$. Note that there are eight different states with these boundary conditions, depicted below with Boltzmann weights from Figures \ref{fig:R.weights} and \ref{fig:gen.bends}:\\

\begingroup
\renewcommand*{\arraystretch}{1.3}
\begin{tabular}{|c|c|c|}%c|

\hline

%%% PICTURE 1 %%%%
$\textrm{wt}\left(
\begin{tikzpicture}[scale=.5,baseline={([yshift=-\the\dimexpr\fontdimen22\textfont2\relax]current bounding box.center)}]
\node [label=right:$x_1$] at (3,2) {};
\node [label=right:$\overline{x}_1$] at (3,1) {};
\node [label=right:$x_2$] at (3,0) {};
\node [label=right:$\overline{x}_2$] at (3,-1) {};
\node [label=left:$\mathbf{1}$] at (-.5,-.5) {};
\node [label=left:$\mathbf{2}$] at (-.5,1.5) {};

\draw [-] 
    (3,-1) -- (2,-1) 
    .. controls (1.5,-1) and (1.5,0) .. (1,0) 
    .. controls (.5,0) and (.5,1) .. (0,1)
    arc (270:90:.5)
    -- (1,2)
    .. controls (1.5,2) and (1.5,1) .. (2,1)
    .. controls (2.5,1) and (2.5,0) .. (3,0);

\draw [-] 
    (3,2) -- (2,2)
    .. controls (1.5,2) and (1.5,1) .. (1,1)
    .. controls (.5,1) and (.5,0) .. (0,0)
    arc (90:270:.5)
    -- (1,-1)
    .. controls (1.5,-1) and (1.5,0) .. (2,0)
    .. controls (2.5,0) and (2.5,1) .. (3,1);
    
\draw [-,line width=1mm,red,opacity=.5,rounded corners=0pt]
    (3,2) -- (2,2)
    .. controls (1.5,2) and (1.5,1) .. (1,1)
    .. controls (.6,1) and (.6,.5) .. (0.5,0.5)
    .. controls (.4,.5) and (.4,1) .. (0,1)
    arc (270:90:.5)
    -- (1,2)
    .. controls (1.5,2) and (1.5,1) .. (2,1)
    .. controls (2.5,1) and (2.5,0) .. (3,0);

%\draw [-,line width=1mm,red,opacity=.5]
%	(0,2) arc (90:180:.5);
%\draw [-,line width=1mm,red,opacity=.5]
%    (0,2) -- (1,2);
%\draw [-,line width=1mm,red,opacity=.5] (1.5,1.5) .. controls (1.5,1.5) and (1.5,2) .. (2,2);
%\draw [-,line width=1mm,red,opacity=.5] (1,2) .. controls (1.5,2) and (1.5,1.5) .. (1.5,1.5);
%\draw [-,line width=1mm,red,opacity=.5] (2,2) -- (3,2);

%\draw [-,line width=1mm,red,opacity=.5]
%	(-.5,1.5) arc (180:270:.5);
%\draw [-,line width=1mm,red,opacity=.5] (0,1) .. controls (.5,1) and (.5,.5) .. (.5,.5);
%\draw [-,line width=1mm,red,opacity=.5] (.5,.5) .. controls (.5,1) and (1,1) .. (1,1);
%\draw [-,line width=1mm,red,opacity=.5] (1,1) .. controls (1,1) and (1.25,1) .. (1.5,1.5);
%\draw [-,line width=1mm,red,opacity=.5] (1.5,1.5) .. controls (1.5,1) and (2,1) .. (2,1);
%\draw [-,line width=1mm,red,opacity=.5] (2,1) .. controls (2.5,1) and (2.5,.5) .. (2.5,.5);
%\draw [-,line width=1mm,red,opacity=.5] (2.5,.5) .. controls (2.5,0) and (3,0) .. (3,0);
\filldraw[black] (-.5,1.5) circle (4pt);
\filldraw[black] (-.5,-.5) circle (4pt);
\end{tikzpicture}\right)$

&

%%% PICTURE 2 %%%%
$\textrm{wt}\left(
\begin{tikzpicture}[scale=.5,baseline={([yshift=-\the\dimexpr\fontdimen22\textfont2\relax]current bounding box.center)}]
\node [label=right:$x_1$] at (3,2) {};
\node [label=right:$\overline{x}_1$] at (3,1) {};
\node [label=right:$x_2$] at (3,0) {};
\node [label=right:$\overline{x}_2$] at (3,-1) {};
\node [label=left:$\mathbf{1}$] at (-.5,-.5) {};
\node [label=left:$\mathbf{2}$] at (-.5,1.5) {};

\draw [-] 
    (3,-1) -- (2,-1) 
    .. controls (1.5,-1) and (1.5,0) .. (1,0) 
    .. controls (.5,0) and (.5,1) .. (0,1)
    arc (270:90:.5)
    -- (1,2)
    .. controls (1.5,2) and (1.5,1) .. (2,1)
    .. controls (2.5,1) and (2.5,0) .. (3,0);

\draw [-] 
    (3,2) -- (2,2)
    .. controls (1.5,2) and (1.5,1) .. (1,1)
    .. controls (.5,1) and (.5,0) .. (0,0)
    arc (90:270:.5)
    -- (1,-1)
    .. controls (1.5,-1) and (1.5,0) .. (2,0)
    .. controls (2.5,0) and (2.5,1) .. (3,1);

\draw [-,line width=1mm,red,opacity=.5,rounded corners=0pt]
    (3,2) -- (2,2)
    .. controls (1.6,2) and (1.6,1.5) .. (1.5,1.5)
    .. controls (1.4,1.5) and (1.4,2) .. (1,2)
    -- (0,2)
    arc (90:270:.5)
    .. controls (.5,1) and (.5,0) .. (1,0)
    .. controls (1.4,0) and (1.4,-.5) .. (1.5,-.5)
    .. controls (1.6,-.5) and (1.6,0) .. (2,0)
    .. controls (2.4,0) and (2.4,0.5) .. (2.5,0.5)
    .. controls (2.6,0.5) and (2.6,0) .. (3,0);
    
%\draw [-,line width=1mm,red,opacity=.5]
%	(0,2) arc (90:180:.5);
%\draw [-,line width=1mm,red,opacity=.5]
%    (0,2) -- (1,2);
%\draw [-,line width=1mm,red,opacity=.5] (1.5,1.5) .. controls (1.5,1.5) and (1.5,2) .. (2,2);
%\draw [-,line width=1mm,red,opacity=.5] (1,2) .. controls (1.5,2) and (1.5,1.5) .. (1.5,1.5);
%\draw [-,line width=1mm,red,opacity=.5] (2,2) -- (3,2);

%\draw [-,line width=1mm,red,opacity=.5]
%	(-.5,1.5) arc (180:270:.5);
%\draw [-,line width=1mm,red,opacity=.5] (0,1) .. controls (.5,1) and (.5,.5) .. (.5,.5);

%\draw [-,line width=1mm,red,opacity=.5] (.5,.5) .. controls (.5,0) and (1,0) .. (1,0);
%\draw [-,line width=1mm,red,opacity=.5] (1,0) .. controls (1,0) and (1.5,0) .. (1.5,-.5);
%\draw [-,line width=1mm,red,opacity=.5] (1.5,-.5) .. controls (1.5,0) and (2,0) .. (2,0);
%\draw [-,line width=1mm,red,opacity=.5] (2,0) .. controls (2.5,0) and (2.5,.5) .. (2.5,.5);
%\draw [-,line width=1mm,red,opacity=.5] (2.5,.5) .. controls (2.5,0) and (3,0) .. (3,0);

\filldraw[black] (-.5,1.5) circle (4pt);
\filldraw[black] (-.5,-.5) circle (4pt);

\end{tikzpicture}\right)$

&

%%%% PICTURE 3 %%%%
$\textrm{wt}\left(
\begin{tikzpicture}[scale=.5,baseline={([yshift=-\the\dimexpr\fontdimen22\textfont2\relax]current bounding box.center)}]
\node [label=right:$x_1$] at (3,2) {};
\node [label=right:$\overline{x}_1$] at (3,1) {};
\node [label=right:$x_2$] at (3,0) {};
\node [label=right:$\overline{x}_2$] at (3,-1) {};
\node [label=left:$\mathbf{1}$] at (-.5,-.5) {};
\node [label=left:$\mathbf{2}$] at (-.5,1.5) {};

\draw [-] 
    (3,-1) -- (2,-1) 
    .. controls (1.5,-1) and (1.5,0) .. (1,0) 
    .. controls (.5,0) and (.5,1) .. (0,1)
    arc (270:90:.5)
    -- (1,2)
    .. controls (1.5,2) and (1.5,1) .. (2,1)
    .. controls (2.5,1) and (2.5,0) .. (3,0);

\draw [-] 
    (3,2) -- (2,2)
    .. controls (1.5,2) and (1.5,1) .. (1,1)
    .. controls (.5,1) and (.5,0) .. (0,0)
    arc (90:270:.5)
    -- (1,-1)
    .. controls (1.5,-1) and (1.5,0) .. (2,0)
    .. controls (2.5,0) and (2.5,1) .. (3,1);

\draw [-,line width=1mm,red,opacity=.5,rounded corners=0pt]
    (3,2) -- (2,2)
    .. controls (1.5,2) and (1.5,1) .. (1,1)
    .. controls (0.5,1) and (0.5,0) .. (0,0)
    arc (90:180:.5);
    
\draw [-,line width=1mm,red,opacity=.5,rounded corners=0pt]
    (3,0) .. controls (2.5,0) and (2.5,1) .. (2,1)
    .. controls (1.5,1) and (1.5,2) .. (1,2)
    -- (0,2)
    arc (90:180:.5);

%\draw [-,line width=1mm,red,opacity=.5]
%	(0,2) arc (90:180:.5);
%\draw [-,line width=1mm,red,opacity=.5]
%    (0,2) -- (1,2);
%\draw [-,line width=1mm,red,opacity=.5] (1.5,1.5) .. controls (1.5,1.5) and (1.5,2) .. (2,2);
%\draw [-,line width=1mm,red,opacity=.5] (1,2) .. controls (1.5,2) and (1.5,1.5) .. (1.5,1.5);
%\draw [-,line width=1mm,red,opacity=.5] (2,2) -- (3,2);

%\draw [-,line width=1mm,red,opacity=.5]
%	(0,0) arc (90:180:.5);
%\draw [-,line width=1mm,red,opacity=.5] (0,0) .. controls (.5,0) and (.5,.5) .. (.5,.5);
%\draw [-,line width=1mm,red,opacity=.5] (.5,.5) .. controls (.5,1) and (1,1) .. (1,1);
%\draw [-,line width=1mm,red,opacity=.5] (1,1) .. controls (1,1) and (1.25,1) .. (1.5,1.5);
%\draw [-,line width=1mm,red,opacity=.5] (1.5,1.5) .. controls (1.5,1) and (2,1) .. (2,1);
%\draw [-,line width=1mm,red,opacity=.5] (2,1) .. controls (2.5,1) and (2.5,.5) .. (2.5,.5);
%\draw [-,line width=1mm,red,opacity=.5] (2.5,.5) .. controls (2.5,0) and (3,0) .. (3,0);

\filldraw[black] (-.5,1.5) circle (4pt);
\filldraw[black] (-.5,-.5) circle (4pt);
\end{tikzpicture}\right)$

\\ \hline

{\footnotesize  $A_1(x_1)D_2(x_2)a_2(1,2)c_2(1,\overline{2})b_2(\overline{1},2)a_1(\overline{1},\overline{2})$} & {\footnotesize $A_1(x_1)D_2(x_2)c_2(1,2)b_2(1,\overline{2})c_1(\overline{1},2)c_2(\overline{1},\overline{2})$} & 
{\footnotesize  $C_1(x_1) C_2(x_2) a_2(1,2)b_1(1,\overline{2})b_2(\overline{1},2)a_1(\overline{1},\overline{2})$}

\\ \hline

%%%% PICTURE 4 %%%%
$\textrm{wt}\left(
\begin{tikzpicture}[scale=.5,baseline={([yshift=-\the\dimexpr\fontdimen22\textfont2\relax]current bounding box.center)}]
\node [label=right:$x_1$] at (3,2) {};
\node [label=right:$\overline{x}_1$] at (3,1) {};
\node [label=right:$x_2$] at (3,0) {};
\node [label=right:$\overline{x}_2$] at (3,-1) {};
\node [label=left:$\mathbf{1}$] at (-.5,-.5) {};
\node [label=left:$\mathbf{2}$] at (-.5,1.5) {};

\draw [-] 
    (3,-1) -- (2,-1) 
    .. controls (1.5,-1) and (1.5,0) .. (1,0) 
    .. controls (.5,0) and (.5,1) .. (0,1)
    arc (270:90:.5)
    -- (1,2)
    .. controls (1.5,2) and (1.5,1) .. (2,1)
    .. controls (2.5,1) and (2.5,0) .. (3,0);

\draw [-] 
    (3,2) -- (2,2)
    .. controls (1.5,2) and (1.5,1) .. (1,1)
    .. controls (.5,1) and (.5,0) .. (0,0)
    arc (90:270:.5)
    -- (1,-1)
    .. controls (1.5,-1) and (1.5,0) .. (2,0)
    .. controls (2.5,0) and (2.5,1) .. (3,1);

\draw [-,line width=1mm,red,opacity=.5,rounded corners=0pt]
    (3,2) -- (2,2)
    .. controls (1.6,2) and (1.6,1.5) .. (1.5,1.5)
    .. controls (1.4,1.5) and (1.4,2) .. (1,2)
    -- (0,2)
    arc (90:180:.5);
    
\draw [-,line width=1mm,red,opacity=.5,rounded corners=0pt]
    (3,0) .. controls (2.6,0) and (2.6,0.5) .. (2.5,0.5)
    .. controls (2.4,0.5) and (2.4,0) .. (2,0)
    .. controls (1.6,0) and (1.6,-.5) .. (1.5,-.5)
    .. controls (1.4,-.5) and (1.4,0) .. (1,0)
    .. controls (0.6,0) and (0.6,0.5) .. (0.5,0.5)
    .. controls (0.4,0.5) and (0.4,0) .. (0,0)
    arc (90:180:.5);
    
%\draw [-,line width=1mm,red,opacity=.5]
%	(0,2) arc (90:180:.5);
%\draw [-,line width=1mm,red,opacity=.5]
%    (0,2) -- (1,2);
%\draw [-,line width=1mm,red,opacity=.5] (1.5,1.5) .. controls (1.5,1.5) and (1.5,2) .. (2,2);
%\draw [-,line width=1mm,red,opacity=.5] (1,2) .. controls (1.5,2) and (1.5,1.5) .. (1.5,1.5);
%\draw [-,line width=1mm,red,opacity=.5] (2,2) -- (3,2);

%\draw [-,line width=1mm,red,opacity=.5]
%	(0,0) arc (90:180:.5);
%\draw [-,line width=1mm,red,opacity=.5] (.5,.5) .. controls (.5,0) and (1,0) .. (1,0);
%\draw [-,line width=1mm,red,opacity=.5] (0,0) .. controls (.5,0) and (.5,.5) .. (.5,.5);
%\draw [-,line width=1mm,red,opacity=.5] (1,0) .. controls (1,0) and (1.5,0) .. (1.5,-.5);
%\draw [-,line width=1mm,red,opacity=.5] (1.5,-.5) .. controls (1.5,0) and (2,0) .. (2,0);
%\draw [-,line width=1mm,red,opacity=.5] (2.5,.5) .. controls (2.5,0) and (3,0) .. (3,0);
%\draw [-,line width=1mm,red,opacity=.5] (2,0) .. controls (2.5,0) and (2.5,.5) .. (2.5,.5);

\filldraw[black] (-.5,1.5) circle (4pt);
\filldraw[black] (-.5,-.5) circle (4pt);
\end{tikzpicture}\right)$

&

%%%% PICTURE 5 %%%%
$\textrm{wt}\left(
\begin{tikzpicture}[scale=.5,baseline={([yshift=-\the\dimexpr\fontdimen22\textfont2\relax]current bounding box.center)}]
\node [label=right:$x_1$] at (3,2) {};
\node [label=right:$\overline{x}_1$] at (3,1) {};
\node [label=right:$x_2$] at (3,0) {};
\node [label=right:$\overline{x}_2$] at (3,-1) {};
\node [label=left:$\mathbf{1}$] at (-.5,-.5) {};
\node [label=left:$\mathbf{2}$] at (-.5,1.5) {};

\draw [-] 
    (3,-1) -- (2,-1) 
    .. controls (1.5,-1) and (1.5,0) .. (1,0) 
    .. controls (.5,0) and (.5,1) .. (0,1)
    arc (270:90:.5)
    -- (1,2)
    .. controls (1.5,2) and (1.5,1) .. (2,1)
    .. controls (2.5,1) and (2.5,0) .. (3,0);

\draw [-] 
    (3,2) -- (2,2)
    .. controls (1.5,2) and (1.5,1) .. (1,1)
    .. controls (.5,1) and (.5,0) .. (0,0)
    arc (90:270:.5)
    -- (1,-1)
    .. controls (1.5,-1) and (1.5,0) .. (2,0)
    .. controls (2.5,0) and (2.5,1) .. (3,1);

\draw [-,line width=1mm,red,opacity=.5,rounded corners=0pt]
    (3,2) -- (2,2)
    .. controls (1.6,2) and (1.6,1.5) .. (1.5,1.5)
    .. controls (1.4,1.5) and (1.4,2) .. (1,2)
    -- (0,2)
    arc (90:180:.5);

\draw [-,line width=1mm,red,opacity=.5,rounded corners=0pt]
    (3,0) .. controls (2.6,0) and (2.6,0.5) .. (2.5,0.5)
    .. controls (2.4,0.5) and (2.4,0) .. (2,0)
    .. controls (1.5,0) and (1.5,-1) .. (1,-1)
    -- (0,-1)
    arc (270:180:.5);

%\draw [-,line width=1mm,red,opacity=.5]
%	(0,2) arc (90:180:.5);
%\draw [-,line width=1mm,red,opacity=.5]
%    (0,2) -- (1,2);
%\draw [-,line width=1mm,red,opacity=.5] (1.5,1.5) .. controls (1.5,1.5) and (1.5,2) .. (2,2);
%\draw [-,line width=1mm,red,opacity=.5] (1,2) .. controls (1.5,2) and (1.5,1.5) .. (1.5,1.5);
%\draw [-,line width=1mm,red,opacity=.5] (2,2) -- (3,2);

%\draw [-,line width=1mm,red,opacity=.5]
%	(-.5,-.5) arc (180:270:.5);
%\draw [-,line width=1mm,red,opacity=.5] (0,-1) -- (1,-1);
%\draw [-,line width=1mm,red,opacity=.5] (1,-1) .. controls (1.5,-1) and (1.5,-.5) .. (1.5,-.5);
%\draw [-,line width=1mm,red,opacity=.5] (1.5,-.5) .. controls (1.5,0) and (2,0) .. (2,0);
%\draw [-,line width=1mm,red,opacity=.5] (2.5,.5) .. controls (2.5,0) and (3,0) .. (3,0);
%\draw [-,line width=1mm,red,opacity=.5] (2,0) .. controls (2.5,0) and (2.5,.5) .. (2.5,.5);

\filldraw[black] (-.5,1.5) circle (4pt);
\filldraw[black] (-.5,-.5) circle (4pt);
\end{tikzpicture}\right)$

&

%%% PICTURE 6 %%%%
$\textrm{wt}\left(
\begin{tikzpicture}[scale=.5,baseline={([yshift=-\the\dimexpr\fontdimen22\textfont2\relax]current bounding box.center)}]
\node [label=right:$x_1$] at (3,2) {};
\node [label=right:$\overline{x}_1$] at (3,1) {};
\node [label=right:$x_2$] at (3,0) {};
\node [label=right:$\overline{x}_2$] at (3,-1) {};
\node [label=left:$\mathbf{1}$] at (-.5,-.5) {};
\node [label=left:$\mathbf{2}$] at (-.5,1.5) {};

\draw [-] 
    (3,-1) -- (2,-1) 
    .. controls (1.5,-1) and (1.5,0) .. (1,0) 
    .. controls (.5,0) and (.5,1) .. (0,1)
    arc (270:90:.5)
    -- (1,2)
    .. controls (1.5,2) and (1.5,1) .. (2,1)
    .. controls (2.5,1) and (2.5,0) .. (3,0);

\draw [-] 
    (3,2) -- (2,2)
    .. controls (1.5,2) and (1.5,1) .. (1,1)
    .. controls (.5,1) and (.5,0) .. (0,0)
    arc (90:270:.5)
    -- (1,-1)
    .. controls (1.5,-1) and (1.5,0) .. (2,0)
    .. controls (2.5,0) and (2.5,1) .. (3,1);

\draw [-,line width=1mm,red,opacity=.5,rounded corners=0pt]
    (3,2) -- (2,2)
    .. controls (1.5,2) and (1.5,1) .. (1,1)
    .. controls (.5,1) and (.5,0) .. (0,0)
    arc (90:180:.5);

\draw [-,line width=1mm,red,opacity=.5,rounded corners=0pt]
    (3,0) .. controls (2.6,0) and (2.6,0.5) .. (2.5,0.5)
    .. controls (2.4,0.5) and (2.4,0) .. (2,0)
    .. controls (1.6,0) and (1.6,-.5) .. (1.5,-.5)
    .. controls (1.4,-.5) and (1.4,0) .. (1,0)
    .. controls (.5,0) and (.5,1) .. (0,1)
    arc (270:180:.5);

%\draw [-,line width=1mm,red,opacity=.5]
%	(-.5,1.5) arc (180:270:.5);
%\draw [-,line width=1mm,red,opacity=.5] (1.5,1.5) .. controls (1.5,1.5) and (1.5,2) .. (2,2);
%\draw [-,line width=1mm,red,opacity=.5] (2,2) -- (3,2);
%\draw [-,line width=1mm,red,opacity=.5] (0,1) .. controls (.5,1) and (.5,.5) .. (.5,.5);
%\draw [-,line width=1mm,red,opacity=.5] (.5,.5) .. controls (.5,1) and (1,1) .. (1,1);
%\draw [-,line width=1mm,red,opacity=.5] (1,1) .. controls (1,1) and (1.25,1) .. (1.5,1.5);

%\draw [-,line width=1mm,red,opacity=.5]
%	(0,0) arc (90:180:.5);
%\draw [-,line width=1mm,red,opacity=.5] (.5,.5) .. controls (.5,0) and (1,0) .. (1,0);
%\draw [-,line width=1mm,red,opacity=.5] (0,0) .. controls (.5,0) and (.5,.5) .. (.5,.5);
%\draw [-,line width=1mm,red,opacity=.5] (1,0) .. controls (1,0) and (1.5,0) .. (1.5,-.5);
%\draw [-,line width=1mm,red,opacity=.5] (1.5,-.5) .. controls (1.5,0) and (2,0) .. (2,0);
%\draw [-,line width=1mm,red,opacity=.5] (2.5,.5) .. controls (2.5,0) and (3,0) .. (3,0);
%\draw [-,line width=1mm,red,opacity=.5] (2,0) .. controls (2.5,0) and (2.5,.5) .. (2.5,.5);

\filldraw[black] (-.5,1.5) circle (4pt);
\filldraw[black] (-.5,-.5) circle (4pt);
\end{tikzpicture}\right)$

\\ \hline

{\footnotesize $C_1(x_1) C_2(x_2) c_2(1,2)c_1(1,\overline{2})c_1(\overline{1},2)c_2(\overline{1},\overline{2})$} & {\footnotesize $B_1(x_1) C_2(x_2) c_2(1,2)a_1(1,\overline{2})c_1(\overline{1},2)b_1(\overline{1},\overline{2})$ }& {\footnotesize $C_1(x_1) B_2(x_2) b_1(1,2)a_2(1,\overline{2})c_1(\overline{1},2)c_2(\overline{1},\overline{2})$}

\\ \hline

%%% PICTURE 7 %%%%
$\textrm{wt}\left(
\begin{tikzpicture}[scale=.5,baseline={([yshift=-\the\dimexpr\fontdimen22\textfont2\relax]current bounding box.center)}]
\node [label=right:$x_1$] at (3,2) {};
\node [label=right:$\overline{x}_1$] at (3,1) {};
\node [label=right:$x_2$] at (3,0) {};
\node [label=right:$\overline{x}_2$] at (3,-1) {};
\node [label=left:$\mathbf{1}$] at (-.5,-.5) {};
\node [label=left:$\mathbf{2}$] at (-.5,1.5) {};

\draw [-] 
    (3,-1) -- (2,-1) 
    .. controls (1.5,-1) and (1.5,0) .. (1,0) 
    .. controls (.5,0) and (.5,1) .. (0,1)
    arc (270:90:.5)
    -- (1,2)
    .. controls (1.5,2) and (1.5,1) .. (2,1)
    .. controls (2.5,1) and (2.5,0) .. (3,0);

\draw [-] 
    (3,2) -- (2,2)
    .. controls (1.5,2) and (1.5,1) .. (1,1)
    .. controls (.5,1) and (.5,0) .. (0,0)
    arc (90:270:.5)
    -- (1,-1)
    .. controls (1.5,-1) and (1.5,0) .. (2,0)
    .. controls (2.5,0) and (2.5,1) .. (3,1);

\draw [-,line width=1mm,red,opacity=.5,rounded corners=0pt]
    (3,2) -- (2,2)
    .. controls (1.5,2) and (1.5,1) .. (1,1)
    .. controls (.6,1) and (.6,.5) .. (.5,.5)
    .. controls (.4,.5) and (.4,1) .. (0,1)
    arc (270:180:.5);

\draw [-,line width=1mm,red,opacity=.5,rounded corners=0pt]
    (3,0) .. controls (2.6,0) and (2.6,0.5) .. (2.5,0.5)
    .. controls (2.4,0.5) and (2.4,0) .. (2,0)
    .. controls (1.5,0) and (1.5,-1) .. (1,-1)
    -- (0,-1)
    arc (270:180:.5);

%\draw [-,line width=1mm,red,opacity=.5]
%	(-.5,1.5) arc (180:270:.5);
%\draw [-,line width=1mm,red,opacity=.5] (1.5,1.5) .. controls (1.5,1.5) and (1.5,2) .. (2,2);
%\draw [-,line width=1mm,red,opacity=.5] (2,2) -- (3,2);
%\draw [-,line width=1mm,red,opacity=.5] (0,1) .. controls (.5,1) and (.5,.5) .. (.5,.5);
%\draw [-,line width=1mm,red,opacity=.5] (.5,.5) .. controls (.5,1) and (1,1) .. (1,1);
%\draw [-,line width=1mm,red,opacity=.5] (1,1) .. controls (1,1) and (1.25,1) .. (1.5,1.5);

%\draw [-,line width=1mm,red,opacity=.5]
%	(-.5,-.5) arc (180:270:.5);
%\draw [-,line width=1mm,red,opacity=.5] (0,-1) -- (1,-1);
%\draw [-,line width=1mm,red,opacity=.5] (1,-1) .. controls (1.5,-1) and (1.5,-.5) .. (1.5,-.5);
%\draw [-,line width=1mm,red,opacity=.5] (1.5,-.5) .. controls (1.5,0) and (2,0) .. (2,0);
%\draw [-,line width=1mm,red,opacity=.5] (2.5,.5) .. controls (2.5,0) and (3,0) .. (3,0);
%\draw [-,line width=1mm,red,opacity=.5] (2,0) .. controls (2.5,0) and (2.5,.5) .. (2.5,.5);

\filldraw[black] (-.5,1.5) circle (4pt);
\filldraw[black] (-.5,-.5) circle (4pt);
\end{tikzpicture}\right)$

&

%%% PICTURE 8 %%%%
$\textrm{wt}\left(
\begin{tikzpicture}[scale=.5,baseline={([yshift=-\the\dimexpr\fontdimen22\textfont2\relax]current bounding box.center)}]
\node [label=right:$x_1$] at (3,2) {};
\node [label=right:$\overline{x}_1$] at (3,1) {};
\node [label=right:$x_2$] at (3,0) {};
\node [label=right:$\overline{x}_2$] at (3,-1) {};
\node [label=left:$\mathbf{1}$] at (-.5,-.5) {};
\node [label=left:$\mathbf{2}$] at (-.5,1.5) {};

\draw [-] 
    (3,-1) -- (2,-1) 
    .. controls (1.5,-1) and (1.5,0) .. (1,0) 
    .. controls (.5,0) and (.5,1) .. (0,1)
    arc (270:90:.5)
    -- (1,2)
    .. controls (1.5,2) and (1.5,1) .. (2,1)
    .. controls (2.5,1) and (2.5,0) .. (3,0);

\draw [-] 
    (3,2) -- (2,2)
    .. controls (1.5,2) and (1.5,1) .. (1,1)
    .. controls (.5,1) and (.5,0) .. (0,0)
    arc (90:270:.5)
    -- (1,-1)
    .. controls (1.5,-1) and (1.5,0) .. (2,0)
    .. controls (2.5,0) and (2.5,1) .. (3,1);

\draw [-,line width=1mm,red,opacity=.5,rounded corners=0pt]
    (3,0) .. controls (2.6,0) and (2.6,0.5) .. (2.5,0.5)
    .. controls (2.4,0.5) and (2.4,0) .. (2,0)
    .. controls (1.5,0) and (1.5,-1) .. (1,-1)
    -- (0,-1)
    arc (270:90:.5)
    .. controls (.5,0) and (.5,1) .. (1,1)
    .. controls (1.5,1) and (1.5,2) .. (2,2)
    -- (3,2);

%\draw [-,line width=1mm,red,opacity=.5]
%	(0,0) arc (90:180:.5);
%\draw [-,line width=1mm,red,opacity=.5] (0,0) .. controls (.5,0) and (.5,.5) .. (.5,.5);
%\draw [-,line width=1mm,red,opacity=.5] (1.5,1.5) .. controls (1.5,1.5) and (1.5,2) .. (2,2);
%\draw [-,line width=1mm,red,opacity=.5] (2,2) -- (3,2);
%\draw [-,line width=1mm,red,opacity=.5] (.5,.5) .. controls (.5,1) and (1,1) .. (1,1);
%\draw [-,line width=1mm,red,opacity=.5] (1,1) .. controls (1,1) and (1.25,1) .. (1.5,1.5);

%\draw [-,line width=1mm,red,opacity=.5]
%	(-.5,-.5) arc (180:270:.5);
%\draw [-,line width=1mm,red,opacity=.5] (0,-1) -- (1,-1);
%\draw [-,line width=1mm,red,opacity=.5] (1,-1) .. controls (1.5,-1) and (1.5,-.5) .. (1.5,-.5);
%\draw [-,line width=1mm,red,opacity=.5] (1.5,-.5) .. controls (1.5,0) and (2,0) .. (2,0);
%\draw [-,line width=1mm,red,opacity=.5] (2.5,.5) .. controls (2.5,0) and (3,0) .. (3,0);
%\draw [-,line width=1mm,red,opacity=.5] (2,0) .. controls (2.5,0) and (2.5,.5) .. (2.5,.5);

\filldraw[black] (-.5,1.5) circle (4pt);
\filldraw[black] (-.5,-.5) circle (4pt);

\end{tikzpicture}\right)$ &

\\ \hline

{\footnotesize $B_1(x_1) B_2(x_2) b_1(1,2)c_2(1,\overline{2})c_1(\overline{1},2)b_1(\overline{1},\overline{2})$} & {\footnotesize $D_1(x_1)A_2(x_2)b_1(1,2)b_1(1,\overline{2})c_1(\overline{1},2)b_1(\overline{1},\overline{2})$ }& 

\\ \hline

\end{tabular}
\endgroup

\vskip .1 in

Note that the partition function of the lattice on the left in \eqref{eq:caduceus} in this case is simply equal to the sum of the Boltzmann weights of the states depicted above. Similarly, in the case $(1,0,0,1)$, the associated partition function of the lattice on the left in \eqref{eq:caduceus} is 

\begin{align*}
    & A_1(x_1)D_2(x_2)c_2(1,2)b_2(1,\overline{2})a_1(\overline{1},2)b_2(\overline{1},\overline{2}) 
    + C_1(x_1) C_2(x_2) c_2(1,2)c_1(1,\overline{2})a_1(\overline{1},2)b_2(\overline{1},\overline{2}) + \\
    & B_1(x_1) C_2(x_2) c_2(1,2)a_1(1,\overline{2})a_1(\overline{1},2)c_1(\overline{1},\overline{2})
     + C_1(x_1) B_2(x_2) b_1(1,2)a_2(1,\overline{2})a_1(\overline{1},2)b_2(\overline{1},\overline{2}) +
     \\  &B_1(x_1) B_2(x_2) b_1(1,2)c_2(1,\overline{2})a_1(\overline{1},2)c_1(\overline{1},\overline{2}) + D_1(x_1)A_2(x_2)b_1(1,2)b_1(1,\overline{2})a_1(\overline{1},2)c_1(\overline{1},\overline{2})
\end{align*}

%a1=y*vee(x2,x1)*down(x2^(-1),x1)*1*down(x2^(-1),x1^(-1))
%a2=vee(x2,x1)*hat(x2^(-1),x1)*1*down(x2^(-1),x1^(-1))
%a3=vee(x2,x1)*1*1*hat(x2^(-1),x1^(-1))
%a4=up(x2,x1)*1*1*down(x2^(-1),x1^(-1))
%a5=up(x2,x1)*vee(x2^(-1),x1)*1*hat(x2^(-1),x1^(-1))
%a6=z*up(x2,x1)*up(x2^(-1),x1)*1*hat(x2^(-1),x1^(-1))

\vskip .1 in

In order for a set of Boltzmann weights to be of uniform regime, the bivalent weights must fall into one of the four cases described in Lemma \ref{le:fish}. Suppose first that we are in regime 1(a), so that $B_1(x) = C_1(x)$, $B_2(x) = C_2(x)$, and $A_1, A_2, D_1, D_2$ are polynomials in $q$. In this case, the partition functions for the $(1,0,1,0)$ and $(1,0,0,1)$ cases have the same denominator, so we can calculate the ratio of terms from \eqref{eq:caduceus}, $F$, by setting the two partition functions equal to each other. This yields 
\[\left(q^2(A_1D_2 - C_1(x_1) C_2(x_2)) - (A_2D_1 - C_1(x_1) C_2(x_2))\right)\frac{(1-q)(qx_2^2 - 1)(x_1 - \overline{x}_2)(\overline{x}_1 - \overline{x}_2)}{(x_1 - qx_2)(x_1 - q\overline{x}_2)(\overline{x}_1 - qx_2)(\overline{x}_1 - q\overline{x}_2)} = 0,\] 
and hence, in order to have $F$ be a constant independent of boundary conditions, it is necessary that 
\[q^2(A_1D_2-C_1(x_1) C_2(x_2)) - (A_2D_1-C_1(x_1) C_2(x_2)) = 0.\] Since $\textrm{deg}(A_j)=\textrm{deg}(D_j)=0$, it must be the case that $\textrm{deg}(C_j)=0$ as well. Upon checking the four remaining cases, we see that satisfying \eqref{eq:cad.condition} is also a sufficient condition to have $F$ be a constant independent of boundary conditions in the case $B_j=C_j$. It is readily verified that if \eqref{eq:cad.condition} is satisfied, then $F= 1$.

Next, consider the case where the weights are regime 2(b), so that $B_j(x) = qx^2C_j(x)$ and $A_1, A_2, D_1, D_2$ are polynomials in $q$. Calculating the ratio of terms from \eqref{eq:caduceus} in the cases $(1,0,1,0)$ and $(1,0,0,1)$ and equating them yields \[\frac{q^3C_1(x_1)C_2(x_2)x_1x_2 + q^2D_2A_1 - qC_1(x_1)C_2(x_2)x_1x_2 - A_2D_1}{C_1(x_2)C_2(x_1)} \cdot \frac{(1-q)(1-qx_2^2)(x_1-\bar{x}_2)(x_1-x_2)}{qx_1^2x_2^2(x_1-qx_2)(x_1-q\bar{x}_2)(\bar{x}_1-qx_2)(\bar{x}_1-q\bar{x}_2)} = 0. \] Since $\textrm{deg}(A_j) = \textrm{deg}(D_j) = 0$, we have that $\textrm{deg}(C_1(x)) = \textrm{deg}(C_2(x)) = -1$ and $\textrm{deg}(B_j(x)) = 1$. Using this fact in the analysis of the remaining cases, it can be shown that satisfying \eqref{eq:cad.condition2} is sufficient to ensure that $F$ is a constant independent of boundary conditions in this case. If $\eqref{eq:cad.condition2}$ is satisfied, then $F = 1$.

%-(q^3*v1*v2*x1*x2 - q*v1*v2*x1*x2 + q^2*y - z)*(q*x2^2 - 1)*(x1*x2 - 1)*(q - 1)*(x1 - x2)/((q*x1*x2 - 1)*(q*x1 - x2)*(q*x2 - x1)*(x1*x2 - q)*q*u1*u2*x2)

Now suppose we are in regime 1(b), so that $B_j(x) = -qC_j(x)$ and $A_j=D_j=0$. Then for $(1,0,1,0)$, 
\begin{align*}
F&= C_1(x_1)C_2(x_2) (a_2(1,2)b_1(1,\overline{2})b_2(\overline{1},2)a_1(\overline{1},\overline{2}) + c_2(1,2)c_1(1,\overline{2})c_1(\overline{1},2)c_2(\overline{1},\overline{2}) - qc_2(1,2)a_1(1,\overline{2})c_1(\overline{1},2)b_1(\overline{1},\overline{2}) \\
&- qb_1(1,2)a_2(1,\overline{2})c_1(\overline{1},2)c_2(\overline{1},\overline{2}) + q^2b_1(1,2)c_2(1,\overline{2})c_1(\overline{1},2)b_1(\overline{1},\overline{2}))/(C_1(x_2)C_2(x_1)) = (C_1(x_1)C_2(x_2))/(C_1(x_2)C_2(x_1)),
\end{align*}
and for $(1,0,0,1)$, 
\begin{align*}
F &= C_1(x_1)C_2(x_2)(c_2(1,2)c_1(1,\overline{2})a_1(\overline{1},2)b_2(\overline{1},\overline{2}) - qc_2(1,2)a_1(1,\overline{2})a_1(\overline{1},2)c_1(\overline{1},\overline{2}) - qb_1(1,2)a_2(1,\overline{2})a_1(\overline{1},2)b_2(\overline{1},\overline{2}) \\
&+ q^2b_1(1,2)c_2(1,\overline{2})a_1(\overline{1},2)c_1(\overline{1},\overline{2}))/(C_1(x_2)C_2(x_1)) = (C_1(x_1)C_2(x_2))/(C_1(x_2)C_2(x_1)).
\end{align*}
Upon checking the four remaining cases, it can be verified that $F= (C_1(x_1)C_2(x_2))/(C_1(x_2)C_2(x_1))$ in this case, regardless of the values of $\alpha,\beta,\gamma,\delta$.

Finally, consider regime 2(a), where $B_j(x) = -x^2C_j(x)$ and $A_j=D_j=0$. In this case, if the boundary conditions are $(1,0,1,0)$, then 
\begin{align*}
F&= C_1(x_1)C_2(x_2) (a_2(1,2)b_1(1,\overline{2})b_2(\overline{1},2)a_1(\overline{1},\overline{2}) + c_2(1,2)c_1(1,\overline{2})c_1(\overline{1},2)c_2(\overline{1},\overline{2}) - x^2c_2(1,2)a_1(1,\overline{2})c_1(\overline{1},2)b_1(\overline{1},\overline{2}) \\
&- x^2b_1(1,2)a_2(1,\overline{2})c_1(\overline{1},2)c_2(\overline{1},\overline{2}) + x^4b_1(1,2)c_2(1,\overline{2})c_1(\overline{1},2)b_1(\overline{1},\overline{2}))/(C_1(x_2)C_2(x_1)) = (C_1(x_1)C_2(x_2))/(C_1(x_2)C_2(x_1)).
\end{align*}
If the boundary conditions are $(1,0,0,1)$, then
\begin{align*}
F &= C_1(x_1)C_2(x_2)(c_2(1,2)c_1(1,\overline{2})a_1(\overline{1},2)b_2(\overline{1},\overline{2}) - x^2c_2(1,2)a_1(1,\overline{2})a_1(\overline{1},2)c_1(\overline{1},\overline{2}) - x^2b_1(1,2)a_2(1,\overline{2})a_1(\overline{1},2)b_2(\overline{1},\overline{2}) \\
&+ x^4b_1(1,2)c_2(1,\overline{2})a_1(\overline{1},2)c_1(\overline{1},\overline{2}))/(C_1(x_2)C_2(x_1)) = (C_1(x_1)C_2(x_2))/(C_1(x_2)C_2(x_1)).
\end{align*}
Upon checking the four remaining cases, it can be verified that $F= (C_1(x_1)C_2(x_2))/(C_1(x_2)C_2(x_1))$ in this case, regardless of the values of $\alpha,\beta,\gamma,\delta$.
\end{proof}

%Note that if we insist on having local weights in the case where $C_j=B_j=1$ , then \eqref{eq:cad.condition} is only satisfied if $A_j = D_j = 1$. {\color{violet} In the rank 2 case, if we instead set $A_1D_2 = 0$ and $A_2D_1 = 1-q^2$, then the partition function $\mathcal{Z}(B_{\lambda})$ of the resulting Type $B/C$ model matches the Type C Hall-Littlewood polynomial for $\lambda$.} {\color{violet} A similar choice of non-local weights for $A_j$, $D_j$ in rank 3 yields a Type $B/C$ model whose partition functions match Type C Hall-Littlewood polynomials; however, this strategy cannot be extended beyond rank 3.}

%\begin{lemma}[Caduceus relation, 3 particles]\label{le:caduceus3}
%Let $\alpha,\beta,\gamma,\delta$ be an arrangement of $\{0,1,1,1\}$.  Furthermore, assume that Lemma \ref{le:fish} holds, and that bend weights are assigned according to Figure \ref{fig:gen.bends}.  

%If $B_1=C_1\neq 0$ and $B_2=C_2\neq0$, then the quantities 

%are proportional with proportionality constant $F \in \mathbb{Z}[q,q^{-1}](x_1^{\pm 1},\cdots,x_r^{\pm 1})$ independent of the choice of boundary labels $\alpha, \beta, \gamma, \delta$ if and only if $B_2D_1= B_1D_2$, in which case $F=1$.
%\end{lemma}

\begin{lemma}[Caduceus relation, 3 particles]\label{le:caduceus3}
Let the $R$-matrix Boltzmann weights be given as in Figure 3, and assume that the bend weights are monomial with coefficients in $\mathbb{C}[q]$ of uniform regime. Then the quantities
\begin{equation}\label{eq:caduceus3}
\mathcal{Z}\left(
\begin{tikzpicture}[scale=.5,baseline={([yshift=-\the\dimexpr\fontdimen22\textfont2\relax]current bounding box.center)}]
\node [label=right:$x_1$] at (4,2) {};
\node [label=right:$\overline{x}_1$] at (4,1) {};
\node [label=right:$x_{2}$] at (4,0) {};
\node [label=right:$\overline{x}_{2}$] at (4,-1) {};
\node [label=left:$\mathbf{1}$] at (-.5,-.5) {};
\node [label=left:$\mathbf{2}$] at (-.5,1.5) {};

\node [circle,draw,scale=.6] at (3.5,2) {$\alpha$};
\node [circle,draw,scale=0.55] at (3.5,1) {$\beta$};
\node [circle,draw,scale=.6] at (3.5,0) {$\gamma$};
\node [circle,draw,scale=0.6] at (3.5,-1) {$\delta$};

\draw [-]
	(0,2) arc (90:270:.5);
\draw [-]
	(0,0) arc (90:270:.5);

\draw [-] (0,2) -- (1,2);
\draw [-] (2,2) -- (3,2);
\draw [-] (0,-1) -- (1,-1);
\draw [-] (2,-1) -- (3,-1);

\draw [-] (1,0) .. controls (.5,0) and (.5,1) .. (0,1);
\draw [-] (1,1) .. controls (.5,1) and (.5,0) .. (0,0);
\draw [-] (2,-1) .. controls (1.5,-1) and (1.5,0) .. (1,0);
\draw [-] (2,0) .. controls (1.5,0) and (1.5,-1) .. (1,-1);
\draw [-] (2,1) .. controls (1.5,1) and (1.5,2) .. (1,2);
\draw [-] (2,2) .. controls (1.5,2) and (1.5,1) .. (1,1);
\draw [-] (3,0) .. controls (2.5,0) and (2.5,1) .. (2,1);
\draw [-] (3,1) .. controls (2.5,1) and (2.5,0) .. (2,0);
\filldraw[black] (-.5,1.5) circle (4pt);
\filldraw[black] (-.5,-.5) circle (4pt);

\end{tikzpicture}\right)
 \quad \text{ and } \quad 
 \textrm{wt}\left(
\begin{tikzpicture}[scale=0.4,baseline={([yshift=-\the\dimexpr\fontdimen22\textfont2\relax]current bounding box.center)}]
\node [label=above:$ $] at (0,1) {};
\node [label=below:$ $] at (0,-1) {};
\node [label=left:$\mathbf{2}$] at (-1,0) {};
\node [circle,draw,scale=.6] at (.5,1) {$\alpha$};
\node [circle,draw,scale=0.55] at (.5,-1) {$\beta$};
\node [label=center:$x_1$] at (2,1) {};
\node [label=center:$\overline{x}_1$] at (2,-1) {};
\draw [-]
	(0,1) arc (90:180:1);
\filldraw[black] (-1,0) circle (4pt);
\draw [-]
	(-1,0) arc (180:270:1);
\draw [-]
	(0,1) arc (90:180:1);
\end{tikzpicture}\right) \cdot\textrm{wt}\left(
\begin{tikzpicture}[scale=0.4,baseline={([yshift=-\the\dimexpr\fontdimen22\textfont2\relax]current bounding box.center)}]
\node [label=above:$ $] at (0,1) {};
\node [label=below:$ $] at (0,-1) {};
\node [circle,draw,scale=.6] at (.5,1) {$\gamma$};
\node [circle,draw,scale=0.6] at (.5,-1) {$\delta$};
\node [label=left:$\mathbf{1}$] at (-1,0) {};
\node [label=center:$x_{2}$] at (2,1) {};
\node [label=center:$\overline{x}_{2}$] at (2,-1) {};
\draw [-]
	(0,1) arc (90:180:1);
\filldraw[black] (-1,0) circle (4pt);
\draw [-]
	(-1,0) arc (180:270:1);
\draw [-]
	(-1,0) arc (180:270:1);
\end{tikzpicture}\right)
\end{equation}
are proportional with proportionality constant $F \in \mathbb{Z}[q,q^{-1}](x_1^{\pm 1}, x_2^{\pm 2})$ independent of the choice of boundary labels $\alpha,\beta,\gamma, \delta$ in $\{0,1,1,1\}$ if and only if one of the following the following regime-dependent conditions hold:
\begin{enumerate}
\item when the weights are regime 1(a) and $C_1,C_2$ are both nonzero, we have $C_1,C_2 \in \mathbb{C}[q]$, and $C_2D_1 = C_1D_2$, in which case $F=1$; 
\item when the weights are regime 2(b), we have $B_j(x) = qm_jx$, $C_j(x) = m_j\bar{x}$, where $m_j \in \mathbb{C}[q]$, and $m_2D_1=q^2m_1D_2$, in which case $F=1$; and
\item when the weights are either regime 1(b), regime 2(a), or regime 1(a) with $C_1=C_2=0$, then the quantities in question are all zero, in which case they are vacuously proportional.
\end{enumerate}
\end{lemma} 

\begin{proof}

As an example, consider the case $(\alpha, \beta, \gamma, \delta) = (1,1,1,0)$. Note that there are three different states with these boundary conditions, depicted below with Boltzmann weights from Figures \ref{fig:R.weights} and \ref{fig:gen.bends}:

\vskip .1 in

\begin{center}
\begingroup
\renewcommand*{\arraystretch}{2}
\begin{tabular}{|c|c|c|}%c|

\hline

%%% PICTURE 1 %%%%
$\textrm{wt}\left(
\begin{tikzpicture}[scale=.5,baseline={([yshift=-\the\dimexpr\fontdimen22\textfont2\relax]current bounding box.center)}]
\node [label=right:$x_1$] at (3,2) {};
\node [label=right:$\overline{x}_1$] at (3,1) {};
\node [label=right:$x_2$] at (3,0) {};
\node [label=right:$\overline{x}_2$] at (3,-1) {};
\node [label=left:$\mathbf{1}$] at (-.5,-.5) {};
\node [label=left:$\mathbf{2}$] at (-.5,1.5) {};

\draw [-] 
    (3,-1) -- (2,-1) 
    .. controls (1.5,-1) and (1.5,0) .. (1,0) 
    .. controls (.5,0) and (.5,1) .. (0,1)
    arc (270:90:.5)
    -- (1,2)
    .. controls (1.5,2) and (1.5,1) .. (2,1)
    .. controls (2.5,1) and (2.5,0) .. (3,0);

\draw [-] 
    (3,2) -- (2,2)
    .. controls (1.5,2) and (1.5,1) .. (1,1)
    .. controls (.5,1) and (.5,0) .. (0,0)
    arc (90:270:.5)
    -- (1,-1)
    .. controls (1.5,-1) and (1.5,0) .. (2,0)
    .. controls (2.5,0) and (2.5,1) .. (3,1);
    
\draw [-,line width=1mm,red,opacity=.5,rounded corners=0pt]
    (3,2) -- (2,2)
    .. controls (1.5,2) and (1.5,1) .. (1,1)
    .. controls (.6,1) and (.6,.5) .. (0.5,0.5)
    .. controls (.4,.5) and (.4,0) .. (0,0)
    arc (90:180:.5);

\draw [-,line width=1mm,red,opacity=.5,rounded corners=0pt]
    (3,0) 
    .. controls (2.5,0) and (2.5,1) .. (2,1)
    .. controls (1.5,1) and (1.5,2) .. (1,2)
    -- (0,2)
    arc (90:270:.5)
    .. controls (0.5,1) and (0.5,0) .. (1,0)
    .. controls (1.4,0) and (1.4,-0.5) .. (1.5,-0.5)
    .. controls (1.6,-0.5) and (1.6,0) .. (2,0)
    .. controls (2.5,0) and (2.5,1) .. (3,1);
    
\filldraw[black] (-.5,1.5) circle (4pt);
\filldraw[black] (-.5,-.5) circle (4pt);
\end{tikzpicture}\right)$

&

%%% PICTURE 2 %%%%
$\textrm{wt}\left(
\begin{tikzpicture}[scale=.5,baseline={([yshift=-\the\dimexpr\fontdimen22\textfont2\relax]current bounding box.center)}]
\node [label=right:$x_1$] at (3,2) {};
\node [label=right:$\overline{x}_1$] at (3,1) {};
\node [label=right:$x_2$] at (3,0) {};
\node [label=right:$\overline{x}_2$] at (3,-1) {};
\node [label=left:$\mathbf{1}$] at (-.5,-.5) {};
\node [label=left:$\mathbf{2}$] at (-.5,1.5) {};

\draw [-] 
    (3,-1) -- (2,-1) 
    .. controls (1.5,-1) and (1.5,0) .. (1,0) 
    .. controls (.5,0) and (.5,1) .. (0,1)
    arc (270:90:.5)
    -- (1,2)
    .. controls (1.5,2) and (1.5,1) .. (2,1)
    .. controls (2.5,1) and (2.5,0) .. (3,0);

\draw [-] 
    (3,2) -- (2,2)
    .. controls (1.5,2) and (1.5,1) .. (1,1)
    .. controls (.5,1) and (.5,0) .. (0,0)
    arc (90:270:.5)
    -- (1,-1)
    .. controls (1.5,-1) and (1.5,0) .. (2,0)
    .. controls (2.5,0) and (2.5,1) .. (3,1);

\draw [-,line width=1mm,red,opacity=.5,rounded corners=0pt]
    (3,2) -- (2,2)
    .. controls (1.5,2) and (1.5,1) .. (1,1)
    .. controls (.6,1) and (.6,.5) .. (0.5,0.5)
    .. controls (.4,.5) and (.4,0) .. (0,0)
    arc (90:270:.5)
    -- (1,-1)
    .. controls (1.5,-1) and (1.5,0) .. (2,0)
    .. controls (2.5,0) and (2.5,1) .. (3,1);
    
\draw [-,line width=1mm,red,opacity=.5,rounded corners=0pt]
    (3,0) 
    .. controls (2.5,0) and (2.5,1) .. (2,1)
    .. controls (1.5,1) and (1.5,2) .. (1,2)
    -- (0,2)
    arc (90:180:.5);

\filldraw[black] (-.5,1.5) circle (4pt);
\filldraw[black] (-.5,-.5) circle (4pt);
\end{tikzpicture}\right)$

&

%%%% PICTURE 3 %%%%
$\textrm{wt}\left(
\begin{tikzpicture}[scale=.5,baseline={([yshift=-\the\dimexpr\fontdimen22\textfont2\relax]current bounding box.center)}]
\node [label=right:$x_1$] at (3,2) {};
\node [label=right:$\overline{x}_1$] at (3,1) {};
\node [label=right:$x_2$] at (3,0) {};
\node [label=right:$\overline{x}_2$] at (3,-1) {};
\node [label=left:$\mathbf{1}$] at (-.5,-.5) {};
\node [label=left:$\mathbf{2}$] at (-.5,1.5) {};

\draw [-] 
    (3,-1) -- (2,-1) 
    .. controls (1.5,-1) and (1.5,0) .. (1,0) 
    .. controls (.5,0) and (.5,1) .. (0,1)
    arc (270:90:.5)
    -- (1,2)
    .. controls (1.5,2) and (1.5,1) .. (2,1)
    .. controls (2.5,1) and (2.5,0) .. (3,0);

\draw [-] 
    (3,2) -- (2,2)
    .. controls (1.5,2) and (1.5,1) .. (1,1)
    .. controls (.5,1) and (.5,0) .. (0,0)
    arc (90:270:.5)
    -- (1,-1)
    .. controls (1.5,-1) and (1.5,0) .. (2,0)
    .. controls (2.5,0) and (2.5,1) .. (3,1);

\draw [-,line width=1mm,red,opacity=.5,rounded corners=0pt]
    (3,0) 
    .. controls (2.5,0) and (2.5,1) .. (2,1)
    .. controls (1.5,1) and (1.5,2) .. (1,2)
    -- (0,2)
    arc (90:270:.5)
    .. controls (0.4,1) and (0.4,0.5) .. (0.5,0.5)
    .. controls (0.6,0.5) and (0.6,1) .. (1,1)
    .. controls (1.5,1) and (1.5,2) .. (2,2)
    -- (3,2);

\draw [-,line width=1mm,red,opacity=.5,rounded corners=0pt]
    (-.5,-.5)
    arc (180:270:.5)
    -- (1,-1)
    .. controls (1.5,-1) and (1.5,0) .. (2,0)
    .. controls (2.5,0) and (2.5,1) .. (3,1);

\filldraw[black] (-.5,1.5) circle (4pt);
\filldraw[black] (-.5,-.5) circle (4pt);
\end{tikzpicture}\right)$

\\ \hline

{\footnotesize $C_1(x_1)D_2(x_2)a_2(1,2)a_2(1,\overline{2})a_2(\overline{1},2)c_2(\overline{1},\overline{2})$ }& {\footnotesize $D_1(x_1)C_2(x_2)a_2(1,2)c_2(1,\overline{2})a_2(\overline{1},2)b_1(\overline{1},\overline{2})$ }& {\footnotesize $B_1(x_1)D_2(x_2)a_2(1,2)b_1(1,\overline{2})a_2(\overline{1},2)b_1(\overline{1},\overline{2})$}

\\ \hline
\end{tabular}
\endgroup
\end{center}

\vskip .1 in

In this case, the ratio of terms from \eqref{eq:caduceus3}, $F$, is equal to 
\begin{align*}
& (C_1(x_1)D_2(x_2)a_2(1,2)a_2(1,\overline{2})a_2(\overline{1},2)c_2(\overline{1},\overline{2}) + D_1(x_1)C_2(x_2)a_2(1,2)c_2(1,\overline{2})a_2(\overline{1},2)b_1(\overline{1},\overline{2}) 
\\&+ B_1(x_1)D_2(x_2)a_2(1,2)b_1(1,\overline{2})a_2(\overline{1},2)b_1(\overline{1},\overline{2}))/C_1(x_1)D_2(x_2). 
\end{align*}

Similarly, in the case $(1,1,0,1)$, one can show that
\begin{align*}
F &= (C_1(x_1)D_2(x_2)a_2(1,2)a_2(1,\overline{2}) c_2(\overline{1},2)b_2(\overline{1},\overline{2}) + B_1(x_1)D_2(x_2)a_2(1,2)c_2(1,\overline{2})c_2(\overline{1},2)c_1(\overline{1},\overline{2})
\\&+ B_1(x_1)D_2(x_2)c_2(1,2)b_2(1,\overline{2})b_1(\overline{1},2)a_2(\overline{1},\overline{2}) + D_1(x_1)C_2(x_2)a_2(1,2)b_1(1,\overline{2})c_2(\overline{1},2)c_1(\overline{1},\overline{2}) 
\\&+ D_1(x_1)C_2(x_2)c_2(1,2)c_1(1,\overline{2})b_1(\overline{1},2)a_2(\overline{1},\overline{2}) + D_1(x_1)B_2(x_2)b_1(1,2)a_2(1,\overline{2})b_1(\overline{1},2)a_2(\overline{1},\overline{2}))/B_1(x_1)D_2(x_2). 
\end{align*}

%bb1=y*v*1*1*vee(x2,x1^(-1))*down(x2^(-1),x1^(-1));
%bb2=y*w*1*vee(x2^(-1),x1)*vee(x2,x1^(-1))*hat(x2^(-1),x1^(-1));
%bb3=y*w*vee(x2,x1)*down(x2^(-1),x1)*up(x2,x1^(-1))*1;
%bb4=v*z*1*up(x2^(-1),x1)*vee(x2,x1^(-1))*hat(x2^(-1),x1^(-1));
%bb5=v*z*vee(x2,x1)*hat(x2^(-1),x1)*up(x2,x1^(-1))*1;
%bb6=w*z*up(x2,x1)*1*up(x2,x1^(-1))*1;

In order for a set of Boltzmann weights to be of uniform regime, the bivalent weights must fall into one of the four cases described in Lemma \ref{le:fish}. If the bivalent weights satisfy cases 1(b), 2(a), or 1(a) with $C_1=C_2=0$ of Lemma \ref{le:fish}, then the result follows trivially, so suppose first that the weights are uniform of regime 1(a) with both $C_1$ and $C_2$ nonzero: $B_j(x) = C_j(x)$ and $A_1, A_2, D_1, D_2$ are polynomials in $q$. By equating the ratio of terms from \eqref{eq:caduceus3} and simplifying, we have \[ \frac{C_1(x_1)D_2-C_2(x_2)D_1}{C_1(x_1)D_2}\cdot \frac{(1-q^2)(qx_2^2-1)(x_1-x_2^{-1})(x_1^{-1}-x_2^{-1})}{(x_1-qx_2)(x_1-qx_2^{-1})(x_1^{-1}-qx_2)(x_1^{-1}-qx_2^{-1})}=0.\] 

In order to have the ratio of terms from \eqref{eq:caduceus3} be a constant independent of boundary conditions, it is necessary that $C_2(x_2)D_1=C_1(x_1)D_2$; in particular, since $\textrm{deg}(D_j) = 0$, it must be the case that $\textrm{deg}(C_j) = 0$. Examining the remaining two cases, the condition $C_2D_1=C_1D_2$ is easily shown to be sufficient. One can quickly verify that $F = 1$ in this case.

Suppose then that the model is uniform of regime 2(b): $B_j(x) = qx^2C_j(x)$ and $A_1, A_2, D_1, D_2 \in \mathbb{C}[q]$. Calculating the ratio of terms from \eqref{eq:caduceus3} in the cases $(1,1,1,0)$ and $(1,1,0,1)$ and equating them yields \[\frac{q^2x_1C_1(x_1)D_2-x_2C_2(x_2)D_1)}{qx_2C_1(x_2)D_2}\frac{(1-q^2)(1-qx_2^2)(x_1-\bar{x}_2)(x_1-x_2)}{x_1x_2(\bar{x}_1-qx_2)(\bar{x}_1-q\bar{x}_2)(x_1-qx_2)(x_1-q\bar{x}_2)}.\] Since $\textrm{deg}(D_j) = 0$, it must be the case that $\textrm{deg}(C_j) = -1$. Using this fact, it is easy to check that $F = 1$ in all four cases.

%-(q^2*v1*x1*y - v2*x2*z)*(q*x2^2 - 1)*(x1*x2 - 1)*(q + 1)*(q - 1)*(x1 - x2)*x1/((q*x1*x2 - 1)*(q*x1 - x2)*(q*x2 - x1)*(x1*x2 - q)*q*u1*x2*y)
\end{proof}

%\begin{lemma}[Caduceus relation, 1 particle]\label{le:caduceus1}
%Let $\alpha,\beta,\gamma,\delta$ be an arrangement of $\{0,0,0,1\}$. Furthermore, assume that Lemma \ref{le:fish} holds, and that bend weights are assigned according to Figure \ref{fig:gen.bends}.  

%If $B_1=C_1\neq 0$ and $B_2=C_2\neq 0$, then the quantities 

%are proportional with proportionality constant $F \in \mathbb{Z}[q,q^{-1}](x_1^{\pm 1},\cdots,x_r^{\pm 1})$ if and only if $B_2A_1= B_1A_2$, in which case $F=1$.

%\end{lemma}

\begin{lemma}[Caduceus relation, 1 particle]\label{le:caduceus1}
Let the $R$-matrix Boltzmann weights be given as in Figure 3, and assume that the bend weights are monomial with coefficients in $\mathbb{C}[q]$ of uniform regime.  %, as in Lemma \ref{le:fish}. If the Fish relation, as expressed in Lemma \ref{le:fish}, holds in both bends, 
Then the quantities

\begin{equation}\label{eq:caduceus1}
\mathcal{Z}\left(
\begin{tikzpicture}[scale=.5,baseline={([yshift=-\the\dimexpr\fontdimen22\textfont2\relax]current bounding box.center)}]
\node [label=right:$x_1$] at (4,2) {};
\node [label=right:$\overline{x}_1$] at (4,1) {};
\node [label=right:$x_{2}$] at (4,0) {};
\node [label=right:$\overline{x}_{2}$] at (4,-1) {};
\node [label=left:$\mathbf{1}$] at (-.5,-.5) {};
\node [label=left:$\mathbf{2}$] at (-.5,1.5) {};

\node [circle,draw,scale=.6] at (3.5,2) {$\alpha$};
\node [circle,draw,scale=0.55] at (3.5,1) {$\beta$};
\node [circle,draw,scale=.6] at (3.5,0) {$\gamma$};
\node [circle,draw,scale=0.6] at (3.5,-1) {$\delta$};

\draw [-]
	(0,2) arc (90:270:.5);
\draw [-]
	(0,0) arc (90:270:.5);

\draw [-] (0,2) -- (1,2);
\draw [-] (2,2) -- (3,2);
\draw [-] (0,-1) -- (1,-1);
\draw [-] (2,-1) -- (3,-1);

\draw [-] (1,0) .. controls (.5,0) and (.5,1) .. (0,1);
\draw [-] (1,1) .. controls (.5,1) and (.5,0) .. (0,0);
\draw [-] (2,-1) .. controls (1.5,-1) and (1.5,0) .. (1,0);
\draw [-] (2,0) .. controls (1.5,0) and (1.5,-1) .. (1,-1);
\draw [-] (2,1) .. controls (1.5,1) and (1.5,2) .. (1,2);
\draw [-] (2,2) .. controls (1.5,2) and (1.5,1) .. (1,1);
\draw [-] (3,0) .. controls (2.5,0) and (2.5,1) .. (2,1);
\draw [-] (3,1) .. controls (2.5,1) and (2.5,0) .. (2,0);
\filldraw[black] (-.5,1.5) circle (4pt);
\filldraw[black] (-.5,-.5) circle (4pt);

\end{tikzpicture}\right)
\quad \text{ and } \quad 
 \textrm{wt}\left(
\begin{tikzpicture}[scale=0.4,baseline={([yshift=-\the\dimexpr\fontdimen22\textfont2\relax]current bounding box.center)}]
\node [label=above:$ $] at (0,1) {};
\node [label=below:$ $] at (0,-1) {};
\node [label=left:$\mathbf{2}$] at (-1,0) {};
\node [circle,draw,scale=.6] at (.5,1) {$\alpha$};
\node [circle,draw,scale=0.55] at (.5,-1) {$\beta$};
\node [label=center:$x_1$] at (2,1) {};
\node [label=center:$\overline{x}_1$] at (2,-1) {};
\draw [-]
	(0,1) arc (90:180:1);
\filldraw[black] (-1,0) circle (4pt);
\draw [-]
	(-1,0) arc (180:270:1);
\draw [-]
	(0,1) arc (90:180:1);
\end{tikzpicture}\right)\cdot \textrm{wt}\left(
\begin{tikzpicture}[scale=0.4,baseline={([yshift=-\the\dimexpr\fontdimen22\textfont2\relax]current bounding box.center)}]
\node [label=above:$ $] at (0,1) {};
\node [label=below:$ $] at (0,-1) {};
\node [label=left:$\mathbf{1}$] at (-1,0) {};
\node [circle,draw,scale=.6] at (.5,1) {$\gamma$};
\node [circle,draw,scale=0.6] at (.5,-1) {$\delta$};
\node [label=center:$x_{2}$] at (2,1) {};
\node [label=center:$\overline{x}_{2}$] at (2,-1) {};
\draw [-]
	(0,1) arc (90:180:1);
\filldraw[black] (-1,0) circle (4pt);
\draw [-]
	(-1,0) arc (180:270:1);
\draw [-]
	(-1,0) arc (180:270:1);
\end{tikzpicture}\right)
\end{equation}
are proportional with proportionality constant $F \in \mathbb{Z}[q,q^{-1}](x_1^{\pm 1}, x_2^{\pm 2})$ independent of the choice of boundary labels $\alpha,\beta,\gamma, \delta$ in $\{0,0,0,1\}$ if and only if one of the following regime-dependent conditions hold:
\begin{enumerate}
\item when the weights are regime 1(a) and $C_1,C_2$ are both nonzero, we have $C_1,C_2 \in \mathbb{C}[q]$, and $C_2A_1 = C_1A_2$, in which case $F=1$; 
\item when the weights are regime 2(b), we have $B_j(x) = qm_jx$, $C_j(x) = m_j\bar{x}$ where $m_j \in \mathbb{C}[q]$, and $m_1A_2=q^2m_2A_1$,
in which case $F=1$; and
\item when the weights are either regime 1(b), regime 2(a), or regime 1(a) with $C_1=C_2=0$, then the quantities in question are all zero, in which case they are vacuously proportional.
\end{enumerate}
\end{lemma}

The proof of Lemma \ref{le:caduceus1} is analogous to the proof of Lemma \ref{le:caduceus3} and is left to the reader.

Now we are ready to define an analogous notion of solvability for these lattice models with bends. This isn't a standard definition in the literature, as this bend geometry is somewhat exceptional, but the following conditions ensure that we can perform an analogous ``train argument'' to the type A case in Figure~\ref{fig:train.type.A} for models with bends.

\begin{definition}\label{def:solvable} We say that a type $B/C$ model is {\bf solvable} (with respect to the trigonometric six-vertex model) if it satisfies the following three local identities:
\begin{itemize}
    \item the Yang-Baxter equation \eqref{eq:YBE} for all pairs $j,k$
    \item uniform regime in the Fish relation \eqref{eq:fish}% for all $j$
    \item the Caduceus relations \eqref{eq:caduceus}, \eqref{eq:caduceus3}, \eqref{eq:caduceus1} for all pairs $j,k$
\end{itemize}
\end{definition}

\begin{remark*}
Whereas our definition relies on these three families of relations, work in \cite{Wheeler-Zinn-Justin} instead focuses on the Yang-Baxter and Fish relations together with the reflection equation of Cherednik and Skylyanin (see for example \cite{Sklyanin}); for a diagrammatic presentation of this relation close to our point of view, see Equation~(23) in Section~2.7 of~\cite{Wheeler-Zinn-Justin}.  Given our choice of $R$-matrix weights from Figure \ref{fig:R.weights}, and in particular the unitarity relation in Proposition~\ref{thm:unitarity}, the reflection equation is equivalent to our Caduceus relations.

General solutions to the reflection equation (and hence the Caduceus relations) for the trigonometric $R$-matrix are explored in~\cite{deVega-Gonzalez-Ruiz} where the weight of the bend vertices is encoded in a so-called $K$-matrix, a notation that we'll use in subsequent sections. We've chosen to reprove the caduceus relations here, assuming the fish relation, because our subsequent analysis depending on the rank will require us to make distinctions between the ``one-,'' ``two-,'' and ``three-particle'' cases as presented above, while a solution to the reflection equation would assume all three are satisfied. Note also that we're assuming the Fish relation in our proof, simplifying some of the analysis. 
\end{remark*}

The solvability of the type $B/C$ model ensures that the partition function inherits functional equations under the action of the Weyl group of type $B/C$, otherwise known as the hyperoctahedral group with generating simple reflections \begin{align}\label{eq:hypteroctahedral.generators} s_i&: (x_1, \ldots, x_i, x_{i+1}, \ldots, x_r) \longmapsto (x_1, \ldots, x_{i+1}, x_i, \ldots x_r) \qquad \text{ for }1 \leq i \leq r-1\\
s_r&: (x_1, \ldots, x_{r-1}, x_r) \longmapsto (x_1, \ldots, x_{r-1}, x_r^{-1})\nonumber.
\end{align}

We demonstrate this hyperoctahedral symmetry of the partition function using the respective train arguments for the associated lattice models with bends in the next two lemmas.

\begin{lemma}\label{le:inversion} Let $\mathcal{B}_\lambda$ be a type $B/C$ solvable model with tetravalent Boltzmann weights from the trigonometric six-vertex model. Then if the bend weights in row $i$ satisfy condition 1(a) or 2(b) of Lemma~\ref{le:fish},
$\mathcal{Z}(\mathcal{B}_{\lambda})$ is invariant under $x_i\leftrightarrow \overline{x}_i$. If instead the bend weights in row $i$ satisfy condition 1(b) or 2(a) of Lemma~\ref{le:fish}, then $(x_i^{-1} - q x_i) \mathcal{Z}(\mathcal{B}_\lambda)$ is invariant under the inversion $x_i\leftrightarrow \overline{x}_i$.
\end{lemma}
\begin{proof}
We begin by adding a single $R$ vertex at the right edge of the ice between the rows with parameters $x_i$ and $\overline{x}_i$, as depicted below in diagram \eqref{eq:inversion1}. Since no paths exit through the right edge of the ice, the only admissible $R$ vertex is of type $a_1(k,j)$ in Figure~\ref{fig:R.weights} with $(k,j)=(\overline{i}, i).$
Hence, we have the equality of partition functions given below (depicted in the case where $i=1$, for simplicity):

\begin{equation}\label{eq:inversion1}
\mathcal{Z}(\mathcal{B}_{\lambda}) = \mathcal{Z}\left(   
\begin{tikzpicture}[scale=.5,baseline={([yshift=-\the\dimexpr\fontdimen22\textfont2\relax]current bounding box.center)}]
\node [label=right:$x_2$] at (5,2) {};
\node [label=right:$\overline{x}_2$] at (5,1) {};
\node [label=right:$\overline{x}_1$] at (5,0) {};
\node [label=right:$x_1$] at (5,-1) {};

\draw [-]
	(0,2) arc (90:270:.5);
\draw [-]
	(0,0) arc (90:270:.5);
\filldraw[black] (-.5,1.5) circle (4pt);
\filldraw[black] (-.5,-.5) circle (4pt);
\node [label=left: {$\mathbf{2}$}] at (-0.4,1.5) {};
\node [label=left: {$\mathbf{1}$}] at (-0.4,-0.5) {};

\draw [-] (0,2) -- (4,2);
\draw [-] (0,1) -- (4,1);
\draw [-] (0,0) -- (4,0);
\draw [-] (0,-1) -- (4,-1);

\draw [-] (0.5,2.5) -- (0.5,-1.5);
\draw [-] (1.5,2.5) -- (1.5,-1.5);
\draw [-] (2.5,2.5) -- (2.5,-1.5);
\draw [-] (3.5,2.5) -- (3.5,-1.5);

\draw [-] (5,-1) .. controls (4.5,-1) and (4.5,0) .. (4,0);
\draw [-] (5,0) .. controls (4.5,0) and (4.5,-1) .. (4,-1);

\end{tikzpicture} \right).\end{equation}
From here, we can repeatedly apply the Yang-Baxter equation, so that the partition function for diagram \eqref{eq:inversion1} is equal to

\begin{equation}\label{eq:inversion2}
\mathcal{Z}\left(  \!\! 
\begin{tikzpicture}[scale=.5,baseline={([yshift=-\the\dimexpr\fontdimen22\textfont2\relax]current bounding box.center)}]
\node [label=right:$x_2$] at (5,2) {};
\node [label=right:$\overline{x}_2$] at (5,1) {};
\node [label=right:$\overline{x}_1$] at (5,0) {};
\node [label=right:$x_1$] at (5,-1) {};

\draw [-]
	(0,2) arc (90:270:.5);
\draw [-]
	(0,0) arc (90:270:.5);
\filldraw[black] (-.5,1.5) circle (4pt);
\filldraw[black] (-.5,-.5) circle (4pt);
\node [label=left: {$\mathbf{2}$}] at (-0.4,1.5) {};
\node [label=left: {$\mathbf{1}$}] at (-0.4,-0.5) {};

\draw [-] (0,2) -- (3,2);
\draw [-] (0,1) -- (3,1);
\draw [-] (0,0) -- (3,0);
\draw [-] (0,-1) -- (3,-1);

\draw [-] (0.5,2.5) -- (0.5,-1.5);
\draw [-] (1.5,2.5) -- (1.5,-1.5);
\draw [-] (2.5,2.5) -- (2.5,-1.5);

\draw [-] (3,2) -- (5,2);
\draw [-] (3,1) -- (5,1);
\draw [-] (4,0) -- (5,0);
\draw [-] (4,-1) -- (5,-1);

\draw [-] (4.5,2.5) -- (4.5,-1.5);

\draw [-] (4,-1) .. controls (3.5,-1) and (3.5,0) .. (3,0);
\draw [-] (4,0) .. controls (3.5,0) and (3.5,-1) .. (3,-1);

\end{tikzpicture} \right) = \cdots =  \mathcal{Z}\left(  \!\!
\begin{tikzpicture}[scale=.5,baseline={([yshift=-\the\dimexpr\fontdimen22\textfont2\relax]current bounding box.center)}]
\node [label=right:$x_2$] at (5,2) {};
\node [label=right:$\overline{x}_2$] at (5,1) {};
\node [label=right:$\overline{x}_1$] at (5,0) {};
\node [label=right:$x_1$] at (5,-1) {};

\draw [-]
	(0,2) arc (90:270:.5);
\draw [-]
	(0,0) arc (90:270:.5);
\filldraw[black] (-.5,1.5) circle (4pt);
\filldraw[black] (-.5,-.5) circle (4pt);
\node [label=left: {$\mathbf{2}$}] at (-0.4,1.5) {};
\node [label=left: {$\mathbf{1}$}] at (-0.4,-0.5) {};

\draw [-] (0,2) -- (5,2);
\draw [-] (0,1) -- (5,1);
\draw [-] (1,0) -- (5,0);
\draw [-] (1,-1) -- (5,-1);

\draw [-] (1.5,2.5) -- (1.5,-1.5);
\draw [-] (2.5,2.5) -- (2.5,-1.5);
\draw [-] (3.5,2.5) -- (3.5,-1.5);
\draw [-] (4.5,2.5) -- (4.5,-1.5);

\draw [-] (1,-1) .. controls (.5,-1) and (.5,0) .. (0,0);
\draw [-] (1,0) .. controls (.5,0) and (.5,-1) .. (0,-1);

\end{tikzpicture} \right) = 
%\left(\frac{1-qx_1^2}{1-x_1^2}\right)
F \cdot \mathcal{Z}\left(   \!\!
\begin{tikzpicture}[scale=.5,baseline={([yshift=-\the\dimexpr\fontdimen22\textfont2\relax]current bounding box.center)}]
\node [label=right:$x_2$] at (4,2) {};
\node [label=right:$\overline{x}_2$] at (4,1) {};
\node [label=right:$\overline{x}_1$] at (4,0) {};
\node [label=right:$x_1$] at (4,-1) {};

\draw [-]
	(0,2) arc (90:270:.5);
\draw [-]
	(0,0) arc (90:270:.5);
\filldraw[black] (-.5,1.5) circle (4pt);
\filldraw[black] (-.5,-.5) circle (4pt);
\node [label=left: {$\mathbf{2}$}] at (-0.4,1.5) {};
\node [label=left: {$\mathbf{1}$}] at (-0.4,-0.5) {};

\draw [-] (0,2) -- (4,2);
\draw [-] (0,1) -- (4,1);
\draw [-] (0,0) -- (4,0);
\draw [-] (0,-1) -- (4,-1);

\draw [-] (0.5,2.5) -- (0.5,-1.5);
\draw [-] (1.5,2.5) -- (1.5,-1.5);
\draw [-] (2.5,2.5) -- (2.5,-1.5);
\draw [-] (3.5,2.5) -- (3.5,-1.5);
\end{tikzpicture} \right).
\end{equation}
where we have used an application of the Fish relation (Lemma \ref{le:fish}) in the final equality, and the respective cases result from the two possibilities for the constant of proportionality $F$ in the fish relation.
\end{proof}

\begin{lemma}\label{le:symmetric}
Let $\mathcal{B}_\lambda$ be a type $B/C$ solvable model with tetravalent Boltzmann weights from the trigonometric six-vertex model. Then if the associated caduceus constants $F=1$ for all simple transpositions $s_i = (i \; i+1)$ with $1 \leq i \leq r-1$, then $\mathcal{Z}(\mathcal{B}_{\lambda})$ is a symmetric function under permutation of the $\{x_i\}$.
\end{lemma}
\begin{proof}
Similarly to the proof of Lemma \ref{le:inversion}, we start by adding four $R$ vertices at the right edge of the ice, as depicted below in rank two for simplicity:
\begin{equation}\label{eq:symmetric1}
%\begin{aligned}
%\left(\frac{x_1-qx_2}{x_1-x_2} \right) \left(\frac{x_1-q\overline{x}_2}{x_1-\overline{x}_2} \right)\cdot & \\
%\left(\frac{\overline{x}_1-qx_2}{\overline{x}_1-x_2} \right)
%\left(\frac{\overline{x}_1-q\overline{x}_2}{\overline{x}_1-\overline{x}_2} \right) & 
\mathcal{Z}(\boldsymbol{x}; \mathcal{B}_{\lambda}) = 
\mathcal{Z}\left(   
\begin{tikzpicture}[scale=.5,baseline={([yshift=-\the\dimexpr\fontdimen22\textfont2\relax]current bounding box.center)}]
\node [label=right:$x_2$] at (4,2) {};
\node [label=right:$\overline{x}_2$] at (4,1) {};
\node [label=right:$x_1$] at (4,0) {};
\node [label=right:$\overline{x}_1$] at (4,-1) {};

\draw [-]
	(0,2) arc (90:270:.5);
\draw [-]
	(0,0) arc (90:270:.5);
\filldraw[black] (-.5,1.5) circle (4pt);
\filldraw[black] (-.5,-.5) circle (4pt);
\node [label=left: {$\mathbf{2}$}] at (-0.4,1.5) {};
\node [label=left: {$\mathbf{1}$}] at (-0.4,-0.5) {};

\draw [-] (0,2) -- (4,2);
\draw [-] (0,1) -- (4,1);
\draw [-] (0,0) -- (4,0);
\draw [-] (0,-1) -- (4,-1);

\draw [-] (0.5,2.5) -- (0.5,-1.5);
\draw [-] (1.5,2.5) -- (1.5,-1.5);
\draw [-] (2.5,2.5) -- (2.5,-1.5);
\draw [-] (3.5,2.5) -- (3.5,-1.5);

\end{tikzpicture} \right) = \mathcal{Z}\left(   
\begin{tikzpicture}[scale=.5,baseline={([yshift=-\the\dimexpr\fontdimen22\textfont2\relax]current bounding box.center)}]
\node [label=right:$x_1$] at (7,2) {};
\node [label=right:$\overline{x}_1$] at (7,1) {};
\node [label=right:$x_2$] at (7,0) {};
\node [label=right:$\overline{x}_2$] at (7,-1) {};

\draw [-]
	(0,2) arc (90:270:.5);
\draw [-]
	(0,0) arc (90:270:.5);
\filldraw[black] (-.5,1.5) circle (4pt);
\filldraw[black] (-.5,-.5) circle (4pt);
\node [label=left: {$\mathbf{2}$}] at (-0.4,1.5) {};
\node [label=left: {$\mathbf{1}$}] at (-0.4,-0.5) {};

\draw [-] (0,2) -- (5,2);
\draw [-] (6,2) -- (7,2);
\draw [-] (0,1) -- (4,1);
\draw [-] (0,0) -- (4,0);
\draw [-] (0,-1) -- (5,-1);
\draw [-] (6,-1) -- (7,-1);

\draw [-] (0.5,2.5) -- (0.5,-1.5);
\draw [-] (1.5,2.5) -- (1.5,-1.5);
\draw [-] (2.5,2.5) -- (2.5,-1.5);
\draw [-] (3.5,2.5) -- (3.5,-1.5);

\draw [-] (5,0) .. controls (4.5,0) and (4.5,1) .. (4,1);
\draw [-] (5,1) .. controls (4.5,1) and (4.5,0) .. (4,0);
\draw [-] (6,-1) .. controls (5.5,-1) and (5.5,0) .. (5,0);
\draw [-] (6,0) .. controls (5.5,0) and (5.5,-1) .. (5,-1);
\draw [-] (6,1) .. controls (5.5,1) and (5.5,2) .. (5,2);
\draw [-] (6,2) .. controls (5.5,2) and (5.5,1) .. (5,1);
\draw [-] (7,0) .. controls (6.5,0) and (6.5,1) .. (6,1);
\draw [-] (7,1) .. controls (6.5,1) and (6.5,0) .. (6,0);

\end{tikzpicture} \right).
%\end{aligned}
\end{equation}
Note that, as in the proof of Lemma \ref{le:inversion}, the boundary conditions along the right edge of the diagram mean that all four of these $R$ vertices are of type $a_1(k,j)$ in Figure~\ref{fig:R.weights} with $(k,j)$ in $\{1,\overline{1},2,\overline{2}\}.$
From here, we can repeatedly apply the Yang-Baxter equation \eqref{eq:YBE}, so that the partition function in \eqref{eq:symmetric1} is equal to 
\begin{align}\label{eq:symmetric2}
\mathcal{Z}&\left(\begin{tikzpicture}[scale=.5,baseline={([yshift=-\the\dimexpr\fontdimen22\textfont2\relax]current bounding box.center)}]
\node [label=right:$x_1$] at (7,2) {};
\node [label=right:$\overline{x}_1$] at (7,1) {};
\node [label=right:$x_2$] at (7,0) {};
\node [label=right:$\overline{x}_2$] at (7,-1) {};
%\node [label=left:$\mathbf{1}$] at (-.5,-.5) {};
%\node [label=left:$\mathbf{2}$] at (-.5,1.5) {};
\draw [-]
	(0,2) arc (90:270:.5);
\draw [-]
	(0,0) arc (90:270:.5);
\filldraw[black] (-.5,1.5) circle (4pt);
\filldraw[black] (-.5,-.5) circle (4pt);
\node [label=left: {$\mathbf{2}$}] at (-0.4,1.5) {};
\node [label=left: {$\mathbf{1}$}] at (-0.4,-0.5) {};
\draw [-] (0,2) -- (4,2);
\draw [-] (5,2) -- (7,2);
\draw [-] (0,1) -- (3,1);
\draw [-] (6,1) -- (7,1);
\draw [-] (0,0) -- (3,0);
\draw [-] (6,0) -- (7,0);
\draw [-] (0,-1) -- (4,-1);
\draw [-] (5,-1) -- (7,-1);
\draw [-] (0.5,2.5) -- (0.5,-1.5);
\draw [-] (1.5,2.5) -- (1.5,-1.5);
\draw [-] (2.5,2.5) -- (2.5,-1.5);
\draw [-] (6.5,2.5) -- (6.5,-1.5);
\draw [-] (4,0) .. controls (3.5,0) and (3.5,1) .. (3,1);
\draw [-] (4,1) .. controls (3.5,1) and (3.5,0) .. (3,0);
\draw [-] (5,-1) .. controls (4.5,-1) and (4.5,0) .. (4,0);
\draw [-] (5,0) .. controls (4.5,0) and (4.5,-1) .. (4,-1);
\draw [-] (5,1) .. controls (4.5,1) and (4.5,2) .. (4,2);
\draw [-] (5,2) .. controls (4.5,2) and (4.5,1) .. (4,1);
\draw [-] (6,0) .. controls (5.5,0) and (5.5,1) .. (5,1);
\draw [-] (6,1) .. controls (5.5,1) and (5.5,0) .. (5,0);
\end{tikzpicture} \right)
= \mathcal{Z}\left(
\begin{tikzpicture}[scale=.5,baseline={([yshift=-\the\dimexpr\fontdimen22\textfont2\relax]current bounding box.center)}]
\node [label=right:$x_1$] at (7,2) {};
\node [label=right:$\overline{x}_1$] at (7,1) {};
\node [label=right:$x_2$] at (7,0) {};
\node [label=right:$\overline{x}_2$] at (7,-1) {};
\node [label=left: {$\mathbf{2}$}] at (-0.4,1.5) {};
\node [label=left: {$\mathbf{1}$}] at (-0.4,-0.5) {};
\draw [-]
	(0,2) arc (90:270:.5);
\draw [-]
	(0,0) arc (90:270:.5);
\filldraw[black] (-.5,1.5) circle (4pt);
\filldraw[black] (-.5,-.5) circle (4pt);
\node [label=left: {$\mathbf{2}$}] at (-0.4,1.5) {};
\node [label=left: {$\mathbf{1}$}] at (-0.4,-0.5) {};
\draw [-] (0,2) -- (1,2);
\draw [-] (2,2) -- (7,2);
\draw [-] (3,1) -- (7,1);
\draw [-] (3,0) -- (7,0);
\draw [-] (0,-1) -- (1,-1);
\draw [-] (2,-1) -- (7,-1);
\draw [-] (3.5,2.5) -- (3.5,-1.5);
\draw [-] (4.5,2.5) -- (4.5,-1.5);
\draw [-] (5.5,2.5) -- (5.5,-1.5);
\draw [-] (6.5,2.5) -- (6.5,-1.5);
\draw [-] (1,0) .. controls (.5,0) and (.5,1) .. (0,1);
\draw [-] (1,1) .. controls (.5,1) and (.5,0) .. (0,0);
\draw [-] (2,-1) .. controls (1.5,-1) and (1.5,0) .. (1,0);
\draw [-] (2,0) .. controls (1.5,0) and (1.5,-1) .. (1,-1);
\draw [-] (2,1) .. controls (1.5,1) and (1.5,2) .. (1,2);
\draw [-] (2,2) .. controls (1.5,2) and (1.5,1) .. (1,1);
\draw [-] (3,0) .. controls (2.5,0) and (2.5,1) .. (2,1);
\draw [-] (3,1) .. controls (2.5,1) and (2.5,0) .. (2,0);
\end{tikzpicture}\right) \\ 
&= \mathcal{Z}\left(   
\begin{tikzpicture}[scale=.5,baseline={([yshift=-\the\dimexpr\fontdimen22\textfont2\relax]current bounding box.center)}]
\node [label=right:$x_1$] at (4,2) {};
\node [label=right:$\overline{x}_1$] at (4,1) {};
\node [label=right:$x_2$] at (4,0) {};
\node [label=right:$\overline{x}_2$] at (4,-1) {};
\node [label=left: {$\mathbf{2}$}] at (-0.4,1.5) {};
\node [label=left: {$\mathbf{1}$}] at (-0.4,-0.5) {};
\draw [-]
	(0,2) arc (90:270:.5);
\draw [-]
	(0,0) arc (90:270:.5);
\filldraw[black] (-.5,1.5) circle (4pt);
\filldraw[black] (-.5,-.5) circle (4pt);
\draw [-] (0,2) -- (4,2);
\draw [-] (0,1) -- (4,1);
\draw [-] (0,0) -- (4,0);
\draw [-] (0,-1) -- (4,-1);
\draw [-] (0.5,2.5) -- (0.5,-1.5);
\draw [-] (1.5,2.5) -- (1.5,-1.5);
\draw [-] (2.5,2.5) -- (2.5,-1.5);
\draw [-] (3.5,2.5) -- (3.5,-1.5);
\end{tikzpicture} \right) 
= 
\mathcal{Z}(s_1(\boldsymbol{x});\mathcal{B}_\lambda) , \nonumber
\end{align}
where we apply the Caduceus relation (Lemma \ref{le:caduceus}, \ref{le:caduceus3}, or \ref{le:caduceus1}) in the final equality, noting that the constant of proportionality $F$ is equal to 1 in all cases, to obtain the desired result.
\end{proof}

Further generalizations of the prior Lemma are possible using other constants of proportionality from the caduceus relations in Lemmas~\ref{le:caduceus}, \ref{le:caduceus3}, and \ref{le:caduceus1}, but we refrain from stating them to focus on the case of symmetric functions. Indeed, combining the previous two lemmas, we may immediately conclude the following result, noting that the subcases involved imply that the caduceus constants $F$ are all equal to one.

\begin{proposition} Let $\mathcal{B}_\lambda$ be a type $B/C$ solvable model with tetravalent Boltzmann weights from the trigonometric six-vertex model. If the bend weights are uniform of case 1(a) or 2(b) in Lemma 3.1, then $\mathcal{Z}(\boldsymbol{x}; \mathcal{B}_\lambda)$ is invariant under the usual action of the hyperoctahedral group on $\boldsymbol{x}$ given above.
\end{proposition}

\section{General Methods for Closed Form Solutions of Solvable B/C Models\label{sec:genmethods}}

We can now take advantage of solvability, providing a general strategy for evaluating the rank $r$ partition functions for lattice models with bends, mimicking the approach taken in Section 3.4 of \cite{Wheeler-Zinn-Justin}. In particular, we adopt the Dirac bra-ket notation to describe the partition function:
\begin{equation} \mathcal{Z}(\mathcal{B}_\lambda) = \braket{\mathbf{K} | \mathbf{C}_0 \mathbf{C}_1 \cdots \mathbf{C}_{\lambda_r} | \varnothing}, \label{mainthmbraket} \end{equation}
where the ket $\ket{\varnothing}$ denotes the right-hand boundary of $2r$ empty rows, viewed as a column vector with $2r$ components. The operators $\mathbf{C}_i$ denote the column transfer matrix in column $i$, and the $\mathbf{K}$ operator in $\bra{K}$ denotes the partition function of the resulting $r$ caps along the left boundary. If we let $W_i$ and $W_{\bar{i}}$ denote the two-dimensional vector space in rows $i$ and $\bar{i}$, respectively, with basis indexed by particle and hole, then $\ket{\varnothing}$ represents the (column) vector of $2r$ holes in $W_r \otimes W_{\bar{r}} \otimes \cdots \otimes W_1 \otimes W_{\bar{1}}$ and each column transfer matrix $\mathbf{C}_i$ is an element of $\textrm{End}(W_r \otimes \cdots \otimes W_{\bar{1}})$. In \cite{Wheeler-Zinn-Justin}, the analysis of this braket is greatly simplified by the use of a matrix $\mathbf{F} := F_{r,\bar{r},\ldots,1,\bar{1}}$ in $\textrm{End}(W_r \otimes \cdots \otimes W_{\bar{1}})$ such that
$$ \mathbf{F}^\sigma \mathbf{R}_\sigma = \mathbf{F} $$
where $\mathbf{R}_\sigma$ is the $R$-matrix that acts on the left of a column of particles by a tangle of crossings (i.e., $R$-matrices of simple transpositions), taking row $i$ to row $\sigma(i)$, and $\mathbf{F}^\sigma := \mathbf{F}_{\sigma(r), \sigma(\bar{r}),\ldots, \sigma(1), \sigma(\bar{1})}$. For example, in the case $\sigma = \left( \begin{array}{c} 2 \bar{2} 1 \bar{1} \\ \bar{2} 1 \bar{1} 2 \end{array} \right)$ in two-line permutation notation, then $\mathbf{R}_\sigma$ has the form: %{\color{red} (If we are following WZJ here, then I think we need to relabel the left of this diagram. If we want to keep this same set of paths, then the permutation should be $(\sigma(2), \sigma(\overline{2}), \sigma(1), \sigma(\overline{1})) = (\overline{2}, 1, \overline{1}, 2)$).}
$$
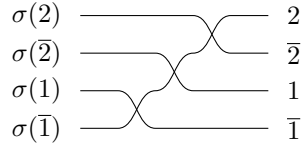
\begin{tikzpicture}[scale=.5, baseline=.5ex]
\node [label=right:$2$] at (4,2) {};
\node [label=right:$\overline{2}$] at (4,1) {};
\node [label=right:$1$] at (4,0) {};
\node [label=right:$\overline{1}$] at (4,-1) {};
\node [label=left:$\sigma(2)$] at (-1,2) {};
\node [label=left:$\sigma(\overline{2})$] at (-1,1) {};
\node [label=left:$\sigma(1)$] at (-1,0) {};
\node [label=left:$\sigma(\overline{1})$] at (-1,-1) {};

\draw [-] 
    (4,2) -- (3,2)
    .. controls (2.5,2) and (2.5,1) .. (2,1)
    .. controls (1.5,1) and (1.5,0) .. (1,0)
    .. controls (0.5,0) and (0.5,-1) .. (0,-1)
    -- (-1,-1);

\draw [-]    
    (4,1) -- (3,1)
    .. controls (2.5,1) and (2.5,2) .. (2,2)
    -- (-1,2);
    
\draw [-]    
    (4,0) -- (2,0)
    .. controls (1.5,0) and (1.5,1) .. (1,1)
    -- (-1,1);
    
\draw [-]    
    (4,-1) -- (1,-1)
    .. controls (.5,-1) and (.5,0) .. (0,0)
    -- (-1,0);

\end{tikzpicture}
$$

%({\color{red} Andy will go through and make changes.  These includes (1) changing the $1,2,3,\cdots to $r,\bar{r},\cdots$, (2) reversing all inequalities in sight, (3) introducing some notation like $\mathcal{I} = \{r,\bar{r},\cdots\}$ as well as some explicit statement about $r \succ \bar{r} \succ \cdots$, and (4) including some exposition to point out that we're adapting the technique from WZJ into our notational conventions}) 

It is not immediately clear that $\mathbf{R}_\sigma$ is well-defined.  In order for the value of $\mathbf{R}_\sigma$ to be consistent regardless of how one writes $\sigma$ as a product of simple reflections, one needs to ensure that products of associated $R$-matrices satisfy the defining relations of the symmetric group. The Yang-Baxter equations (Proposition~\ref{le:YBE}) give the braid relations, while the unitarity relation (Proposition~\ref{thm:unitarity}) shows that R-matrices for simple reflections square to the identity, ensuring that crossing strands and crossing back result in no change.

An explicit expression for $\mathbf{F}_{r,\bar{r},\cdots,1,\bar{1}}$ is given by making some small changes to an analogous computation in \cite[Appendix A]{Wheeler-Zinn-Justin}. To state this result, we will assume the ordering $r \succ \bar{r} \succ \cdots \succ 1 \succ \bar{1}$, we will write $\mathcal{I}=\{r,\bar{r},\cdots,1,\bar{1}\}$, and for any given $i \in \mathcal{I}\setminus \{r\}$ we will write $i\dotplus 1$ for the immediate successor of $i$ in $\mathcal{I}$ (e.g., $1 \dotplus 1 = \bar{2}$ and $\bar{2}\dotplus 1 = 2$):
$$ \mathbf{F}_{r,\bar{r},\cdots,1,\bar{1}} = \sum_{i\in\mathcal{I}} \left( \prod_{i \preccurlyeq j \preccurlyeq r} \mathbf{E}_j^{(22)} \prod_{\bar{1} \preccurlyeq k \prec i} \mathbf{E}_k^{(11)} \right) + \sum_{i\in \mathcal{I}} \mathop{\sum_{\rho \in S_{\mathcal{I}}}}_{\rho(i)\succ\rho(i\dotplus 1)} \left(\prod_{i \preccurlyeq j \preccurlyeq r} \mathbf{E}_{\rho(j)}^{(22)} \prod_{\bar{1}\preccurlyeq k \prec i} \mathbf{E}_{\rho(k)}^{(11)}\right)\mathbf{R}_\rho^{r \bar{r}\dots 1\bar{1}} $$
where $\mathbf{E}_j^{(kk)}$ denotes the elementary $2 \times 2$ matrix with entry 1 at position $(k,k)$ and 0 elsewhere, as an endomorphism of the vector space $W_j$, and the final sum is over permutations $\rho$ with a single inversion $\rho(i) \succ \rho(i\dotplus 1)$. As noted in \cite{Wheeler-Zinn-Justin}, one can determine $\mathbf{F}^{-1}$ by identifying a natural operator 
\begin{equation} \label{eq:fstarform} \mathbf{F}^\ast_{r,\bar{r},\cdots,1,\bar{1}}= \sum_{i\in\mathcal{I}} 
\left(\prod_{i \preccurlyeq j \preccurlyeq r} \mathbf{E}_j^{(11)} \prod_{\bar{1}\preccurlyeq k \prec i} \mathbf{E}_k^{(22)}\right)+\sum_{i\in\mathcal{I}} \mathop{\sum_{\rho \in S_{\mathcal{I}}}}_{\rho(i)\succ\rho(i\dotplus 1)} \mathbf{R}^\rho_{r\bar{r}\cdots 1\bar{1}} \left(\prod_{i \preccurlyeq j \preccurlyeq r} \mathbf{E}_{\rho(j)}^{(11)} \prod_{\bar{1} \preccurlyeq k \prec i} \mathbf{E}_{\rho(k)}^{(22)}\right)
\end{equation}
such that $\mathbf{F} \mathbf{F}^\ast$ is an explicit diagonal matrix. Specifically, if we define \begin{equation}\label{eq:Delta}
[\Delta_{kl}(x_k,x_l)]^{j_kj_l}_{i_ki_l} = \delta_{i_k,j_k}\delta_{i_l,j_l} b_{i_k,i_l}(x_k,x_l), \qquad b_{i_k,i_l}(x_k,x_l) = \left\{\begin{array}{ll}1,&\text{ if } i_k=i_l,\\(x_l-x_k)/(x_l-qx_k),&\text{ if }i_k<i_l,\\(x_k-x_l)/(x_k-qx_l),&\text{ if }i_k>i_l,\end{array}\right.\end{equation}
then we have $\mathbf{F}_{r\dots \bar{1}}\mathbf{F}^\ast_{r\dots \bar{1}} = \prod_{\bar{1} \preccurlyeq l \prec k \preccurlyeq r} \Delta_{kl}(x_k,x_l)$, and so we have $\mathbf{F}^{-1}_{r \cdots \bar{1}} = \mathbf{F}^\ast _{r \cdots \bar{1}} \prod_{\bar{1} \preccurlyeq l \prec k \preccurlyeq r} \Delta^{-1}_{kl} (x_k,x_l).$ Using this explicit formula, one can check that
\begin{equation} \label{finvonvacuum}
     \mathbf{F}_{r \cdots \bar{1}}^{-1} \ket{\varnothing} = \ket{\varnothing}. 
\end{equation}

With these facts in hand, we can return to analyzing the braket in (\ref{mainthmbraket}). We find
\begin{eqnarray}
\braket{\mathbf{K} | \mathbf{C}_0 \cdots \mathbf{C}_{\lambda_r} | \varnothing} & = & \braket{\mathbf{K} | \mathbf{F}^{-1} \mathbf{F} \mathbf{C}_0 \mathbf{F}^{-1} \mathbf{F} \cdots \mathbf{C}_{\lambda_r} \mathbf{F}^{-1}| \varnothing} \\
& = & \braket{ \mathbf{K} | \mathbf{F}^{-1} \widetilde{\mathbf{C}}_0 \cdots \widetilde{\mathbf{C}}_{\lambda_r} | \varnothing } \label{eq:twoparts}
\end{eqnarray}
where we have used (\ref{finvonvacuum}) and introduced the notation $\widetilde{\mathbf{C}}_i := \mathbf{F} \mathbf{C}_i \mathbf{F}^{-1}$ for a new family of column operators, the so called ``twisted'' column operators. The Boltzmann weights for the column operators $\widetilde{\mathbf{C}}_i$ are no longer local, but rather depend on the configurations above and below it in the same column. Nevertheless, the weights may be given explicitly in Figure \ref{fig:nonlocal.weights}.

\begin{figure}[!ht]
\begin{tabular}
{|c|c|c|c|c|c|c|c|}
\hline
%&&&&&\\
$\begin{tikzpicture}
\coordinate (left) at (0.25,1);
\coordinate (right) at (1.75,1);
\coordinate (top) at (1,1.75);
\coordinate (bottom) at (1,0.25);
\node [label=left:$x_j$] (L) at (left) {};
\node [label=right:$ $] (R) at (right) {};
\draw [-] (L.east) -- (R.west);
\draw [-] (top) -- (bottom);
\draw [-,line width=1mm,red,opacity=.5] (top) -- (bottom);
\node [rectangle,draw,red,fill=white] (T) at (1,1.35) {\tiny{$m$}};
\node [rectangle,draw,red,fill=white] (T) at (1,.65) {\tiny{$m$}};
\end{tikzpicture}$
&
$\begin{tikzpicture}
\coordinate (left) at (0.25,1);
\coordinate (right) at (1.75,1);
\coordinate (top) at (1,1.75);
\coordinate (bottom) at (1,0.25);
\node [label=left:$x_j$] (L) at (left) {};
\node [label=right:$ $] (R) at (right) {};
\draw [-] (L.east) -- (R.west);
\draw [-] (top) -- (bottom);
\draw [-,line width=1mm,red,opacity=.5] (top) -- (bottom);
\draw [-,line width=1mm,red,opacity=.5] (L) -- (R);
\node [rectangle,draw,red,fill=white] (T) at (1,1.35) {\tiny{$m$}};
\node [rectangle,draw,red,fill=white] (B) at (1,.65) {\tiny{$m$}};
\end{tikzpicture}$
&
$\begin{tikzpicture}
\coordinate (left) at (0.25,1);
\coordinate (right) at (1.75,1);
\coordinate (top) at (1,1.75);
\coordinate (bottom) at (1,0.25);
\node [label=left:$x_j$] (L) at (left) {};
\node [label=right:$ $] (R) at (right) {};
\node [] at (bottom) {};
\draw [-] (L.east) -- (R.west);
\draw [-] (top) -- (bottom);
\draw [-,line width=1mm,red,opacity=.5] (top) -- (bottom);
\draw [-,line width=1mm,red,opacity=.5] (1,1) -- (R);
\node [rectangle,draw,red,fill=white] (T) at (1,1.35) {\tiny{$m$}};
\node [rectangle,draw,red,fill=white] (B) at (1,.65) {\tiny{$m-1$}};
\end{tikzpicture}$
&
$\begin{tikzpicture}
\coordinate (left) at (0.25,1);
\coordinate (right) at (1.75,1);
\coordinate (top) at (1,1.75);
\coordinate (bottom) at (1,0.25);
\node [label=left:$x_j$] (L) at (left) {};
\node [label=right:$ $] (R) at (right) {};
\node []  at (bottom) {};
\node []  at (top) {};
\draw [-] (L.east) -- (R.west);
\draw [-] (top) -- (bottom);
\draw [-,line width=1mm,red,opacity=.5] (top) -- (bottom);
\draw [-,line width=1mm,red,opacity=.5] (L) -- (1,1);
\node [rectangle,draw,red,fill=white] (T) at (1,1.35) {\tiny{$m$}};
\node [rectangle,draw,red,fill=white] (B) at (1,.65) {\tiny{$m+1$}};
\end{tikzpicture}$
\\
%&&&&&\\ 
\hline
$1$ & $x_j\cdot \displaystyle\prod_{k \in \mathcal{A}_{j,i}} \frac{x_j-qx_k}{x_j-x_k}$ & $0$ & $x_j$\\ \hline
\end{tabular}
\caption{Weights for the column operators $\widetilde{\mathbf{C}}_i$}\label{fig:nonlocal.weights}
\end{figure}

Here the Boltzmann weight in the second column of the table depends on a set $\mathcal{A}_{j,i}$, defined as the set of indices $k \ne j$ in $\mathcal{I}$ such that the vertex $v_{k,i}$ in row $k$ and column $i$ is of the type appearing in the rightmost column of Figure \ref{fig:nonlocal.weights} (i.e., those vertices where a path enters from the left and moves downward). This set is the only way in which the above Boltzmann weights fail to be local.  The interested reader can find an example of the weight associated to a particular configuration in rank 3 in Figure \ref{fig:sample.computation.of.bra.ket}.

Thus we've reduced the calculation of the partition function for general rank to a pair of computations according to~
(\ref{eq:twoparts}), a computation of $\bra{\mathbf{K}} \mathbf{F}^{-1}$ (graphically represented as a set of braids attached to bend vertices) and $\widetilde{\mathbf{C}}_0 \cdots \widetilde{\mathbf{C}}_{\lambda_r} \ket{\varnothing}$ (graphically represented as a partition function consisting entirely of tetravalent vertices in the rectangular lattice). As we will see in subsequent sections, explicit formulas for the former are very dependent on the choice of weights for the bend vertices in a given rank, while the latter may be evaluated in general. In particular, we will see that $\bra{\mathbf{K}} \mathbf{F}^{-1}$ consists of partition functions whose right-hand boundary has paths exiting in each row of the top half of the model and no paths in rows in the bottom half of the model; for ease of notation, we will use $\hat{u}$ to denote the vector with this occupancy, so that for example in rank three we have $$\hat{u} = \stasisthree.$$ Since the states in $\bra{\mathbf{K}} \mathbf{F}^{-1}$ must pair with those in $\widetilde{\mathbf{C}}_0 \cdots \widetilde{\mathbf{C}}_{\lambda_r} \ket{\varnothing}$, this greatly reduces the number of cases we must consider in the latter ket (i.e., the possible left-hand boundaries resulting from this ket).  The right-hand side of Figure~\ref{fig:sample.computation.of.bra.ket} illustrates one such non-zero configuration in this process.

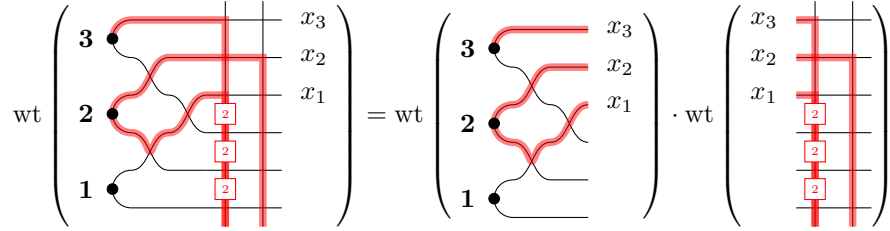
\begin{figure}[ht!]
$$\textrm{wt}\left(
\begin{tikzpicture}[scale=.5,baseline={([yshift=-\the\dimexpr\fontdimen22\textfont2\relax]current bounding box.center)}]
\node [label=left:$\mathbf{3}$] at (-.5,1.5) {};
\node [label=left:$\mathbf{2}$] at (-.5,-.5) {};
\node [label=left:$\mathbf{1}$] at (-.5,-2.5) {};
\node [label=right:$x_3$] at (4,2) {};
\node [label=right:$x_2$] at (4,1) {};
\node [label=right:$x_1$] at (4,0) {};
%\node [label=right:$\overline{x}_3$] at (4,-1) {};
%\node [label=right:$\overline{x}_2$] at (4,-2) {};
%\node [label=right:$\overline{x}_1$] at (4,-3) {};
\draw [-]
	(0,2) arc (90:270:.5);
\draw [-]
	(0,0) arc (90:270:.5);
\draw [-]
	(0,-2) arc (90:270:.5);
\draw [-] (2,2) -- (4,2);
\draw [-] (2,1) -- (4,1);
\draw [-] (2,0) -- (4,0);
\draw [-] (2,-1) -- (4,-1);
\draw [-] (2,-2) -- (4,-2);
\draw [-] (2,-3) -- (4,-3);

\draw [-] (2.5,2.5) -- (2.5,-3.5);
\draw [-] (3.5,2.5) -- (3.5,-3.5);

\draw [-] (0,2) -- (2,2);
\draw [-] (0,1) .. controls (.5,1) and (.5,0) .. (1,0)
            .. controls (1.5,0) and (1.5,-1) .. (2,-1);
\draw [-] (0,0) .. controls (.5,0) and (.5,1) .. (1,1) -- (2,1);
\draw [-] (0,-1) .. controls (.5,-1) and (.5,-2) .. (1,-2) -- (2,-2);
\draw [-] (0,-2) .. controls (0.5,-2) and (0.5,-1) .. (1,-1)
            .. controls (1.5,-1) and (1.5,0) .. (2,0);
\draw [-] (0,-3) -- (2,-3);

\draw [-,line width=1mm,red,opacity=.5,rounded corners=0pt]
    (-.5,1.5) arc (180:90:.5) 
    -- (2.5,2)
    -- (2.5,-3.5);

\draw [-,line width=1mm,red,opacity=.5,rounded corners=0pt]
    (-.5,-.5) arc (180:90:.5) 
    .. controls (.5,0) and (.5,1) .. (1,1) 
    -- (3.5,1)
    -- (3.5,-3.5);

\draw [-,line width=1mm,red,opacity=.5,rounded corners=0pt]
    (-.5,-.5) arc (180:270:.5) % at (0,-1)
    .. controls (.4,-1) and (.4,-1.5) .. (0.5,-1.5)
    .. controls (0.6,-1.5) and (0.6,-1) .. (1,-1)
    .. controls (1.5,-1) and (1.5,0) .. (2,0)
    -- (2.5,0)
    -- (2.5,-3.5);

\node [rectangle,draw,red,fill=white,scale=.75] at (2.5,-.5) {\tiny{$2$}};
\node [rectangle,draw,red,fill=white,scale=.75] at (2.5,-1.5) {\tiny{$2$}};
\node [rectangle,draw,red,fill=white,scale=.75] at (2.5,-2.5) {\tiny{$2$}};

\filldraw[black] (-.5,1.5) circle (4pt);
\filldraw[black] (-.5,-.5) circle (4pt);
\filldraw[black] (-.5,-2.5) circle (4pt);

%    (3,2) -- (2,2)
%    .. controls (1.5,2) and (1.5,1) .. (1,1)
%    .. controls (.6,1) and (.6,.5) .. (0.5,0.5)
%    .. controls (.4,.5) and (.4,0) .. (0,0)
%    arc (90:180:.5);

%\draw [rounded corners, dotted] (0,2.5) -- (0,-3.5) -- (2,-3.5) -- (2,2.5) -- cycle;
\end{tikzpicture}\right) = 
\textrm{wt}\left(
\begin{tikzpicture}[scale=.5,baseline={([yshift=-\the\dimexpr\fontdimen22\textfont2\relax]current bounding box.center)}]
\node [label=left:$\mathbf{3}$] at (-.5,1.5) {};
\node [label=left:$\mathbf{2}$] at (-.5,-.5) {};
\node [label=left:$\mathbf{1}$] at (-.5,-2.5) {};
\node [label=right:$x_3$] at (2,2) {};
%\node [label=left:$\overline{x}_3$] at (-.5,1) {};
\node [label=right:$x_2$] at (2,1) {};
%\node [label=left:$\overline{x}_2$] at (-.5,-1) {};
\node [label=right:$x_1$] at (2,0) {};
%\node [label=left:$\overline{x}_1$] at (-.5,-3) {};
%\node [label=right:$x_3$] at (4,2) {};
%\node [label=right:$x_2$] at (4,1) {};
%\node [label=right:$x_1$] at (4,0) {};

\draw [-]
	(0,2) arc (90:270:.5);
\draw [-]
	(0,0) arc (90:270:.5);
\draw [-]
	(0,-2) arc (90:270:.5);

%\draw [-] (2,2) -- (4,2);
%\draw [-] (2,1) -- (4,1);
%\draw [-] (2,0) -- (4,0);
%\draw [-] (2,-1) -- (4,-1);
%\draw [-] (2,-2) -- (4,-2);
%\draw [-] (2,-3) -- (4,-3);

%\draw [-] (2.5,2.5) -- (2.5,-3.5);
%\draw [-] (3.5,2.5) -- (3.5,-3.5);

\draw [-] (0,2) -- (2,2);
\draw [-] (0,1) .. controls (.5,1) and (.5,0) .. (1,0)
            .. controls (1.5,0) and (1.5,-1) .. (2,-1);
\draw [-] (0,0) .. controls (.5,0) and (.5,1) .. (1,1) -- (2,1);
\draw [-] (0,-1) .. controls (.5,-1) and (.5,-2) .. (1,-2) -- (2,-2);
\draw [-] (0,-2) .. controls (0.5,-2) and (0.5,-1) .. (1,-1)
            .. controls (1.5,-1) and (1.5,0) .. (2,0);
\draw [-] (0,-3) -- (2,-3);

\draw [-,line width=1mm,red,opacity=.5,rounded corners=0pt]
    (-.5,1.5) arc (180:90:.5) 
    -- (2,2);

\draw [-,line width=1mm,red,opacity=.5,rounded corners=0pt]
    (-.5,-.5) arc (180:90:.5) 
    .. controls (.5,0) and (.5,1) .. (1,1) 
    -- (2,1);

\draw [-,line width=1mm,red,opacity=.5,rounded corners=0pt]
    (-.5,-.5) arc (180:270:.5) % at (0,-1)
    .. controls (.4,-1) and (.4,-1.5) .. (0.5,-1.5)
    .. controls (0.6,-1.5) and (0.6,-1) .. (1,-1)
    .. controls (1.5,-1) and (1.5,0) .. (2,0);

%\node [rectangle,draw,red,fill=white,scale=.75] at (2.5,-.5) {\tiny{$2$}};
%\node [rectangle,draw,red,fill=white,scale=.75] at (2.5,-1.5) {\tiny{$2$}};
%\node [rectangle,draw,red,fill=white,scale=.75] at (2.5,-2.5) {\tiny{$2$}};

\filldraw[black] (-.5,1.5) circle (4pt);
\filldraw[black] (-.5,-.5) circle (4pt);
\filldraw[black] (-.5,-2.5) circle (4pt);

%    (3,2) -- (2,2)
%    .. controls (1.5,2) and (1.5,1) .. (1,1)
%    .. controls (.6,1) and (.6,.5) .. (0.5,0.5)
%    .. controls (.4,.5) and (.4,0) .. (0,0)
%    arc (90:180:.5);

%\draw [rounded corners, dotted] (0,2.5) -- (0,-3.5) -- (2,-3.5) -- (2,2.5) -- cycle;
\end{tikzpicture}\right) \cdot 
\textrm{wt}\left(
\begin{tikzpicture}[scale=.5,baseline={([yshift=-\the\dimexpr\fontdimen22\textfont2\relax]current bounding box.center)}]
\node [label=left:$x_3$] at (2,2) {};
\node [label=left:$x_2$] at (2,1) {};
\node [label=left:$x_1$] at (2,0) {};

\draw [-] (2,2) -- (4,2);
\draw [-] (2,1) -- (4,1);
\draw [-] (2,0) -- (4,0);
\draw [-] (2,-1) -- (4,-1);
\draw [-] (2,-2) -- (4,-2);
\draw [-] (2,-3) -- (4,-3);

\draw [-] (2.5,2.5) -- (2.5,-3.5);
\draw [-] (3.5,2.5) -- (3.5,-3.5);

\draw [-,line width=1mm,red,opacity=.5,rounded corners=0pt]
    (2,2) -- (2.5,2) -- (2.5,-3.5);

\draw [-,line width=1mm,red,opacity=.5,rounded corners=0pt]
    (2,1) -- (3.5,1) -- (3.5,-3.5);

\draw [-,line width=1mm,red,opacity=.5,rounded corners=0pt]
    (2,0) -- (2.5,0) -- (2.5,-3.5);

\node [rectangle,draw,red,fill=white,scale=.75] at (2.5,-.5) {\tiny{$2$}};
\node [rectangle,draw,red,fill=white,scale=.75] at (2.5,-1.5) {\tiny{$2$}};
\node [rectangle,draw,red,fill=white,scale=.75] at (2.5,-2.5) {\tiny{$2$}};

%    (3,2) -- (2,2)
%    .. controls (1.5,2) and (1.5,1) .. (1,1)
%    .. controls (.6,1) and (.6,.5) .. (0.5,0.5)
%    .. controls (.4,.5) and (.4,0) .. (0,0)
%    arc (90:180:.5);

%\draw [rounded corners, dotted] (0,2.5) -- (0,-3.5) -- (2,-3.5) -- (2,2.5) -- cycle;
\end{tikzpicture}\right)$$
\caption{The weight of a state associated to $\lambda = (0,0,1)$ is broken into two components.  The term on the left corresponds to a portion of $\bra{\mathbf{K}}{\mathbf{F}^{-1}}$, and the term on the right corresponds to a portion of $\widetilde{\mathbf{C}}_0\widetilde{\mathbf{C}}_1\ket{\varnothing}$. Using the twisting weights from Figure \ref{fig:R.weights} and the generic bend weights from \ref{fig:gen.bends}, the contribution from the left term is $\bra{\hat{u}} C_3\cdot  D_2\cdot \frac{x_2-\overline{x}_3}{x_2-q\overline{x}_3}\cdot \frac{x_1-\overline{x}_3}{x_1-q\overline{x}_3} \cdot \frac{(1-q)\overline{x}_2}{x_1-q\overline{x}_2}$. Using the column weights from Figure \ref{fig:nonlocal.weights}, the contribution from the term on the right from $\widetilde{\mathbf{C}}_0 \cdot \widetilde{\mathbf{C}}_1 \ket{\varnothing}$ is $\protect \left( x_3 \cdot x_2 \frac{x_2-q\overline{x}_2}{x_2-\overline{x}_2} \cdot \overline{x}_2\right) \cdot \left(x_2\right) \ket{\hat{u}}$.}
\label{fig:sample.computation.of.bra.ket}
\end{figure}

As it will be useful in subsequent sections, we will now compute the term $\widetilde{\mathbf{C}}_0 \cdots \widetilde{\mathbf{C}}_{\lambda_r} \ket{\varnothing}$.

\begin{theorem}\label{thm:twisted.grid} Let $\lambda = (\lambda_1 \geqslant \cdots \geqslant \lambda_r)$ be any partition. Then
\[ \bra{\hat{u}} \widetilde{\mathbf{C}}_0 \cdots \widetilde{\mathbf{C}}_{\lambda_r} \ket{\varnothing} = \sum_{\sigma \in S_r/S_r^{\lambda}} y_{\sigma(1)}^{\lambda_1+1}\cdots y_{\sigma(r)}^{\lambda_r+1} \prod_{i,j: \lambda_j>\lambda_i} \frac{y_{\sigma(j)}-qy_{\sigma(i)}}{y_{\sigma(j)}-y_{\sigma(i)}},\]
where $S_r^{\lambda}$ is the subgroup of $S_r$ that stabilizes $\lambda$ and $y_1,\hdots, y_r$ denote the spectral parameters for the top $r$ rows of the lattice (in descending order).
\end{theorem}
\begin{proof}

To begin, consider the path of a given particle in a given state of $\bra{\hat{u}} \widetilde{\mathbf{C}}_0 \cdots \widetilde{\mathbf{C}}_{\lambda_r} \ket{\varnothing}$. This particle emanates from the bottom boundary along column $\lambda_i$, and, because the Boltzmann weight in the third column of Figure \ref{fig:nonlocal.weights} is 0, makes a single turn to exit via the left boundary along row $j$, where $j$ is one of the top $r$ rows of the lattice. Since any of the $r$ particles can exit via the left boundary of any of the top $r$ rows, we can see that there is a bijection between states of $\bra{\hat{u}} \widetilde{\mathbf{C}}_0 \cdots \widetilde{\mathbf{C}}_{\lambda_r} \ket{\varnothing}$ and elements of $S_r/S_r^{\lambda}$. As an example, consider the state below: 
\[
\textrm{wt}\left(
\begin{tikzpicture}[scale=.25,baseline={([yshift=-\the\dimexpr\fontdimen22\textfont2\relax]current bounding box.center)}, every node/.style={transform shape}]
\node [label=left:\Huge{$y_r$}] at (-1,2) {};
\node [label=left:\Huge{$y_{r-1}$}] at (-1,1) {};
\node [] at (-1,0) {};
\node [label=left:\Huge{$y_2$}] at (-1,-1) {};
\node [label=left:\Huge{$y_1$}] at (-1,-2) {};
\node [] at (0,-3) {};
\node [] at (0,-4) {};
\node [] at (-1,-5) {};
\node [] at (-1,-6) {};
\node [] at (-1,-7) {};
\node [label=below:\Huge{$\lambda_r$}] at (0.5,-7.5) {};
\node [label=below:\Huge{$\lambda_{r-1}$}] at (1.5,-7.5) {};
\node [label=below:\Huge{$\lambda_2$}] at (4.5,-7.5) {};
\node [label=below:\Huge{$\lambda_1$}] at (7.5,-7.5) {};
\draw [-] (-1,2) -- (3,2);
\draw [densely dotted] (3,2) -- (4,2);
\draw [-] (4,2) -- (8,2);
\draw [-] (-1,1) -- (3,1);
\draw [densely dotted] (3,1) -- (4,1);
\draw [-] (4,1) -- (8,1);
\draw [-] (-1,-1) -- (3,-1);
\draw [densely dotted] (3,-1) -- (4,-1);
\draw [-] (4,-1) -- (8,-1);
\draw [-] (-1,-2) -- (3,-2);
\draw [densely dotted] (3,-2) -- (4,-2);
\draw [-] (4,-2) -- (8,-2);
\draw [-] (-1,-3) -- (3,-3);
\draw [densely dotted] (3,-3) -- (4,-3);
\draw [-] (4,-3) -- (8,-3);
\draw [-] (-1,-4) -- (3,-4);
\draw [densely dotted] (3,-4) -- (4,-4);
\draw [-] (4,-4) -- (8,-4);
\draw [-] (-1,-6) -- (3,-6);
\draw [densely dotted] (3,-6) -- (4,-6);
\draw [-] (4,-6) -- (8,-6);
\draw [-] (-1,-7) -- (3,-7);
\draw [densely dotted] (3,-7) -- (4,-7);
\draw [-] (4,-7) -- (8,-7);

\draw [-] (-.5,2.5) -- (-.5,.5);
\draw [densely dotted] (-.5,.5) -- (-.5,-.5);
\draw [-] (-.5,-.5) -- (-.5,-4.5);
\draw [densely dotted] (-.5,-4.5) -- (-.5,-5.5);
\draw [-] (-.5,-5.5) -- (-.5,-7.5);
\draw [-] (0.5,2.5) -- (0.5,.5);
\draw [densely dotted] (0.5,.5) -- (0.5,-.5);
\draw [-] (0.5,-.5) -- (0.5,-4.5);
\draw [densely dotted] (0.5,-4.5) -- (0.5,-5.5);
\draw [-] (0.5,-5.5) -- (0.5,-7.5);
\draw [-] (1.5,2.5) -- (1.5,.5);
\draw [densely dotted] (1.5,.5) -- (1.5,-.5);
\draw [-] (1.5,-.5) -- (1.5,-4.5);
\draw [densely dotted] (1.5,-4.5) -- (1.5,-5.5);
\draw [-] (1.5,-5.5) -- (1.5,-7.5);
\draw [-] (2.5,2.5) -- (2.5,.5);
\draw [densely dotted] (2.5,.5) -- (2.5,-.5);
\draw [-] (2.5,-.5) -- (2.5,-4.5);
\draw [densely dotted] (2.5,-4.5) -- (2.5,-5.5);
\draw [-] (2.5,-5.5) -- (2.5,-7.5);
\draw [-] (4.5,2.5) -- (4.5,.5);
\draw [densely dotted] (4.5,.5) -- (4.5,-.5);
\draw [-] (4.5,-.5) -- (4.5,-4.5);
\draw [densely dotted] (4.5,-4.5) -- (4.5,-5.5);
\draw [-] (4.5,-5.5) -- (4.5,-7.5);
\draw [-] (5.5,2.5) -- (5.5,.5);
\draw [densely dotted] (5.5,.5) -- (5.5,-.5);
\draw [-] (5.5,-.5) -- (5.5,-4.5);
\draw [densely dotted] (5.5,-4.5) -- (5.5,-5.5);
\draw [-] (5.5,-5.5) -- (5.5,-7.5);
\draw [-] (6.5,2.5) -- (6.5,.5);
\draw [densely dotted] (6.5,.5) -- (6.5,-.5);
\draw [-] (6.5,-.5) -- (6.5,-4.5);
\draw [densely dotted] (6.5,-4.5) -- (6.5,-5.5);
\draw [-] (6.5,-5.5) -- (6.5,-7.5);
\draw [-] (7.5,2.5) -- (7.5,.5);
\draw [densely dotted] (7.5,.5) -- (7.5,-.5);
\draw [-] (7.5,-.5) -- (7.5,-4.5);
\draw [densely dotted] (7.5,-4.5) -- (7.5,-5.5);
\draw [-] (7.5,-5.5) -- (7.5,-7.5);

\draw [-,line width=1mm,red,opacity=0.5]
    (-1,2) -- (.5,2) -- (.5,-7.5);
\draw [-,line width=1mm,red,opacity=0.5]
    (-1,1) -- (1.5,1) -- (1.5,-7.5);
\draw [-,line width=1mm,red,opacity=0.5]
    (-1,-1) -- (4.5,-1) -- (4.5,-7.5);
\draw [-,line width=1mm,red,opacity=0.5]
    (-1,-2) -- (7.5,-2) -- (7.5,-7.5);

\filldraw[white] (-1,2) circle (4pt);
\filldraw[white] (-1,1) circle (4pt);
\filldraw[white] (-1,-1) circle (4pt);
\filldraw[white] (-1,-2) circle (4pt);
\filldraw[red,opacity=.5] (-1,2) circle (4pt);
\filldraw[red,opacity=.5] (-1,1) circle (4pt);
\filldraw[red,opacity=.5] (-1,-1) circle (4pt);
\filldraw[red,opacity=.5] (-1,-2) circle (4pt);
\draw[fill=white] (-1,-3) circle (4pt);
\draw[fill=white] (-1,-4) circle (4pt);
\draw[fill=white] (-1,-6) circle (4pt);
\draw[fill=white] (-1,-7) circle (4pt);
\end{tikzpicture}\right)
= 
y_1^{\lambda_1+1}\cdots y_{r}^{\lambda_r+1} \prod_{i,j: \lambda_j>\lambda_i} \frac{y_j - qy_i}{y_j-y_i}.
\]

If we instead consider the state where the particle emanating from column $\lambda_i$ exits along row $\sigma(i)$ (rather than row $i$), for a given $\sigma \in S_r/S_r^{\lambda}$, then we see that the Boltzmann weight for this state is $y_{\sigma(1)}^{\lambda_1+1}\cdots y_{\sigma(r)}^{\lambda_r+1} \prod_{i,j: \lambda_j>\lambda_i} \frac{y_{\sigma(j)}-qy_{\sigma(i)}}{y_{\sigma(j)}-y_{\sigma(i)}}$. Summing over all possible states, we see that \[\bra{\hat{u}} \widetilde{\mathbf{C}}_0 \cdots \widetilde{\mathbf{C}}_{\lambda_r} \ket{\varnothing}= \sum_{\sigma \in S_r/S_r^{\lambda}} y_{\sigma(1)}^{\lambda_1+1}\cdots y_{\sigma(r)}^{\lambda_r+1} \prod_{i,j: \lambda_j>\lambda_i} \frac{y_{\sigma(j)}-qy_{\sigma(i)}}{y_{\sigma(j)}-y_{\sigma(i)}}. \qedhere \] 
\end{proof}

Following \cite{Wheeler-Zinn-Justin} (see also \cite{Venkateswaran}), we note that $\bra{\hat{u}} \widetilde{\mathbf{C}}_0 \cdots \widetilde{\mathbf{C}}_{\lambda_r} \ket{\varnothing}$ can be expressed as a sum over $S_r$: \[ \bra{\hat{u}} \widetilde{\mathbf{C}}_0 \cdots \widetilde{\mathbf{C}}_{\lambda_r} \ket{\varnothing} = \frac{1}{c_{\lambda}(q)}\sum_{\sigma \in S_r} y_{\sigma(1)}^{\lambda_1+1}\cdots y_{\sigma(r)}^{\lambda_r+1} \prod_{i,j:\lambda_j>\lambda_i} \frac{y_{\sigma(j)}-qy_{\sigma(i)}}{y_{\sigma(j)}-y_{\sigma(i)}},\]  where $c_{\lambda}(q) = \prod_{i\geq 1} \prod_{j=1}^{m_i(\lambda)} \frac{1-q^j}{1-q}$ and $m_i(\lambda)$ is the number of $\lambda_j$ equal to $i$ for each $i\geq 0$.

Before proceeding to evaluate the partition functions for various ranks under the scheme described above, we first include a result from \cite{Wheeler-Zinn-Justin} concerning the partition function that arises from the bend weight values $C_j = 1$, $B_j = -q$, and $A_j=D_j=0$.

\begin{theorem}[Wheeler-Zinn-Justin, Thm. 3 and Rmk. 5 in~\cite{Wheeler-Zinn-Justin}]\label{thm:WZJ}
Let $\lambda = (\lambda_1 \geqslant \lambda_2 \geqslant \cdots \geqslant \lambda_r)$ be a partition.  With rectangular Boltzmann weights as in Figure~\ref{fig:six.vertex.model} and bend weights $C_j = 1$, $B_j = -q$, and $A_j=D_j=0$ we have,
\[ \mathcal{Z}(\mathcal{B}_\lambda) = c_{\lambda}(q)^{-1}\prod_{i=1}^r x_i \cdot \prod_{\substack{\alpha \in \Phi^{+},\\ \alpha \textup{ short}}} (1-qx^{-\alpha}) \cdot \sum_{w \in W} w\left(\prod_{i=1}^r x_i^{\lambda_i}\cdot \prod_{\alpha \in \Phi^{+}} \frac{1}{1-x^{-\alpha}} \cdot \prod_{\substack{\alpha \in \Phi^{+},\\ \alpha \textup{ long}}} (1-qx^{-\alpha})\right),\]
with $c_{\lambda}(q)$ defined as above.
\end{theorem}

This result is the special case of (64) in \cite{Wheeler-Zinn-Justin} where two parameters that determine weights associated to certain ``bonus columns" are both set to $0$; for further context, see Section \ref{sec:intro}.

\section{Rank One Solvable Models}\label{sec:rank1}

A rank one lattice consists of a single pair of rows, connected by a bend. In these lattices, our family of boundary conditions permits only a single particle to enter along the bottom and exit at the bend. Since the caduceus relation is not relevant to this case, we will say that a model is ``rank one solvable" if it satisfies the Yang-Baxter equation and Fish relation. Lemma \ref{le:fish} gives the necessary conditions, reducing to four possible cases of solvable rank one models. We continue to use the abbreviated notation $B_j(x)$ for $B_j(x,\bar{x}),$ etc., in the following results.

\begin{theorem}\label{thm:spherical1}
Let the rank $r=1$ and let $\lambda = (\lambda_1)$ be a one-part partition. If $\mathcal{B}_{\lambda}$ is a solvable rank one model, then 
\begin{enumerate}
\item[(1a)] if $B_1(x)=C_1(x)$, we have $\displaystyle \mathcal{Z}(\mathcal{B}_\lambda) = C_1(x_1) \sum_{w \in W_0} w\left(x_1^{\lambda_1+1}\cdot \frac{1-qx^{-\alpha_1}}{1-x^{-\alpha_1}}\right)$;
\item[(1b)] if $B_1(x) = -qC_1(x)$,  we have $\displaystyle \mathcal{Z}(\mathcal{B}_\lambda) = C_1 \frac{(1-qx^{-\alpha_1})}{1-x^{-\alpha_1}}\cdot 
\sum_{w \in W_0} (-1)^{\ell(w)}w\left(x_1^{\lambda_1+1}\right)$;
\item[(2a)] if $B_1(x) = -x^2C_1(x)$,  we have $\displaystyle \mathcal{Z}(\mathcal{B}_\lambda) = C_1x_1 \frac{(1-qx^{-\alpha_1})}{1-x^{-\alpha_1}}\cdot 
\sum_{w \in W_0} (-1)^{\ell(w)} w\left(x_1^{\lambda_1}\right)$; and
\item[(2b)] if $B_1(x)=qx^2C_1(x)$, we have  $\displaystyle \mathcal{Z}(\mathcal{B}_\lambda) = x_1 C_1(x_1) \sum_{w \in W_0} w\left(x_1^{\lambda_1}\cdot \frac{1-qx^{-\alpha_1}}{1-x^{-\alpha_1}}\right)$.
\end{enumerate}
Here $B_1(x)$ and $C_1(x)$ refer to bend weights as in Figure \ref{fig:gen.bends}. Moreover, $W_0$ denotes the Weyl group in rank one, so two elements with $\ell(w)$ the length.
\end{theorem}

This may seem like a grandiose way of expressing the sum of two terms, but we mean to suggest expressions that will reappear in higher rank in subsequent sections. In particular, the reader can see that rank one solvable models with u-turn bend split into two families according to the constant of proportionality $F$ appearing in the Fish relation (Lemma~\ref{le:fish}). Members of the family all give, up to a simple explicit factor, one of the two orthogonal polynomials of Macdonald associated to characters of this two-element $W_0$ described in the introduction. Cases (1a) and (2b) have $F=1$ and give rise to the symmetric Hall-Littlewood polynomials and cases (1b) and (2a) have $F=(x^2-q)/(1-qx^2)$ and give the $q$-antisymmetric deformations of Weyl's character formula.\footnote{We say ``$q$-antisymmetric'' as setting $q=1$ gives the usual antisymmetric alternator. This is precisely the family of orthogonal polynomials associated to the sign character of $W_0$ by Macdonald. As written, case (1b) exactly matches the Casselman-Shalika formula for a $p$-adic Whittaker function and recovers the Weyl character formula at $q=0$.} The proof requires the following simple lemma according to the values of $F$, and written in the language of our general approach detailed in the previous section.

\begin{lemma}\label{le:K_sph1} In rank one, if the Fish constant of proportionality is $F=1$ (i.e., in cases (1a) and (2b) of Lemma \ref{le:fish}), then
\begin{equation} \label{eq:boundary1} \bra{\mathbf{K}}\mathbf{F}^{-1} = \left[ C_1(x_1) \frac{x_1-q\overline{x}_1}{x_1-\overline{x}_1} \right] \cdot (\bra{1}_1 \otimes \bra{-1}_{\overline{1}}) + \left[ C_1(x_1) \frac{\overline{x}_1-qx_1}{\overline{x}_1-x_1} \cdot \begin{cases} 1 & \textrm{if $B_1(x)=C_1(x)$} \\ x_1^2 & \textrm{if $B_1(x)=q x^2 C_1(x)$}  \end{cases} \right]\cdot (\bra{-1}_1 \otimes \bra{1}_{\overline{1}}),\end{equation}
If instead $F =(x^2-q)/(1-qx^2)$ (i.e., in cases (1b) and (2a) of Lemma \ref{le:fish}), then
\begin{equation} \label{eq:boundary2} \bra{\mathbf{K}}\mathbf{F}^{-1} = \left[ C_1(x_1) \frac{x_1-q\overline{x}_1}{x_1 - \overline{x}_1} \right] \cdot (\bra{1}_1 \otimes \bra{-1}_{\overline{1}}) + \left[ C_1(x_1) \frac{x_1-q\overline{x}_1}{x_1-\overline{x}_1} \cdot \begin{cases} -1 & \textrm{if $B_1(x)=-qC_1(x)$} \\ -x_1^2 & \textrm{if $B_1(x)=-x^2 C_1(x)$}  \end{cases} \right]\cdot (\bra{-1}_1 \otimes \bra{1}_{\overline{1}}).\end{equation}

Here \[\bra{1}_1 \otimes \bra{-1}_{\overline{1}} = \begin{tikzpicture}[scale=0.5, baseline=-.5ex]
\node [label=above:$ $] at (0,1) {};
\node [label=below:$ $] at (0,-1) {};
\node [label=left:$\mathbf{1}$] at (-1,0) {};
\node [label=center:$x_1$] at (.5,1) {};
\node [label=center:$\overline{x}_1$] at (.5,-1) {};
\draw [-]
	(0,1) arc (90:270:1);
\draw [-,line width=1mm,red,opacity=0.5]
	(0,1) arc (90:180:1);
\filldraw[black] (-1,0) circle (4pt);
\end{tikzpicture} \text{ and }
\bra{-1}_1 \otimes \bra{1}_{\overline{1}} = \begin{tikzpicture}[scale=0.5, baseline=-.5ex]
\node [label=above:$ $] at (0,1) {};
\node [label=below:$ $] at (0,-1) {};
\node [label=left:$\mathbf{1}$] at (-1,0) {};
\node [label=center:$x_1$] at (.5,1) {};
\node [label=center:$\overline{x}_1$] at (.5,-1) {};
\draw [-]
	(0,1) arc (90:270:1);
\draw [-,line width=1mm,red,opacity=0.5]
	(-1,0) arc (180:270:1);
\filldraw[black] (-1,0) circle (4pt);
\end{tikzpicture}. \]

%\bra{1}_k = particle, \bra{-1}_k = no particle

\end{lemma}
%{\color{red}It seems odd to me that we're being as general as possible in our hypotheses by saying $B_j=C_j$, but then in the proof we assume $C_j=1$.  Perhaps we should just keep the $C_1$ around?}

\begin{proof}
We'll prove this result by verifying that both components of $\bra{\mathbf{K}}\mathbf{F}^{-1}$ have the form given on the right of \eqref{eq:boundary1}. Suppose that $j_{\bar{1}},j_1\in \{0,1\}$, where 1 denotes the presence of a particle at a vertex, and 0 denotes the lack of a particle. In the case $j_1 = 1$ and $j_{\bar{1}} = 0$, by \eqref{eq:fstarform} and \eqref{eq:Delta}, we have 
\begin{align*}
(\bra{\mathbf{K}}\mathbf{F}^{-1})^{j_1j_{\bar{1}}} &= \left(\bra{\mathbf{K}} \mathbf{R}_{1\bar{1}}^{1\bar{1}}\cdot \Delta_{\bar{1}1}^{-1}(x_{\bar{1}},x_1)\right)^{j_1j_{\bar{1}}} = (\bra{\mathbf{K}} \mathbf{R}_{1\bar{1}}^{1\bar{1}})^{j_1j_{\bar{1}}} \cdot b_{j_{\bar{1}},j_1}(\overline{x}_1,x_1)
= (\bra{\mathbf{K}} \mathbf{R}_{1\bar{1}}^{1\bar{1}})^{j_1j_{\bar{1}}} \cdot \frac{x_{1}-q\overline{x}_{1}}{x_{1}-\overline{x}_{1}}.
\end{align*}
Since \[(\bra{\mathbf{K}}\mathbf{R}_{1\bar{1}}^{1\bar{1}})^{j_1j_{\bar{1}}} = 
\textrm{wt}\left(\begin{tikzpicture}[scale=0.5,baseline={([yshift=-\the\dimexpr\fontdimen22\textfont2\relax]current bounding box.center)}]
\node [label=above:$ $] at (0,1) {};
\node [label=below:$ $] at (0,-1) {};
\node [label=left:$\mathbf{1}$] at (-1,0) {};
\node [label=center:$x_1$] at (.5,1) {};
\node [label=center:$\overline{x}_1$] at (.5,-1) {};
\draw [-]
	(0,1) arc (90:270:1);
\draw [-,line width=1mm,red,opacity=0.5]
	(0,1) arc (90:180:1);
\filldraw[black] (-1,0) circle (4pt);
\end{tikzpicture}\right)=C_1(x_1),\] we have $(\bra{\mathbf{K}}\mathbf{F}^{-1})^{j_1j_{\bar{1}}} = C_1(x_1) \displaystyle\frac{x_1-q\overline{x}_1}{x_1-\overline{x}_1}.$

On the other hand, when $j_1 = 0$ and $j_{\bar{1}}=1$, we have 
\begin{align*} 
(\bra{\mathbf{K}}\mathbf{R}_{1\bar{1}}^{\bar{1}1})^{j_1j_{\bar{1}}} & = 
\mathcal{Z}\left(
\begin{tikzpicture}[scale=0.5,baseline={([yshift=-\the\dimexpr\fontdimen22\textfont2\relax]current bounding box.center)}]
\node [label=above:$ $] at (0,1) {};
\node [label=below:$ $] at (0,-1) {};
\node [label=left:$\mathbf{1}$] at (-1,0) {};
\node [label=right:$\overline{x}_1$] at (2,1) {};
\node [label=right:$x_1$] at (2,-1) {};
\draw [-]
	(0,1) arc (90:270:1);
\draw [-]
    (0,1) .. controls (1,1) and (1,-1) .. (2,-1);
\draw [-]
    (0,-1) .. controls (1,-1) and (1,1) .. (2,1);
%\draw [-,line width=1mm,red,opacity=0.5]
%	(0,1) arc (90:270:1);
%\draw [-,line width=1mm,red,opacity=0.5]
%    (0,1) .. controls (1,1) and (1,-1) .. (2,-1);
%\draw [-,line width=1mm,red,opacity=0.5]
%    (0,-1) .. controls (1,-1) and (1,1) .. (2,1);
\filldraw[black] (-1,0) circle (4pt);
\filldraw[white] (2,1) circle (4pt);
\draw[fill=white] (2,-1) circle (4pt);
\filldraw[red,opacity=.5] (2,1) circle (4pt);
\end{tikzpicture}
\right)
= 
\textrm{wt}
\left(
\begin{tikzpicture}[scale=0.5,baseline={([yshift=-\the\dimexpr\fontdimen22\textfont2\relax]current bounding box.center)}]
\node [label=above:$ $] at (0,1) {};
\node [label=below:$ $] at (0,-1) {};
\node [label=left:$\mathbf{1}$] at (-1,0) {};
\node [label=right:$\overline{x}_1$] at (2,1) {};
\node [label=right:$x_1$] at (2,-1) {};
\draw [-]
	(0,1) arc (90:270:1);
\draw [-]
    (0,1) .. controls (1,1) and (1,-1) .. (2,-1);
\draw [-]
    (0,-1) .. controls (1,-1) and (1,1) .. (2,1);
\draw [-,line width=1mm,red,opacity=0.5]
	(-1,0) arc (180:270:1);
%\draw [-,line width=1mm,red,opacity=0.5]
%    (0,1) .. controls (1,1) and (1,-1) .. (2,-1);
\draw [-,line width=1mm,red,opacity=0.5]
    (0,-1) .. controls (1,-1) and (1,1) .. (2,1);
\filldraw[black] (-1,0) circle (4pt);
%\filldraw[white] (2,1) circle (4pt);
%\filldraw[white] (2,-1) circle (4pt);
%\filldraw[red,opacity=.5] (2,1) circle (4pt);
\end{tikzpicture}
\right)
+
\textrm{wt}
\left(
\begin{tikzpicture}[scale=0.5,baseline={([yshift=-\the\dimexpr\fontdimen22\textfont2\relax]current bounding box.center)}]
\node [label=above:$ $] at (0,1) {};
\node [label=below:$ $] at (0,-1) {};
\node [label=left:$\mathbf{1}$] at (-1,0) {};
\node [label=right:$\overline{x}_1$] at (2,1) {};
\node [label=right:$x_1$] at (2,-1) {};
\draw [-]
	(0,1) arc (90:270:1);
\draw [-]
    (0,1) .. controls (1,1) and (1,-1) .. (2,-1);
\draw [-]
    (0,-1) .. controls (1,-1) and (1,1) .. (2,1);
\draw [-,line width=1mm,red,opacity=0.5]
	(-1,0) arc (180:90:1);
\draw [-,line width=1mm,red,opacity=0.5]
    (0,1) .. controls (0.8,1) and (0.8,0) .. (1,0)
    .. controls (1.2,0) and (1.2,1) .. (2,1);
%\draw [-,line width=1mm,red,opacity=0.5]
%    (0,-1) .. controls (1,-1) and (1,1) .. (2,1);
\filldraw[black] (-1,0) circle (4pt);
%\filldraw[white] (2,1) circle (4pt);
%\filldraw[white] (2,-1) circle (4pt);
%\filldraw[red,opacity=.5] (2,1) circle (4pt);
\end{tikzpicture}
\right) \\
& = C_1(x_1) \begin{cases} 1 & \textrm{if $B_1(x)=C_1(x)$,} \\ x_1^2 & \textrm{if $B_1(x) = q x^2 C_1(x)$,} \end{cases}
\end{align*}
and hence evaluating $(\bra{\mathbf{K}}\mathbf{F}^{-1})^{j_1j_{\bar{1}}} = (\bra{\mathbf{K}} \mathbf{R}_{1\bar{1}}^{\bar{1}1})^{j_1j_{\bar{1}}} \cdot b_{j_1,j_{\bar{1}}}(\overline{x}_1,x_1)$ gives the desired result. We leave the cases with constant of proportionality $F=(x^2-q)/(1-qx^2)$ by similar computation to the reader.
\end{proof}

\begin{proof}[Proof of Theorem \ref{thm:spherical1}] We prove cases (1a) and (2b), noting that Case (1b) follows from results of Wheeler and Zinn-Justin (c.f.~ \cite{Wheeler-Zinn-Justin} Section 3.3 Remark 5). Case (2a) is similar and left to the reader.

By \eqref{mainthmbraket} and \eqref{eq:twoparts}, we have $Z(\mathcal{B}_{\lambda}) = \bra{\mathbf{K}} \mathbf{F}^{-1} \mathbf{\widetilde{C}_0} \cdots \mathbf{\widetilde{C}_{\lambda_r}} \ket{\varnothing}$. Note that the rank one case of Theorem \ref{thm:twisted.grid} gives $\bra{\stasisone} \widetilde{\mathbf{C}}_0 \cdots \widetilde{\mathbf{C}}_{\lambda_r} \ket{\varnothing} =  y_{1}^{\lambda_1+1}$, where $y_1$ represents the spectral parameter for the particle entering on the top row of the lattice. Combining this with Lemma \ref{le:K_sph1}, we see for example that when $B_1(x)=C_1(x)$, 
\begin{align*}
Z(\mathcal{B}_{\lambda}) &= \bra{\mathbf{K}} \mathbf{F}^{-1} \mathbf{\widetilde{C}_0} \cdots \mathbf{\widetilde{C}_{\lambda_r}} \ket{\varnothing} \\
&= C_1(x_1)\frac{x_1-q\overline{x}_1}{x_1-\overline{x}_1}\cdot x_1^{\lambda_1+1} + C_1(x_1)\frac{\overline{x}_1-qx_1}{\overline{x}_1-x_1}\cdot \overline{x}_1^{\lambda_1+1}. 
\end{align*}
The case of $B_1(x)=qx^2 C_1(x)$ follows by a similarly straightforward substitution.
\end{proof}

\section{Rank Two Solvable Models}\label{sec:rank2}

If we restrict ourselves to the rank two case, then the general caduceus relation has fewer cases. Indeed, the only case of the caduceus we need to treat is the case with two particles along the boundary, namely Lemma~\ref{le:caduceus}, owing to the fact that our family of boundary conditions in rank two permits only two particles to enter along the bottom and exit at the bends. Parroting our definition from the prior section for this special case, we record the following:
\begin{definition} A type $B/C$ lattice model with two pairs of rows is said to be {\bf rank two solvable} if it satisfies the Yang-Baxter equation, is uniform regime in the Fish relation (Lemma~\ref{le:fish}), and satisfies the two particle Caduceus relation of Lemma~\ref{le:caduceus}.
\end{definition}

% \begin{proposition}\label{prop:rank.two.solvable.possibilities} The only possible choices of bend 
% weights to complete the trigonometric six-vertex model to a {\bf rank two solvable} $B/C$ model are as % follows:
% \begin{enumerate}
%     \item\label{condition:rank.2.solvable.bends} $B_j = C_j$ for $j=1,2$ and $q^2 (A_1 D_2- B_1 B_2) - % (A_2 D_1 - B_1 B_2) = 0$, or 
%     \item $B_j = -q C_j$ and $A_j = D_j = 0$ for all $j$.
% \end{enumerate}
% \end{proposition}
% 
% \begin{proof}
% Because of the fish relation (Lemma \ref{le:fish}) we must have either $B_j = C_j$, or that 
% $B_j = -qC_j$ while $A_j = D_j = 0$.  The Caduceus relation for 2 particles 
% (Lemma  \ref{le:caduceus}) tells us that in the former case we must then 
% have $q^2 (A_1 D_2- B_1 B_2) - (A_2 D_1 - B_1 B_2) = 0$.
% \end{proof}

One can systematically treat each subcase from Lemma~\ref{le:fish} and the two particle Caduceus relation, and the resulting partition functions have uniform expressions within each subcase, much as we saw for rank one models in the prior section. For example, subcase (1b) of Lemma~\ref{le:fish} rules out double bend configurations (i.e., $A_j = D_j = 0$) and all such weights lead to variants of the Wheeler-Zinn-Justin model \cite{Wheeler-Zinn-Justin}. For this reason, we will not consider the partition functions in this case here. Instead we focus on Subcases (1a) and (2b) of Lemma~\ref{le:fish} where the Fish and Caduceus constants $F$ are equal to one, leading to partition functions symmetric under the action of the hyperoctahedral group. (In this result, recall the definition of $s_2$ from Equation (\ref{eq:hypteroctahedral.generators}).)

\begin{theorem}\label{thm:rank2.generic}
Let rank $r=2$, and let $\lambda = (\lambda_1 \geqslant \lambda_2)$ be a partition. Suppose our Boltzmann weights are from a solvable, type $B/C$ lattice model for the trigonometric six-vertex model with $R$-weights as in (\ref{fig:R.weights}). Suppose further that the uniform Fish relation constant of proportionality is one; that is, assume either uniform in regime 1(a) (so that $C_i=B_i$ for $i \in \{1,2\}$) or uniform in regime 2(b) (so that $C_i = m_i \bar x$ for $i \in \{1,2\}$). Then one has \begin{align*}
\mathcal{Z}(\mathcal{B}_\lambda) &= \frac{\kappa}{c_{\lambda}(q)}\sum_{w \in W} w\left(x_1^{\lambda_1+\delta_1}x_2^{\lambda_2+\delta_2}\cdot \prod_{\alpha \in \Phi^+} \frac{1-q\mathbf{x}^{-\alpha}}{1-\mathbf{x}^{-\alpha}}\right) 
\\&+\frac{A_1D_2}{c_{\lambda}(q)}\left[\sum_{w \in W} w\left(x_1^{\lambda_1+1}x_2^{\lambda_2+1}\cdot \frac{x_1-qx_2}{x_1-x_2}\cdot \frac{x_2-q\overline{x}_1}{x_2-\overline{x}_1} \cdot \frac{(1-q)\overline{x}_2}{x_1-\overline{x}_2} \cdot \frac{x_1-q\overline{x}_1}{x_1-\overline{x}_1} \cdot \frac{x_2-q\overline{x}_2}{x_2-\overline{x}_2}\right)\right.
\\&+ \left.\sum_{w \in W/\langle s_2 \rangle} w\left(x_1^{\lambda_1-\lambda_2}\cdot \frac{x_1-qx_2}{x_1-x_2}\cdot \frac{\overline{x}_1-qx_2}{\overline{x}_1-x_2} \cdot \frac{x_1-q\overline{x}_2}{x_1-\overline{x}_2} \cdot \frac{\overline{x}_1-q\overline{x}_2}{\overline{x}_1-\overline{x}_2} \cdot \frac{x_1-q\overline{x}_1}{x_1-\overline{x}_1}\right)\right],
%\\&= c_{\lambda}(q)^{-1}\sum_{w \in W} w\left(\frac{x_1-q\overline{x}_2}{x_1-\overline{x}_2}\cdot \frac{x_1-qx_2}{x_1-x_2}\cdot \frac{x_1-q\overline{x}_1}{x_1-\overline{x}_1}\cdot \frac{x_1^{\lambda_1+1}}{x_2-\overline{x}_1} \right.
%\\&\cdot \left.\left(x_2^{\lambda_2+1} (x_2-q\overline{x}_1)\cdot \frac{x_2-q\overline{x}_2}{x_2-\overline{x}_2}-\frac{1}{2}\overline{x}_1^{\lambda_2+1}(\overline{x}_1-qx_2)\cdot \frac{\overline{x}_1-q\overline{x}_2}{\overline{x}_1-\overline{x}_2}\right)\right)
\end{align*}
where we have written $c_{\lambda}(q) = \left\{\begin{array}{cc} 1 & \text{if $\lambda_1> \lambda_2$} \\ 1+q & \text{if $\lambda_1 = \lambda_2$,}\end{array}\right\}$, 
$\delta_i = \left\{\begin{array}{cc}1&\text{in regime 1(a)}\\0&\text{in regime 2(b),}\end{array}\right\}$ for $i \in \{1,2\}$, and 
$\kappa = \left\{\begin{array}{cc}C_1C_2&\text{in regime 1(a)}\\m_1m_2&\text{in regime 2(b).}\end{array}\right\}$
\end{theorem}

In order to prove Theorem \ref{thm:rank2.generic}, we will follow the approach outlined in the previous section: first we will compute $\bra{\mathbf{K}}\mathbf{F}^{-1}$ (in Lemma \ref{le:K.rank2.generic}), and then we will combine this with Theorem \ref{thm:twisted.grid} to complete the evaluation.

\begin{lemma}\label{le:K.rank2.generic} Let rank $r=2$, and suppose our Boltzmann weights are from a solvable, type $B/C$ lattice model for the trigonometric six-vertex model with $R$-weights as in (\ref{fig:R.weights}). Suppose further that the uniform Fish relation constant of proportionality is one. Then, with notation as in Section~\ref{sec:genmethods}, $\bra{\mathbf{K}}\mathbf{F}^{-1}$ is equal to

\begin{align} \label{eq:boundary}   \sum_{\{\epsilon_1,\epsilon_2\} \in \{\pm 1\}^2} & \frac{x_1^{\epsilon_1}-q\overline{x}_1^{\epsilon_1}}{x_1^{\epsilon_1}-\overline{x}_1^{\epsilon_1}} \cdot \frac{x_2^{\epsilon_2}-q\overline{x}_2^{\epsilon_2}}{x_2^{\epsilon_2}-\overline{x}_2^{\epsilon_2}} \cdot \frac{x_2^{\epsilon_2}-q\overline{x}_1^{\epsilon_1}}{x_2^{\epsilon_2}-\overline{x}_1^{\epsilon_1}} \cdot \nonumber \\& \left(C_1(x_1^{\epsilon_1})C_2(x_2^{\epsilon_2}) + A_1(x_1^{\epsilon_1})D_2(x_2^{\epsilon_2}) \cdot \frac{(1-q)\bar{x}_2^{\epsilon_2}}{x_1^{\epsilon_1}-\bar{x}_2^{\epsilon_2}} \right) \cdot (\bra{\epsilon_1}_1 \otimes \bra{-\epsilon_1}_{\overline{1}})\otimes (\bra{\epsilon_2}_2 \otimes \bra{-\epsilon_2}_{\overline{2}}),
\\&+ A_1(x_2)D_2(x_1)\prod_{\{\epsilon_1,\epsilon_2\}\in \{\pm 1\}^2} \frac{x_1^{\epsilon_1}-q\bar{x}_2^{\epsilon_2}}{x_1^{\epsilon_1}-\bar{x}_2^{\epsilon_2}}\cdot (\bra{\epsilon_1}_1 \otimes \bra{\epsilon_1}_{\overline{1}})\otimes (\bra{-\epsilon_2}_2 \otimes \bra{-\epsilon_2}_{\overline{2}})\nonumber
\\&+ A_1(x_1)D_2(x_2)\prod_{\{\epsilon_1,\epsilon_2\}\in \{\pm 1\}^2} \frac{x_2^{\epsilon_2}-q\bar{x}_1^{\epsilon_1}}{x_2^{\epsilon_2}-\bar{x}_1^{\epsilon_1}}\cdot (\bra{-\epsilon_1}_1 \otimes \bra{-\epsilon_1}_{\overline{1}})\otimes (\bra{\epsilon_2}_2 \otimes \bra{\epsilon_2}_{\overline{2}}),\nonumber
\end{align}
where \[
\bra{1}_i \otimes \bra{-1}_{\overline{i}} = \begin{tikzpicture}[scale=0.5, baseline=-.5ex]
\node [label=above:$ $] at (0,1) {};
\node [label=below:$ $] at (0,-1) {};
\node [label=left:$\mathbf{i}$] at (-1,0) {};
\node [label=center:$x_i$] at (.5,1) {};
\node [label=center:$\overline{x}_i$] at (.5,-1) {};
\draw [-]
	(0,1) arc (90:270:1);
\draw [-,line width=1mm,red,opacity=0.5]
	(0,1) arc (90:180:1);
\filldraw[black] (-1,0) circle (4pt);
\end{tikzpicture},\quad
\bra{-1}_i \otimes \bra{1}_{\overline{i}} = \begin{tikzpicture}[scale=0.5, baseline=-.5ex]
\node [label=above:$ $] at (0,1) {};
\node [label=below:$ $] at (0,-1) {};
\node [label=left:$\mathbf{i}$] at (-1,0) {};
\node [label=center:$x_i$] at (.5,1) {};
\node [label=center:$\overline{x}_i$] at (.5,-1) {};
\draw [-]
	(0,1) arc (90:270:1);
\draw [-,line width=1mm,red,opacity=0.5]
	(-1,0) arc (180:270:1);
\filldraw[black] (-1,0) circle (4pt);
\end{tikzpicture},\quad
\bra{1}_i \otimes \bra{1}_{\overline{i}} = \begin{tikzpicture}[scale=0.5, baseline=-.5ex]
\node [label=above:$ $] at (0,1) {};
\node [label=below:$ $] at (0,-1) {};
\node [label=left:$\mathbf{i}$] at (-1,0) {};
\node [label=center:$x_i$] at (.5,1) {};
\node [label=center:$\overline{x}_i$] at (.5,-1) {};
\draw [-]
	(0,1) arc (90:270:1);
\draw [-,line width=1mm,red,opacity=0.5]
	(0,1) arc (90:270:1);
\filldraw[black] (-1,0) circle (4pt);
\end{tikzpicture} \text{ and }
\bra{-1}_i \otimes \bra{-1}_{\overline{i}} = \begin{tikzpicture}[scale=0.5, baseline=-.5ex]
\node [label=above:$ $] at (0,1) {};
\node [label=below:$ $] at (0,-1) {};
\node [label=left:$\mathbf{i}$] at (-1,0) {};
\node [label=center:$x_i$] at (.5,1) {};
\node [label=center:$\overline{x}_i$] at (.5,-1) {};
\draw [-]
	(0,1) arc (90:270:1);
\filldraw[black] (-1,0) circle (4pt);
\end{tikzpicture}. \]
\end{lemma}

Note that the assumption on solvability with Fish constant one ensures that $A_1$ and $D_2$ have degree $0$, so there's no need to keep track of the associated spectral parameters in the above expression. We have chosen to retain them to track the contributions from each permutation in the proof below. However, we {\bf do} make use of the independence of spectral parameter in the subsequent proof of Theorem~\ref{thm:rank2.generic}.

\begin{proof}
Following the approach taken in \cite{Wheeler-Zinn-Justin}, we'll verify that each component of $\bra{\mathbf{K}}\mathbf{F}^{-1}$ has the form given on the right of \eqref{eq:boundary}. To begin, then, suppose $j_{\bar{1}}, j_1, j_{\bar{2}}, j_2 \in \{0,1\}$, where 1 denotes the presence of a particle at a vertex, and 0 denotes the lack of a particle. By \eqref{eq:fstarform} and \eqref{eq:Delta}, we have
\begin{align*}
(\bra{\mathbf{K}}\mathbf{F}^{-1})^{j_2j_{\bar{2}}j_1j_{\bar{1}}} &= (\bra{\mathbf{K}} \mathbf{R}_{2\bar{2}1\bar{1}}^{\rho} \prod_{\substack{k<l \\ k,l \in \{\bar{1},1,\bar{2},2\}}}  \Delta_{kl}^{-1}(x_k,x_l) )^{j_2j_{\bar{2}}j_1j_{\bar{1}}}
\\&= (\bra{\mathbf{K}} \mathbf{R}_{2\bar{2}1\bar{1}}^{\rho} )^{j_2j_{\bar{2}}j_1j_{\bar{1}}} \cdot \prod_{\substack{k<l \\ k,l \in \{\bar{1},1,\bar{2},2\}}} b_{j_k,j_l}^{-1}(x_k,x_l),
\end{align*}
where $\rho$ is a permutation of $\{\bar{1},1,\bar{2},2\}$ such that $j_{\rho(1)} = j_{\rho(\overline{1})} = 0$ and $j_{\rho(2)} = j_{\rho(\bar{2})} = 1$.

Since the $R$-matrix weights $a_1(k,j)$ and $a_2(k,j)$ are both equal to 1, it suffices to consider permutations $\rho$ such that $\rho(2)\geq \rho(\overline{2})$ and $\rho(1) \geq \rho(\overline{1})$. This leaves us with six cases: $\rho = (2,\overline{2},1,\overline{1})$, $(2,1,\overline{2},\overline{1})$, $(2,\overline{1},\overline{2},1)$, $(\overline{2},1,2,\overline{1})$, $(\overline{2},\overline{1},2,1)$, $(1,\overline{1},2,\overline{2})$. 
Recall that we are writing the permutation $\rho$ in ``one-line'' notation where the components indicate the respective images of $(2,\overline{2},1,\overline{1})$.
% We will see that $(\bra{\mathbf{K}}\mathbf{F}^{-1})^{j_2j_{\bar{2}}j_1j_{\bar{1}}} = 0$ for two of these six cases, which is why the right hand side of \eqref{eq:boundary} only has four terms in it.

First, suppose that $\rho = (2,\overline{2},1,\overline{1})$ (and hence $j_2 = j_{\bar{2}} = 1$, and $j_1 = j_{\bar{1}} = 0$). Then $(\bra{\mathbf{K}}\mathbf{R}_{2\bar{2}1\bar{1}}^{\rho})^{j_2j_{\bar{2}}j_1j_{\bar{1}}} = A_1(x_1) D_2(x_2)$; diagrammatically, we have

$$\mathcal{Z}\left(
\begin{tikzpicture}[scale=.5,baseline={([yshift=-\the\dimexpr\fontdimen22\textfont2\relax]current bounding box.center)}]
\node [label=right:$x_2$] at (3,2) {};
\node [label=right:$\overline{x}_2$] at (3,1) {};
\node [label=right:$x_1$] at (3,0) {};
\node [label=right:$\overline{x}_1$] at (3,-1) {};
\node [label=left:$\mathbf{2}$] at (1.5,1.5) {};
\node [label=left:$\mathbf{1}$] at (1.5,-.5) {};

\draw [-] 
    (3,-1) -- (2,-1) 
    arc (270:90:.5)
    -- (3,0);

\draw [-] 
    (3,2) -- (2,2)
    arc (90:270:.5)
    -- (3,1);

\filldraw[black] (1.5,1.5) circle (4pt);
\filldraw[black] (1.5,-.5) circle (4pt);
\filldraw[white] (3,2) circle (4pt);
\filldraw[white] (3,1) circle (4pt);
\filldraw[red,opacity=.5] (3,2) circle (4pt);
\filldraw[red,opacity=.5] (3,1) circle (4pt);
\draw[fill=white] (3,0) circle (4pt);
\draw[fill=white] (3,-1) circle (4pt);

\end{tikzpicture}\right) =
\textrm{wt}\left(
\begin{tikzpicture}[scale=.5,baseline={([yshift=-\the\dimexpr\fontdimen22\textfont2\relax]current bounding box.center)}]
\node [label=right:$x_2$] at (3,2) {};
\node [label=right:$\overline{x}_2$] at (3,1) {};
\node [label=right:$x_1$] at (3,0) {};
\node [label=right:$\overline{x}_1$] at (3,-1) {};
\node [label=left:$\mathbf{2}$] at (1.5,1.5) {};
\node [label=left:$\mathbf{1}$] at (1.5,-.5) {};

\draw [-] 
    (3,-1) -- (2,-1) 
    arc (270:90:.5)
    -- (3,0);

\draw [-] 
    (3,2) -- (2,2)
    arc (90:270:.5)
    -- (3,1);

\draw [-,line width=1mm,red,opacity=.5,rounded corners=0pt]
    (3,2) -- (2,2)
    arc (90:270:.5)
    -- (3,1);

\filldraw[black] (1.5,1.5) circle (4pt);
\filldraw[black] (1.5,-.5) circle (4pt);

\end{tikzpicture}\right)
=
A_1(x_1) D_2(x_2).
$$ 
Hence, \[(\bra{\mathbf{K}}\mathbf{F}^{-1})^{j_2j_{\bar{2}}j_1j_{\bar{1}}} = A_1(x_1) D_2(x_2) \cdot \prod_{\substack{k<l \\ k,l \in \{\bar{1},1,\bar{2},2\}}} b_{j_k,j_l}^{-1}(x_k,x_l) = A_1(x_1) D_2(x_2) \cdot \prod_{\{\epsilon_1, \epsilon_2\} \in \{\pm 1\}^2} \frac{x_2^{\epsilon_2}-q\bar{x}_1^{\epsilon_1}}{x_2^{\epsilon_2}-\bar{x}_1^{\epsilon_1}}.\] 

Next, suppose that $\rho = (2,1,\overline{2},\overline{1})$ (so $j_2 = j_1 = 1$, and $j_{\bar{2}} = j_{\bar{1}} = 0$). Note that \begin{equation}\label{eq:bcoeffs2} \prod_{\substack{k<l \\ k,l \in \{\bar{1},1,\bar{2},2\}}} b_{j_k,j_l}^{-1}(x_k,x_l) = \prod_{k,l \in \{\rho(2),\rho(\overline{2})\}} \frac{x_k-q\overline{x}_l}{x_k-\overline{x}_l} \qquad \text{ when $\overline{\rho(2)} \neq \rho(\bar{2})$},\end{equation} so that, in these cases, \begin{equation}\label{eq:twisted.bdry} (\bra{\mathbf{K}}\mathbf{F}^{-1})^{j_2j_{\bar{2}}j_1j_{\bar{1}}} = \left(\bra{\mathbf{K}} \mathbf{R}_{2\bar{2}1\bar{1}}^{\rho}\right)^{j_2j_{\bar{2}}j_1j_{\bar{1}}} \cdot \prod_{k,l \in \{\rho(2),\rho(\overline{2})\}} \frac{x_k-q\overline{x}_l}{x_k-\overline{x}_l}. \end{equation}
We have 

\begin{align*}
(\bra{\mathbf{K}}\mathbf{R}_{2\bar{2}1\bar{1}}^{\rho})^{j_2j_{\bar{2}}j_1j_{\bar{1}}} &= \mathcal{Z}\left(
\begin{tikzpicture}[scale=.5,baseline={([yshift=-\the\dimexpr\fontdimen22\textfont2\relax]current bounding box.center)}]
\node [label=right:$x_2$] at (3,2) {};
\node [label=right:$x_1$] at (3,1) {};
\node [label=right:$\overline{x}_2$] at (3,0) {};
\node [label=right:$\overline{x}_1$] at (3,-1) {};
\node [label=left:$\mathbf{2}$] at (0.5,1.5) {};
\node [label=left:$\mathbf{1}$] at (0.5,-.5) {};
\draw [-]
    (3,-1) -- (1,-1)
    arc(270:90:.5)
    .. controls (1.5,0) and (1.5,1) .. (2,1)
    -- (3,1);
\draw [-]    
    (3,2) -- (1,2)
    arc (90:270:.5)
    .. controls (1.5,1) and (1.5,0) .. (2,0)
    -- (3,0);
%\draw [-,line width=1mm,red,opacity=.5,rounded corners=0pt]
%    (3,0) .. controls (2.6,0) and (2.6,0.5) .. (2.5,0.5)
%    .. controls (2.4,0.5) and (2.4,0) .. (2,0)
%    .. controls (1.5,0) and (1.5,-1) .. (1,-1)
%    -- (0,-1)
%    arc (270:90:.5)
%    .. controls (.5,0) and (.5,1) .. (1,1)
%    .. controls (1.5,1) and (1.5,2) .. (2,2)
%    -- (3,2);
\filldraw[black] (.5,1.5) circle (4pt);
\filldraw[black] (.5,-.5) circle (4pt);
\filldraw[white] (3,2) circle (4pt);
\filldraw[white] (3,1) circle (4pt);
\filldraw[red,opacity=.5] (3,2) circle (4pt);
\filldraw[red,opacity=.5] (3,1) circle (4pt);
\draw[fill=white] (3,0) circle (4pt);
\draw[fill=white] (3,-1) circle (4pt);
\end{tikzpicture}\right)
=
\textrm{wt}\left(
\begin{tikzpicture}[scale=.5,baseline={([yshift=-\the\dimexpr\fontdimen22\textfont2\relax]current bounding box.center)}]
\node [label=right:$x_2$] at (3,2) {};
\node [label=right:$x_1$] at (3,1) {};
\node [label=right:$\overline{x}_2$] at (3,0) {};
\node [label=right:$\overline{x}_1$] at (3,-1) {};
\node [label=left:$\mathbf{2}$] at (0.5,1.5) {};
\node [label=left:$\mathbf{1}$] at (0.5,-.5) {};
\draw [-] 
    (3,-1) -- (1,-1) 
    arc(270:90:.5)
    .. controls (1.5,0) and (1.5,1) .. (2,1) 
    -- (3,1);  
\draw [-]    
    (3,2) -- (1,2)
    arc (90:270:.5)
    .. controls (1.5,1) and (1.5,0) .. (2,0)
    -- (3,0);
\draw [-,line width=1mm,red,opacity=.5,rounded corners=0pt]
    (3,2) -- (1,2)
    arc (90:180:.5);
\draw [-,line width=1mm,red,opacity=.5,rounded corners=0pt]
    (3,1) -- (2,1)
    .. controls (1.5,1) and (1.5,0) .. (1,0)
    arc (90:180:.5);
\filldraw[black] (.5,1.5) circle (4pt);
\filldraw[black] (.5,-.5) circle (4pt);
%\filldraw[white] (3,2) circle (4pt);
%\filldraw[white] (3,1) circle (4pt);
%\filldraw[red,opacity=.5] (3,2) circle (4pt);
%\filldraw[red,opacity=.5] (3,1) circle (4pt);
%\draw[fill=white] (3,0) circle (4pt);
%\draw[fill=white] (3,-1) circle (4pt);
\end{tikzpicture}\right)
+
\textrm{wt}\left(
\begin{tikzpicture}[scale=.5,baseline={([yshift=-\the\dimexpr\fontdimen22\textfont2\relax]current bounding box.center)}]
\node [label=right:$x_2$] at (3,2) {};
\node [label=right:$x_1$] at (3,1) {};
\node [label=right:$\overline{x}_2$] at (3,0) {};
\node [label=right:$\overline{x}_1$] at (3,-1) {};
\node [label=left:$\mathbf{2}$] at (0.5,1.5) {};
\node [label=left:$\mathbf{1}$] at (0.5,-.5) {};
\draw [-] 
    (3,-1) -- (1,-1) 
    arc(270:90:.5)
    .. controls (1.5,0) and (1.5,1) .. (2,1) 
    -- (3,1);
\draw [-]    
    (3,2) -- (1,2)
    arc (90:270:.5)
    .. controls (1.5,1) and (1.5,0) .. (2,0)
    -- (3,0);
\draw [-,line width=1mm,red,opacity=.5,rounded corners=0pt]
    (3,2) -- (1,2)
    arc (90:270:.5)
    .. controls (1.4,1) and (1.4,.5) .. (1.5,.5)
    .. controls (1.6,.5) and (1.6,1) .. (2,1)
    -- (3,1);
\filldraw[black] (.5,1.5) circle (4pt);
\filldraw[black] (.5,-.5) circle (4pt);
%\filldraw[white] (3,2) circle (4pt);
%\filldraw[white] (3,1) circle (4pt);
%\filldraw[red,opacity=.5] (3,2) circle (4pt);
%\filldraw[red,opacity=.5] (3,1) circle (4pt);
%\draw[fill=white] (3,0) circle (4pt);
%\draw[fill=white] (3,-1) circle (4pt);
\end{tikzpicture}\right). \\
&= C_1(x_1) C_2(x_2) \cdot \frac{x_1-\overline{x}_2}{x_1-q\overline{x}_2} + A_1(x_1) D_2(x_2) \cdot \frac{(1-q)\bar{x}_2}{x_1-q\bar{x}_2}. 
\end{align*}
Hence,
\begin{align*}
(\bra{\mathbf{K}}\mathbf{F}^{-1})^{j_2j_{\bar{2}}j_1j_{\bar{1}}} &= (\bra{\mathbf{K}}\mathbf{R}_{2\bar{2}1\bar{1}}^{\rho})^{j_2j_{\bar{2}}j_1j_{\bar{1}}} \cdot \prod_{k,l \in \{2,1\}} \frac{x_k-q\overline{x}_l}{x_k-\overline{x}_l}
\\&=\left(C_1(x_1)C_2(x_2) \cdot \frac{x_1-\overline{x}_2}{x_1-q\overline{x}_2} + A_1(x_1)D_2(x_2) \cdot \frac{(1-q)\bar{x}_2}{x_1-q\bar{x}_2}\right) \cdot \prod_{k,l \in \{2,1\}} \frac{x_k-q\overline{x}_l}{x_k-\overline{x}_l} 
\\&= \left(C_1(x_1)C_2(x_2) + A_1(x_1)D_2(x_2) \cdot \frac{(1-q)\bar{x}_2}{x_1-\bar{x}_2}\right) \cdot \frac{x_1-q\overline{x}_1}{x_1-\overline{x}_1} \cdot \frac{x_2-q\overline{x}_2}{x_2-\overline{x}_2} \cdot \frac{x_2-q\overline{x}_1}{x_2-\overline{x}_1}.
\end{align*}

For $\rho = (2,\overline{1},\overline{2},1)$, consider \[(\bra{\mathbf{K}}\mathbf{R}_{2\bar{2}1\bar{1}}^{\rho})^{j_2j_{\bar{2}}j_1j_{\bar{1}}} = 
\mathcal{Z}\left(
\begin{tikzpicture}[scale=.5,baseline={([yshift=-\the\dimexpr\fontdimen22\textfont2\relax]current bounding box.center)}]
\node [label=right:$x_2$] at (3,2) {};
\node [label=right:$\overline{x}_1$] at (3,1) {};
\node [label=right:$\overline{x}_2$] at (3,0) {};
\node [label=right:$x_1$] at (3,-1) {};
\node [label=left:$\mathbf{2}$] at (-0.5,1.5) {};
\node [label=left:$\mathbf{1}$] at (-0.5,-.5) {};

\draw [-] 
    (3,1) -- (2,1) 
    .. controls (1.5,1) and (1.5,0) .. (1,0) 
    .. controls (0.5,0) and (0.5,-1) .. (0,-1)
    arc (270:90:.5)
    .. controls (0.5,0) and (0.5,-1) .. (1,-1)
    -- (3,-1);
    
\draw [-]    
    (3,2) -- (0,2)
    arc (90:270:.5)
    -- (1,1)
    .. controls (1.5,1) and (1.5,0) .. (2,0)
    -- (3,0);

%\draw [-,line width=1mm,red,opacity=.5,rounded corners=0pt]
%    (3,0) .. controls (2.6,0) and (2.6,0.5) .. (2.5,0.5)
%    .. controls (2.4,0.5) and (2.4,0) .. (2,0)
%    .. controls (1.5,0) and (1.5,-1) .. (1,-1)
%    -- (0,-1)
%    arc (270:90:.5)
%    .. controls (.5,0) and (.5,1) .. (1,1)
%    .. controls (1.5,1) and (1.5,2) .. (2,2)
%    -- (3,2);

\filldraw[black] (-.5,1.5) circle (4pt);
\filldraw[black] (-.5,-.5) circle (4pt);
\filldraw[white] (3,2) circle (4pt);
\filldraw[white] (3,1) circle (4pt);
\filldraw[red,opacity=.5] (3,2) circle (4pt);
\filldraw[red,opacity=.5] (3,1) circle (4pt);
\draw[fill=white] (3,0) circle (4pt);
\draw[fill=white] (3,-1) circle (4pt);
\end{tikzpicture}\right)
= \mathcal{Z}\left(
\begin{tikzpicture}[scale=.5,baseline={([yshift=-\the\dimexpr\fontdimen22\textfont2\relax]current bounding box.center)}]
\node [label=right:$x_2$] at (3,2) {};
\node [label=right:$\overline{x}_1$] at (3,1) {};
\node [label=right:$\overline{x}_2$] at (3,0) {};
\node [label=right:$x_1$] at (3,-1) {};
\node [label=left:$\mathbf{2}$] at (0.5,1.5) {};
\node [label=left:$\mathbf{1}$] at (0.5,-.5) {};

\draw [-] 
    (3,-1) -- (1,-1) 
    arc(270:90:.5)
    .. controls (1.5,0) and (1.5,1) .. (2,1) 
    -- (3,1);
    
\draw [-]    
    (3,2) -- (1,2)
    arc (90:270:.5)
    .. controls (1.5,1) and (1.5,0) .. (2,0)
    -- (3,0);

%\draw [-,line width=1mm,red,opacity=.5,rounded corners=0pt]
%    (3,0) .. controls (2.6,0) and (2.6,0.5) .. (2.5,0.5)
%    .. controls (2.4,0.5) and (2.4,0) .. (2,0)
%    .. controls (1.5,0) and (1.5,-1) .. (1,-1)
%    -- (0,-1)
%    arc (270:90:.5)
%    .. controls (.5,0) and (.5,1) .. (1,1)
%    .. controls (1.5,1) and (1.5,2) .. (2,2)
%    -- (3,2);

\filldraw[black] (.5,1.5) circle (4pt);
\filldraw[black] (.5,-.5) circle (4pt);
\filldraw[white] (3,2) circle (4pt);
\filldraw[white] (3,1) circle (4pt);
\filldraw[red,opacity=.5] (3,2) circle (4pt);
\filldraw[red,opacity=.5] (3,1) circle (4pt);
\draw[fill=white] (3,0) circle (4pt);
\draw[fill=white] (3,-1) circle (4pt);
\end{tikzpicture}\right)
,\] where the second equality comes from the fish equation, noting that the associated constant of proportionality is assumed to be one. Applying the same analysis as in the $\rho = (2,1,\overline{2},\overline{1})$ case, we see that

\begin{align*}
(\bra{\mathbf{K}}\mathbf{F}^{-1})^{j_2j_{\bar{2}}j_1j_{\bar{1}}} &= (\bra{\mathbf{K}}\mathbf{R}_{2\bar{2}1\bar{1}}^{\rho})^{j_2j_{\bar{2}}j_1j_{\bar{1}}} \cdot \prod_{k,l \in \{2,\overline{1}\}} \frac{x_k-q\overline{x}_l}{x_k-\overline{x}_l}
\\&= \left(C_1(\overline{x}_1)C_2(x_2) \cdot \frac{\overline{x}_1-\overline{x}_2}{\overline{x}_1-q\overline{x}_2} + A_1(\overline{x}_1)D_2(x_2) \cdot \frac{(1-q)\bar{x}_2}{\overline{x}_1-q\bar{x}_2}\right) \cdot \prod_{k,l \in \{2,\bar{1}\}} \frac{x_k-q\overline{x}_l}{x_k-\overline{x}_l} 
\\&= \left(C_1(\overline{x}_1)C_2(x_2) + A_1(\overline{x}_1)D_2(x_2) \cdot \frac{(1-q)\bar{x}_2}{\overline{x}_1-\bar{x}_2}\right) \cdot \frac{\overline{x}_1-qx_1}{\overline{x}_1-x_1} \cdot \frac{x_2-q\overline{x}_2}{x_2-\overline{x}_2} \cdot \frac{x_2-qx_1}{x_2-x_1}.
\end{align*}

The cases where $\rho = (\overline{2},1,2,\overline{1})$ and $\rho = (\overline{2},\overline{1},2,1)$ can be checked in a very similar manner to the method used in the previous case, using one and two instances of the fish relation, respectively. In the case of $\rho = (\overline{2},1,2,\overline{1})$, we find
\begin{align*}(\bra{\mathbf{K}}\mathbf{R}_{2\bar{2}1\bar{1}}^{\rho})^{j_2j_{\bar{2}}j_1j_{\bar{1}}} = & 
\left(C_1(x_1)C_2(\overline{x}_2) \cdot \frac{x_1-x_2}{x_1-q x_2} + A_1(x_1)D_2(\overline{x}_2) \cdot \frac{(1-q) x_2}{x_1-q x_2}\right), \end{align*}
while for $\rho = (\overline{2},\overline{1},2,1)$, we obtain
\begin{align*}(\bra{\mathbf{K}}\mathbf{R}_{2\bar{2}1\bar{1}}^{\rho})^{j_2j_{\bar{2}}j_1j_{\bar{1}}} = & 
\left(C_1(\overline{x}_1)C_2(\overline{x}_2) \cdot \frac{\overline{x}_1-x_2}{\overline{x}_1-q x_2} + A_1(\overline{x}_1)D_2(\overline{x}_2) \cdot \frac{(1-q) x_2}{\overline{x}_1-q x_2}\right), \end{align*}
and their overall contribution can again be computed as in~(\ref{eq:twisted.bdry}). 
Finally, suppose that $\rho = (1,\overline{1},2,\overline{2})$, then, by the Caduceus relation for two particles (Lemma \ref{le:caduceus}), we find
$$\mathcal{Z}\left(
\begin{tikzpicture}[scale=.5,baseline={([yshift=-\the\dimexpr\fontdimen22\textfont2\relax]current bounding box.center)}]
\node [label=right:$x_1$] at (3,2) {};
\node [label=right:$\overline{x}_1$] at (3,1) {};
\node [label=right:$x_2$] at (3,0) {};
\node [label=right:$\overline{x}_2$] at (3,-1) {};
\node [label=left:$\mathbf{2}$] at (-0.5,1.5) {};
\node [label=left:$\mathbf{1}$] at (-0.5,-.5) {};

\draw [-] 
    (3,-1) -- (2,-1) 
    .. controls (1.5,-1) and (1.5,0) .. (1,0) 
    .. controls (.5,0) and (.5,1) .. (0,1)
    arc (270:90:.5)
    -- (1,2)
    .. controls (1.5,2) and (1.5,1) .. (2,1)
    .. controls (2.5,1) and (2.5,0) .. (3,0);

\draw [-] 
    (3,2) -- (2,2)
    .. controls (1.5,2) and (1.5,1) .. (1,1)
    .. controls (.5,1) and (.5,0) .. (0,0)
    arc (90:270:.5)
    -- (1,-1)
    .. controls (1.5,-1) and (1.5,0) .. (2,0)
    .. controls (2.5,0) and (2.5,1) .. (3,1);

\filldraw[black] (-.5,1.5) circle (4pt);
\filldraw[black] (-.5,-.5) circle (4pt);
\filldraw[white] (3,2) circle (4pt);
\filldraw[white] (3,1) circle (4pt);
\filldraw[red,opacity=.5] (3,2) circle (4pt);
\filldraw[red,opacity=.5] (3,1) circle (4pt);
\draw[fill=white] (3,0) circle (4pt);
\draw[fill=white] (3,-1) circle (4pt);

\end{tikzpicture}\right) = 
\mathcal{Z}\left(
\begin{tikzpicture}[scale=.5,baseline={([yshift=-\the\dimexpr\fontdimen22\textfont2\relax]current bounding box.center)}]
\node [label=right:$x_1$] at (3,2) {};
\node [label=right:$\overline{x}_1$] at (3,1) {};
\node [label=right:$x_2$] at (3,0) {};
\node [label=right:$\overline{x}_2$] at (3,-1) {};
\node [label=left:$\mathbf{2}$] at (1.5,1.5) {};
\node [label=left:$\mathbf{1}$] at (1.5,-.5) {};

\draw [-] 
    (3,-1) -- (2,-1) 
    arc (270:90:.5)
    -- (3,0);

\draw [-] 
    (3,2) -- (2,2)
    arc (90:270:.5)
    -- (3,1);

\filldraw[black] (1.5,1.5) circle (4pt);
\filldraw[black] (1.5,-.5) circle (4pt);
\filldraw[white] (3,2) circle (4pt);
\filldraw[white] (3,1) circle (4pt);
\filldraw[red,opacity=.5] (3,2) circle (4pt);
\filldraw[red,opacity=.5] (3,1) circle (4pt);
\draw[fill=white] (3,0) circle (4pt);
\draw[fill=white] (3,-1) circle (4pt);

\end{tikzpicture}\right) =
\textrm{wt}\left(
\begin{tikzpicture}[scale=.5,baseline={([yshift=-\the\dimexpr\fontdimen22\textfont2\relax]current bounding box.center)}]
\node [label=right:$x_1$] at (3,2) {};
\node [label=right:$\overline{x}_1$] at (3,1) {};
\node [label=right:$x_2$] at (3,0) {};
\node [label=right:$\overline{x}_2$] at (3,-1) {};
\node [label=left:$\mathbf{2}$] at (1.5,1.5) {};
\node [label=left:$\mathbf{1}$] at (1.5,-.5) {};

\draw [-] 
    (3,-1) -- (2,-1) 
    arc (270:90:.5)
    -- (3,0);

\draw [-] 
    (3,2) -- (2,2)
    arc (90:270:.5)
    -- (3,1);

\draw [-,line width=1mm,red,opacity=.5,rounded corners=0pt]
    (3,2) -- (2,2)
    arc (90:270:.5)
    -- (3,1);

\filldraw[black] (1.5,1.5) circle (4pt);
\filldraw[black] (1.5,-.5) circle (4pt);

\end{tikzpicture}\right)
=
A_1(x_2)D_2(x_1).
$$
Note in particular that the caduceus constant of proportionality is also forced to be one, by our assumption on the Fish constants of proportionality.
With the spectral parameters swapped as a consequence of the Caduceus relation, we have \[(\bra{\mathbf{K}}\mathbf{F}^{-1})^{j_2j_{\bar{2}}j_1j_{\bar{1}}} = A_1(x_2)D_2(x_1)\cdot \prod_{\substack{k>l \\ k,l \in \{\bar{1},1,\bar{2},2\}}} b_{j_k,j_l}^{-1}(x_k,x_l) = A_1(x_2)D_2(x_1) \cdot \prod_{\{\epsilon_1, \epsilon_2\} \in \{\pm 1\}^2} \frac{x_1^{\epsilon_1}-q\bar{x}_2^{\epsilon_2}}{x_1^{\epsilon_1}-\bar{x}_2^{\epsilon_2}}. \qedhere \] 
\end{proof}

\begin{proof}[Proof of Theorem \ref{thm:rank2.generic}] 
First, by \eqref{mainthmbraket} and \eqref{eq:twoparts}, we have $Z(\mathcal{B}_{\lambda}) = \bra{\mathbf{K}} \mathbf{F}^{-1} \mathbf{\widetilde{C}_0} \cdots \mathbf{\widetilde{C}_{\lambda_r}} \ket{\varnothing}$. In rank 2, Theorem \ref{thm:twisted.grid} tells us that \[\bra{\hat{u}} \widetilde{\mathbf{C}}_0 \cdots \widetilde{\mathbf{C}}_{\lambda_r} \ket{\varnothing} = \left\{\begin{array}{ll} \sum_{\sigma \in S_2} y_{\sigma(1)}^{\lambda_1+1}y_{\sigma(2)}^{\lambda_2+1}\cdot \frac{y_{\sigma(1)} - qy_{\sigma(2)}}{y_{\sigma(1)}-y_{\sigma(2)}} & \text{ if $\lambda_1 > \lambda_2$}\\ y_1^{\lambda_1+1}y_2^{\lambda_2+1} & \text{ if $\lambda_1 = \lambda_2,$} \end{array}\right.\] 
where $y_1$ represents the spectral parameter for the particle entering on the top row of the lattice, and $y_2$ represents the spectral parameter particle entering on the second row from the top. Combining this with Lemma \ref{le:K.rank2.generic} in the case where $\lambda_1 \neq \lambda_2$, we see that $\mathcal{Z}(\mathcal{B}_{\lambda}) = \bra{\mathbf{K}} \mathbf{F}^{-1} \mathbf{\widetilde{C}_0} \cdots \mathbf{\widetilde{C}_{\lambda_r}} \ket{\varnothing}$ equals
\begin{align*} 
&\left[\sum_{\{\epsilon_1,\epsilon_2\} \in \{\pm 1\}^2} \frac{x_1^{\epsilon_1}-q\overline{x}_1^{\epsilon_1}}{x_1^{\epsilon_1}-\overline{x}_1^{\epsilon_1}} \cdot \frac{x_2^{\epsilon_2}-q\overline{x}_2^{\epsilon_2}}{x_2^{\epsilon_2}-\overline{x}_2^{\epsilon_2}} \cdot \frac{x_2^{\epsilon_2}-q\overline{x}_1^{\epsilon_1}}{x_2^{\epsilon_2}-\overline{x}_1^{\epsilon_1}}\cdot \left(C_1(x_1^{\epsilon_1}) C_2(x_2^{\epsilon_2}) + A_1D_2 \cdot \frac{(1-q)\bar{x}_2^{\epsilon_2}}{x_1^{\epsilon_1}-\bar{x}_2^{\epsilon_2}} \right)\right.
\\&\cdot \left.\left((x_{1}^{\epsilon_{1}})^{\lambda_1+1}(x_{2}^{\epsilon_{2}})^{\lambda_2+1}\cdot \frac{x_{1}^{\epsilon_{1}} - qx_{2}^{\epsilon_{2}}}{x_{1}^{\epsilon_{1}}-x_{2}^{\epsilon_{2}}} + (x_{2}^{\epsilon_{2}})^{\lambda_1+1}(x_{1}^{\epsilon_{1}})^{\lambda_2+1}\cdot \frac{x_{2}^{\epsilon_{2}} - qx_{1}^{\epsilon_{1}}}{x_{2}^{\epsilon_{2}}-x_{1}^{\epsilon_{1}}}\right)\right]
\\&+A_1D_2\prod_{\{\epsilon_1,\epsilon_2\}\in \{\pm 1\}^2} \frac{x_1^{\epsilon_1}-q\bar{x}_2^{\epsilon_2}}{x_1^{\epsilon_1}-\bar{x}_2^{\epsilon_2}}\cdot \left( x_1^{\lambda_1-\lambda_2}\cdot \frac{x_1 - q\overline{x}_1}{x_1-\overline{x}_1} + x_1^{\lambda_2-\lambda_1}\cdot \frac{\overline{x}_1 - qx_1}{\overline{x}_1-x_1}\right)
\\&+A_1D_2\prod_{\{\epsilon_1,\epsilon_2\}\in \{\pm 1\}^2} \frac{x_2^{\epsilon_2}-q\bar{x}_1^{\epsilon_1}}{x_2^{\epsilon_2}-\bar{x}_1^{\epsilon_1}}\cdot \left( x_2^{\lambda_1-\lambda_2}\cdot \frac{x_2 - q\overline{x}_2}{x_2-\overline{x}_2} + x_2^{\lambda_2-\lambda_1}\cdot \frac{\overline{x}_2 - qx_2}{\overline{x}_2-x_2}\right),
\end{align*}
where we have used the fact that when the Fish constant of proportionality is one, so we are uniform of case 1(a) or 2(b) in Lemma~\ref{le:fish}, then $deg(A_1)=deg(D_2)=0$ according to Lemma~\ref{le:caduceus}, and so we may drop the spectral dependence (i.e., the dependence on $x_i$) in the statement of Lemma~\ref{le:K.rank2.generic}.

To analyze the case where $\lambda_1 \ne \lambda_2$ further, we now separate the calculation into cases -- that where the fish relation is uniform of case 1(a) and of 2(b). In case 1(a), Lemma~\ref{le:caduceus} ensures $C_i$ is independent of spectral parameter, so we may immediately write the partition function above as
\begin{align*}
Z(\mathcal{B}_{\lambda}) = \; \; & C_1C_2\sum_{w\in W} w\left( x_1^{\lambda_1+1}x_2^{\lambda_2+1} \cdot \prod_{\alpha \in \Phi^+} \frac{1-qx^{-\alpha}}{1-x^{-\alpha}}\right)
\\&+A_1D_2\left[\sum_{w \in W} w\left(x_1^{\lambda_1+1}x_2^{\lambda_2+1}\cdot \frac{x_1-qx_2}{x_1-x_2}\cdot \frac{x_2-q\overline{x}_1}{x_2-\overline{x}_1} \cdot \frac{(1-q)\overline{x}_2}{x_1-\overline{x}_2} \cdot \frac{x_1-q\overline{x}_1}{x_1-\overline{x}_1} \cdot \frac{x_2-q\overline{x}_2}{x_2-\overline{x}_2}\right)\right.
\\&+ \left.\sum_{w \in W/\langle s_2 \rangle} w\left(x_1^{\lambda_1-\lambda_2}\cdot \frac{x_1-qx_2}{x_1-x_2}\cdot \frac{\overline{x}_1-qx_2}{\overline{x}_1-x_2} \cdot \frac{x_1-q\overline{x}_2}{x_1-\overline{x}_2} \cdot \frac{\overline{x}_1-q\overline{x}_2}{\overline{x}_1-\overline{x}_2} \cdot \frac{x_1-q\overline{x}_1}{x_1-\overline{x}_1}\right)\right].
\end{align*}
If instead we are uniform of case 2(b), recall that Lemma~\ref{le:caduceus} ensures that $C_i(x)=m_i \overline{x}$ for some $m_i \in \mathbb{C}[q].$ One may then check that:
\begin{align*}
Z(\mathcal{B}_{\lambda}) = \; \; & m_1 m_2 \sum_{w\in W} w\left( x_1^{\lambda_1}x_2^{\lambda_2} \cdot \prod_{\alpha \in \Phi^+} \frac{1-qx^{-\alpha}}{1-x^{-\alpha}}\right)
\\&+A_1D_2\left[\sum_{w \in W} w\left(x_1^{\lambda_1+1}x_2^{\lambda_2+1}\cdot \frac{x_1-qx_2}{x_1-x_2}\cdot \frac{x_2-q\overline{x}_1}{x_2-\overline{x}_1} \cdot \frac{(1-q)\overline{x}_2}{x_1-\overline{x}_2} \cdot \frac{x_1-q\overline{x}_1}{x_1-\overline{x}_1} \cdot \frac{x_2-q\overline{x}_2}{x_2-\overline{x}_2}\right)\right.
\\&+ \left.\sum_{w \in W/\langle s_2 \rangle} w\left(x_1^{\lambda_1-\lambda_2}\cdot \frac{x_1-qx_2}{x_1-x_2}\cdot \frac{\overline{x}_1-qx_2}{\overline{x}_1-x_2} \cdot \frac{x_1-q\overline{x}_2}{x_1-\overline{x}_2} \cdot \frac{\overline{x}_1-q\overline{x}_2}{\overline{x}_1-\overline{x}_2} \cdot \frac{x_1-q\overline{x}_1}{x_1-\overline{x}_1}\right)\right].
\end{align*}
Case 2(b) is handled similarly where again $C_i(x)=m_j \overline{x}$ for some $m_j \in \mathbb{C}[q]$, giving the result in this case.

On the other hand, if $\lambda_1= \lambda_2$, then $\mathcal{Z}(\mathcal{B}_{\lambda})$ is equal to
\begin{align*}
&\sum_{\{\epsilon_1,\epsilon_2\} \in \{\pm 1\}^2} \frac{x_1^{\epsilon_1}-q\overline{x}_1^{\epsilon_1}}{x_1^{\epsilon_1}-\overline{x}_1^{\epsilon_1}} \cdot \frac{x_2^{\epsilon_2}-q\overline{x}_2^{\epsilon_2}}{x_2^{\epsilon_2}-\overline{x}_2^{\epsilon_2}} \cdot \frac{x_2^{\epsilon_2}-q\overline{x}_1^{\epsilon_1}}{x_2^{\epsilon_2}-\overline{x}_1^{\epsilon_1}} \cdot \left(C_1(x_1^{\epsilon_1}) C_2(x_2^{\epsilon_2}) + A_1D_2\frac{(1-q)\overline{x}_2^{\epsilon_2}}{x_1^{\epsilon_1}-\overline{x}_2^{\epsilon_2}}\right)\cdot (x_{1}^{\epsilon_{1}})^{\lambda_1+1}(x_{2}^{\epsilon_{2}})^{\lambda_2+1}
\\&+A_1D_2\prod_{\{\epsilon_1,\epsilon_2\}\in \{\pm 1\}^2} \frac{x_1^{\epsilon_1}-q\bar{x}_2^{\epsilon_2}}{x_1^{\epsilon_1}-\bar{x}_2^{\epsilon_2}} +A_1D_2\prod_{\{\epsilon_1,\epsilon_2\}\in \{\pm 1\}^2} \frac{x_2^{\epsilon_2}-q\bar{x}_1^{\epsilon_1}}{x_2^{\epsilon_2}-\bar{x}_1^{\epsilon_1}}.
\end{align*}
Again, assuming we are in case 1(a) of the Fish relation, then the weights $C_i$ are independent of the $x_i$ and we may write:
\begin{align*}
\mathcal{Z}(\mathcal{B}_{\lambda})= \; \; & \frac{C_1C_2}{1+q}\cdot \sum_{w\in W} w\left( x_1^{\lambda_1+1}x_2^{\lambda_2+1} \cdot \prod_{\alpha \in \Phi^+} \frac{1-qx^{-\alpha}}{1-x^{-\alpha}}\right). 
\\&+\frac{A_1D_2}{1+q}\left[\sum_{w \in W} w\left(x_1^{\lambda_1+1}x_2^{\lambda_2+1}\cdot \frac{x_1-qx_2}{x_1-x_2}\cdot \frac{x_2-q\overline{x}_1}{x_2-\overline{x}_1} \cdot \frac{(1-q)\overline{x}_2}{x_1-\overline{x}_2} \cdot \frac{x_1-q\overline{x}_1}{x_1-\overline{x}_1} \cdot \frac{x_2-q\overline{x}_2}{x_2-\overline{x}_2}\right)\right.
\\&+ \left.\sum_{w \in W/\langle s_2 \rangle} w\left(x_1^{\lambda_1-\lambda_2}\cdot \frac{x_1-qx_2}{x_1-x_2}\cdot \frac{\overline{x}_1-qx_2}{\overline{x}_1-x_2} \cdot \frac{x_1-q\overline{x}_2}{x_1-\overline{x}_2} \cdot \frac{\overline{x}_1-q\overline{x}_2}{\overline{x}_1-\overline{x}_2} \cdot \frac{x_1-q\overline{x}_1}{x_1-\overline{x}_1}\right)\right]. \qedhere
\end{align*}
\end{proof}

In particular, when the bend weights are uniform in Case 2(b), there's a particularly nice choice of weights realizing the $B/C$ Hall-Littlewood polynomials for all partitions in rank two.

\begin{corollary}
    With Boltzmann weights from the trigonometric six-vertex model and bend weights (labeled as in Figure~\ref{fig:gen.bends}) satisfying
    $$ C_i(x) = \overline{x}, \quad A_1 D_2 = 0, \quad A_2 D_1 = q^3 - q,$$
    then for every dominant weight $(\lambda_1, \lambda_2)$ (i.e., pairs with $\lambda_1 \geqslant \lambda_2$), the partition function $\mathcal{Z}(\mathcal{B}_{\lambda})$ is the type $B/C$ Hall-Littlewood polynomial:
    $$ \mathcal{Z}(\mathcal{B}_{\lambda}) = \frac{1}{c_{\lambda}(q)}\sum_{w \in W} w\left(x_1^{\lambda_1}x_2^{\lambda_2}\cdot \prod_{\alpha \in \Phi^+} \frac{1-q\mathbf{x}^{-\alpha}}{1-\mathbf{x}^{-\alpha}}\right) \quad \text{where} \quad c_{\lambda}(q) = \left\{\begin{array}{cc} 1 & \text{if $\lambda_1> \lambda_2$} \\ 1+q & \text{if $\lambda_1 = \lambda_2$}\end{array}\right\}.$$
\end{corollary}

\section{Rank Three Solvable Models}

In contrast to the rank two case, solvability in rank three requires the full generality of all the Lemmas on the caduceus relation featured in Section~\ref{sec:type.B}. Indeed, since Definition \ref{def:solvable} is more restrictive than the definition of a ``rank two solvable" model, there will be fewer solvable models in rank three. As the next result shows, these additional restrictions lead to many more refined conditions on the possible bend weights. Remarkably, we show as a special case of Theorem~\ref{thm:spherical3} that one such choice still results in rank three $B/C$ Hall-Littlewood polynomials for all partitions.

\begin{proposition}\label{prop:rank3wts}
The only possible choices of bend weights to complete the trigonometric six-vertex model to a solvable $B/C$ model in rank three are as follows:
\begin{enumerate}
\item $B_j(x) = -qC_j(x)$ and $A_j(x) = D_j(x) = 0$ for all $j$.
\item $B_j(x) = -x^2C_j(x)$ and $A_j(x) = D_j(x) = 0$ for all $j$.
\item $A_j,B_j,C_j,D_j \in \mathbb{C}[q]$, with $B_j = C_j$ for $j=1,2,3$, and $A_iC_jD_k = C_1C_2C_3$ for all arrangements $(i,j,k)$ of $\{1,2,3\}$, or
\item $A_j,B_j,C_j,D_j \in \mathbb{C}[q]$, with $B_j = C_j$ for $j=1,2,3$, and \newline $\left\{\begin{array}{l} A_1, D_3 = 0 \\ A_3C_2D_1 = A_3C_1D_2 = A_2C_3D_1 = (1-q^2)C_1C_2C_3, \text{ or } \end{array}\right.$
\item $A_j,B_j,C_j,D_j \in \mathbb{C}[q]$, with $B_j = C_j$ for $j=1,2,3$, and \newline $\left\{\begin{array}{l} A_3, D_1 = 0 \\ A_1C_2D_3 = A_2C_1D_3 = A_1C_3D_2 = (1-q^2)C_1C_2C_3, \text{ or } \end{array}\right.$
\item $B_j(x) = qm_jx$, $C_j(x) = m_j\bar{x}$, where $m_j \in \mathbb{C}[q]$ and $A_j, D_j \in \mathbb{C}[q]$ for $j=1,2,3$ and \newline $q^{2(k-j)}m_iA_jD_k = qm_1m_2m_3$ for all arrangements $(i,j,k)$ of $\{1,2,3\}$, or
\item $B_j(x) = qm_jx$, $C_j(x) = m_j\bar{x}$, where $m_j \in \mathbb{C}[q]$ and $A_j, D_j \in \mathbb{C}[q]$ for $j=1,2,3$ and \newline $\left\{\begin{array}{l} A_1,D_3 = 0 \\ m_1A_3D_2 =  m_3A_2D_1 = (q^3-q)m_1m_2m_3 \\ m_2A_3D_1 = (q^5-q^3)m_1m_2m_3, \text{ or } \end{array} \right.$
\item $B_j(x) = qm_jx$, $C_j(x) = m_j\bar{x}$, where $m_j \in \mathbb{C}[q]$ and $A_j, D_j \in \mathbb{C}[q]$ for $j=1,2,3$ and \newline $\left\{\begin{array}{l} A_3,D_1 = 0 \\ m_1A_2D_3 = m_3A_1D_2 = (q^{-1}-q)m_1m_2m_3 \\ m_2A_1D_3 = (q^{-3}-q^{-1})m_1m_2m_3, \text{ or } \end{array} \right.$
\end{enumerate}
\end{proposition}

\begin{proof}

To begin, recall that a type $B/C$ model is of uniform regime and satisfies the two particle caduceus relation if and only if the model satisfies one of the four conditions listed in Lemma \ref{le:caduceus}. Note that cases (1) and (2) above correspond to cases (3) and (4) of Lemma \ref{le:caduceus}, respectively. By Lemma \ref{le:fish}, $A_j(x) = D_j(x) = 0$ in these cases, so that the caduceus condition with 3 particles and 1 particle is trivially satisfied. Hence, the choice of bend weights described cases (1) and (2) complete the trigonometric six-vertex model to a solvable $B/C$ model in rank three.

Suppose instead that a given model has bivalent weights as described in case (1) of Lemma \ref{le:caduceus}. In this case, we have the following set of constraints on the bivalent weights from Lemmas \ref{le:caduceus}, \ref{le:caduceus3}, and \ref{le:caduceus1}:

\begin{equation}\label{eq:cad.condition3}
\left\{\begin{array}{ll} (i) & C_3(q^2(A_1D_2-C_1C_2) - (A_2D_1-C_1C_2)) = 0 \\ (ii) & C_1(q^2(A_2D_3-C_2C_3) - (A_3D_2-C_2C_3)) = 0 \\
(iii) & A_3C_1D_2 = A_3C_2D_1 \\ (iv) & A_1C_2D_3 = A_1C_3D_2 \\ (v) & D_3C_1A_2 = D_3C_2A_1 \\ (vi) & D_1C_2A_3 = D_1C_3A_2. \end{array}\right.\end{equation}

Recall that in this case, $A_j,B_j,C_j,D_j \in \mathbb{C}[q]$ and $B_j = C_j$ for $j=1,2,3$. Note that (iii) and (iv) (resp.~(v) and (vi)) include an extra factor of $A_j$ (resp.~$D_j$) when compared with the result from Lemma \ref{le:caduceus3} (resp. Lemma \ref{le:caduceus1}) due to the presence of the third bivalent vertex in rank three. For similar reasons, there is a factor of $C_3$ in (i) and a factor of $C_1$ in (ii).

Before we begin our case analysis, we note that if $C_j = 0$ for any $j$, then every state in the associated model has weight 0, since the only states in the model that don't have a factor of $C_j$ would have a weight of 0 by (iii)-(vi), and hence we may assume that $C_j \neq 0$ for all $j$. %Hence, we can cancel out the factors of $B_3$ and $B_1$ from (i) and (ii), respectively.

First, suppose that $A_1C_2D_3 \neq 0$ and $A_3C_2D_1 \neq 0$. This case will lead us to scenario (3) above. By (iv) and (v), we have $A_1C_2D_3 = A_1C_3D_2 = A_2C_1D_3$, so $A_1 = \frac{C_1}{C_2}A_2$ and $D_3 = \frac{C_3}{C_2}D_2$. Similarly, by (iii) and (vi), we have $A_3C_2D_1 = A_3C_1D_2 = A_2C_3D_1$, and hence $A_3 = \frac{C_3}{C_2}A_2$ and $D_1 = \frac{C_1}{C_2}D_2$. Thus, $A_1C_2D_3 = A_3C_2D_1 = \frac{C_1C_3}{C_2^2}A_2C_2D_2$, so by (iii)-(vi), we have $A_iC_jD_k = \frac{C_1C_3}{C_2^2}A_2C_2D_2$ where $(i,j,k)$ is an arrangement of $\{1,2,3\}$. Additionally, by (i), we have $(q^2-1)(\frac{C_1C_3}{C_2^2}A_2C_2D_2-C_1C_2C_3) = 0$, so $C_1C_2C_3 = \frac{C_1C_3}{C_2^2}A_2C_2D_2$ as well, which brings us to scenario (3) above.

Now, suppose that $A_1C_2D_3 = 0$. Then, by (iv) and (v), respectively, we have $A_1C_3D_2 = 0$ and $A_2C_1D_3 = 0$. Hence, by (i) and (ii) we have $A_2C_3D_1 = (1-q^2)C_1C_2C_3$ and $A_3C_1D_2 = (1-q^2)C_1C_2C_3$. From this, we see that $A_2, D_2 \neq 0$. Since $A_1C_3D_2 = 0$, we must have $A_1 = 0$, and, since $A_2C_1D_3 = 0$, we must have $D_3 = 0$. Summarizing, if $A_1C_2D_3 = 0$, then 
\begin{equation*}
%\label{eq:rank3bends}
\left\{\begin{array}{l} A_1, D_3 = 0 \\ A_3C_2D_1 = A_3C_1D_2 = A_2C_3D_1 = (1-q^2)C_1C_2C_3, \end{array}\right.
\end{equation*}
which is scenario (4) above.

If, instead, $A_3C_2D_1 = 0$, then due to the symmetry within conditions (i)-(vi), we have 
\begin{equation*}
%\label{eq:rank3bends}
\left\{\begin{array}{l} A_3, D_1 = 0 \\ A_1C_2D_3 = A_2C_1D_3 = A_1C_3D_2 = (1-q^2)C_1C_2C_3, \end{array}\right.
\end{equation*}
which is scenario (5) above. Finally, note that if $A_1C_2D_3 = 0$ and $A_3C_2D_1 = 0$, then $A_1,D_1 = 0$, so by (i) $(1-q^2)C_1C_2C_3 = 0$, which leads to a trivial model, as discussed at the beginning of the proof.

Cases (6), (7), and (8) emerge from an analysis of rank three models that satisfy case (2) of Lemmas \ref{le:caduceus}, \ref{le:caduceus3}, and \ref{le:caduceus1} in much the same way that cases (3), (4), and (5) were established above. 
\end{proof}

Note that if $C_j(x) = 1$ in scenario (1), then we recover the model from \cite{Wheeler-Zinn-Justin}. Meanwhile, scenario (3) is equivalent to having $A_j = B_j = C_j = D_j = 1$, since in this case every state receives the same weight contribution from its bivalent vertices. In Theorem \ref{thm:spherical3}, we show that the zonal spherical function in type C can be realized as the partition function of a model from scenario (4) or as the partition function of a model from scenario (7).

%In cases (2) and (3), there is some flexibility in the choice of weights for $A_j$ and $D_j$. As we saw in the rank two case in Section \ref{sec:rank2}, any state with an ``empty bend" will also contain a ``full bend" at one of the other two bivalent vertices, so any assignment of Boltzmann weights to $A_j$ and $D_j$ that preserves the constraints $B_j = C_j$ and $A_iD_j = (1-q^2)B_1^2$ for the appropriate pairs of vertices would be equivalent to cases (2) or (3).

%\newline $A_j = \left\{\begin{array}{cc} 0 & \text{if $j=1$} \\ 1 & \text{if $j=2,3$} \end{array}\right. \text{, and } D_j = \left\{\begin{array}{cc} (1-q^2)B_j^2 & \text{if $j=1,2$} \\ 0 & \text{if $j=3$} \end{array}\right.$, or
%\newline $A_j = \left\{\begin{array}{cc} 1 & \text{if $j=1,2$} \\ 0 & \text{if $j=3$} \end{array}\right. \text{, and } D_j = \left\{\begin{array}{cc} 0 & \text{if $j=1$} \\ (1-q^2)B_j^2 & \text{if $j=2,3$} \end{array}\right.$, or

\begin{theorem}\label{thm:spherical3}
Let rank $r=3$ and $\lambda = (\lambda_1 \geqslant \lambda_2 \geqslant \lambda_3)$. If $\mathcal{B}_{\lambda}$ is solvable then 
\begin{enumerate}
\item $\displaystyle \mathcal{Z}(\mathcal{B}_\lambda) = \frac{C_1C_2C_3}{c_{\lambda}(q)}\sum_{w \in W} w\left(x_1^{\lambda_1+1}x_2^{\lambda_2+1}x_3^{\lambda_3+1}\cdot \prod_{\alpha \in \Phi^+} \frac{1-qx^{-\alpha}}{1-x^{-\alpha}}\right)$ if $B_{\lambda}$ has bivalent weights as in case (4) of Proposition \ref{prop:rank3wts}, or
\item $\displaystyle \mathcal{Z}(\mathcal{B}_\lambda) = x_1x_2x_3 \frac{C_1(x_1)C_2(x_2)C_3(x_3)}{c_{\lambda}(q)}\sum_{w \in W} w\left(x_1^{\lambda_1}x_2^{\lambda_2}x_3^{\lambda_3}\cdot \prod_{\alpha \in \Phi^+} \frac{1-qx^{-\alpha}}{1-x^{-\alpha}}\right)$ if $B_{\lambda}$ has bivalent weights as in case (7) of Proposition \ref{prop:rank3wts},
\end{enumerate}
where $c_{\lambda}(q) = \left\{\begin{array}{ll} 1 & \text{if $\lambda_1 > \lambda_2 > \lambda_3$} \\ 1+q & \text{if $\lambda_1 = \lambda_2 > \lambda_3$ or $\lambda_1 > \lambda_2 = \lambda_3$} \\ (1+q)(1+q+q^2) & \text{if $\lambda_1 = \lambda_2 = \lambda_3$}. \end{array}\right.$
\end{theorem}

\begin{lemma}\label{le:K_sph3}
Let rank $r=3$. If the bivalent weights satisfy condition (4) or (7) of Proposition \ref{prop:rank3wts}, so that the Fish constant of proportionality $F=1$, then 
\begin{equation} \label{eq:boundary3} \bra{\mathbf{K}} \mathbf{F}^{-1} =\sum_{\{\epsilon_1, \epsilon_2, \epsilon_3\} \in \{\pm 1\}^3} C_1(x_1^{\epsilon_1})C_2(x_2^{\epsilon_2})C_3(x_3^{\epsilon_3}) \prod_{3\geq k\geq l \geq 1} \left(\frac{x_k^{\epsilon_k}-q\overline{x}_l^{\epsilon_l}}{x_k^{\epsilon_k}-\overline{x}_l^{\epsilon_l}}\right) \bigotimes_{i=1}^3 (\bra{\epsilon_i}_i \otimes \bra{-\epsilon_i}_{\overline{i}}),\end{equation}
\end{lemma}
%\begin{equation} \label{eq:boundary3a} \bra{\mathbf{K}} \mathbf{F}^{-1} = C_1(x_1)C_2(x_2)C_3(x_3)\cdot \sum_{\{\epsilon_1, \epsilon_2, \epsilon_3\} \in \{\pm 1\}^3} \prod_{3\geq k\geq l \geq 1} \left(\frac{x_k^{\epsilon_k}-q\overline{x}_l^{\epsilon_l}}{x_k^{\epsilon_k}-\overline{x}_l^{\epsilon_l}}\right) \bigotimes_{i=1}^3 (\bra{\epsilon_i}_i \otimes \bra{-\epsilon_i}_{\overline{i}}),\end{equation} if $B_j(x) = C_j(x)$, and 

\begin{proof}
We will prove \eqref{eq:boundary3} by verifying that each component of $\bra{\mathbf{K}}\mathbf{F}^{-1}$ has the form given on the right of \eqref{eq:boundary3} in a manner similar to our approach in the proof of Lemma \ref{le:K.rank2.generic}. Suppose that $j_{\bar{1}},j_1, j_{\bar{2}},j_2, j_{\bar{3}},j_3\in \{0,1\}$, where 1 denotes the presence of a particle at a vertex, and 0 denotes the lack of a particle. By Equations \eqref{eq:fstarform} and \eqref{eq:Delta}, we have 
\begin{align*}
(\bra{\mathbf{K}}\mathbf{F}^{-1})^{j_3j_{\bar{3}}j_2j_{\bar{2}}j_1j_{\bar{1}}} &= \left(\bra{\mathbf{\mathbf{K}}} \mathbf{R}_{3\bar{3}2\bar{2}1\bar{1}}^{\rho} \prod_{\substack{k<l \\ k,l \in \{\bar{1},1,\bar{2},2,\bar{3},3\}}}  \Delta_{kl}^{-1}(x_k,x_l)\right)^{j_3j_{\bar{3}}j_2j_{\bar{2}}j_1j_{\bar{1}}}
\\&= \left(\bra{\mathbf{\mathbf{K}}} \mathbf{\mathbf{R}}_{3\bar{3}2\bar{2}1\bar{1}}^{\rho}\right)^{j_3j_{\bar{3}}j_2j_{\bar{2}}j_1j_{\bar{1}}} \cdot \prod_{\substack{k<l \\ k,l \in \{\bar{1},1,\bar{2},2,\bar{3},3\}}} b_{j_k,j_l}^{-1}(x_k,x_l)
\end{align*}
where $\rho$ is a permutation of $\{\bar{1},1,\bar{2},2,\bar{3},3\}$ such that $j_{\rho(\bar{2})} = j_{\rho(1)} = j_{\rho(\bar{1})} = 0$ and $j_{\rho(3)} = j_{\rho(\bar{3})} = j_{\rho(2)} = 1$. 

Since the $R$-matrix weights $a_1(k,j)$ and $a_2(k,j)$ are both equal to 1, it suffices to consider permutations $\rho$ such that $\rho(j) \geq \rho(\bar{j})$ for $j=1,2,3$. Thus, we have twenty cases to consider: one for each selection of three elements from the set above. For twelve of these cases, $(\bra{\mathbf{K}}\mathbf{F}^{-1})^{j_3j_{\bar{3}}j_2j_{\bar{2}}j_1j_{\bar{1}}} = 0$, which is why the right hand side of \eqref{eq:boundary3} only has eight terms in it.

First, let's consider the non-zero cases. Let $\rho_0 = (3,2,1,\overline{3},\overline{2},\overline{1})$. Since we are assuming that the bivalent weights satisfy either condition (4) or condition (7), we have that $D_3=A_1=0$, and hence there is only one way to fill the corresponding diagram to yield a non-zero weight. Explicitly, $(\bra{\mathbf{K}}\mathbf{R}_{3\bar{3}2\bar{2}1\bar{1}}^{\rho_0})^{j_3j_{\bar{3}}j_2j_{\bar{2}}j_1j_{\bar{1}}}$ is equal to \begin{align*} 
\mathcal{Z}\left(
\begin{tikzpicture}[scale=.5,baseline={([yshift=-\the\dimexpr\fontdimen22\textfont2\relax]current bounding box.center)}]
\node [label=left:$x_3$] at (-.5,2) {};
\node [label=left:$\overline{x}_3$] at (-.5,1) {};
%\node [label=left:$\vdots$] at (-.7,1.2) {};
\node [label=left:$x_2$] at (-.5,0) {};
\node [label=left:$\overline{x}_2$] at (-.5,-1) {};
\node [label=left:$x_1$] at (-.5,-2) {};
\node [label=left:$\overline{x}_1$] at (-.5,-3) {};
%\node [label=right:$\omega(?)$] at (4,2) {};
%\node [label=right:$\omega(?)$] at (4,1) {};
%\node [label=right:$\omega(?)$] at (4,0) {}; 
\draw [-]
	(0,2) arc (90:270:.5);
\draw [-]
	(0,0) arc (90:270:.5);
\draw [-]
	(0,-2) arc (90:270:.5);
\draw [-] 
    (0,2) -- (2.5,2);
\draw [-] 
    (0,1) .. controls (0.5,1) and (0.5,0) .. (1,0)
    .. controls (1.5,0) and (1.5,-1) .. (2,-1)
    -- (2.5,-1); 
\draw [-]
    (0,0) .. controls (0.5,0) and (0.5,1) .. (1,1)
    -- (2.5,1);
\draw [-]
    (0,-1) .. controls (0.5,-1) and (0.5,-2) .. (1,-2)
    -- (2.5,-2);
\draw [-]
    (0,-2) .. controls (0.5,-2) and (0.5,-1) .. (1,-1)
    .. controls (1.5,-1) and (1.5,0) .. (2,0)
    -- (2.5,0);
\draw [-]
    (0,-3) -- (2.5,-3);
%\draw [-,line width=1mm,red,opacity=.5,rounded corners=0pt]
%    (2,2) -- (2.5,2);
%\draw [-,line width=1mm,red,opacity=.5,rounded corners=0pt]
%    (2,1) -- (1.5,1);
%\draw [-,line width=1mm,red,opacity=.5,rounded corners=0pt]
%    (1,0) -- (1.5,0);
\filldraw[black] (-.5,1.5) circle (4pt);
\filldraw[black] (-.5,-.5) circle (4pt);
\filldraw[black] (-.5,-2.5) circle (4pt);
\filldraw[white] (2.5,2) circle (4pt);
\filldraw[white] (2.5,1) circle (4pt);
\filldraw[white] (2.5,0) circle (4pt);
\filldraw[red,opacity=.5] (2.5,2) circle (4pt);
\filldraw[red,opacity=.5] (2.5,1) circle (4pt);
\filldraw[red,opacity=.5] (2.5,0) circle (4pt);
\draw[fill=white] (2.5,-1) circle (4pt);
\draw[fill=white] (2.5,-2) circle (4pt);
\draw[fill=white] (2.5,-3) circle (4pt);
\end{tikzpicture}\right) = 
\textrm{wt}
\left(
\begin{tikzpicture}[scale=.5,baseline={([yshift=-\the\dimexpr\fontdimen22\textfont2\relax]current bounding box.center)}]
\node [label=left:$x_3$] at (-.5,2) {};
\node [label=left:$\overline{x}_3$] at (-.5,1) {};
%\node [label=left:$\vdots$] at (-.7,1.2) {};
\node [label=left:$x_2$] at (-.5,0) {};
\node [label=left:$\overline{x}_2$] at (-.5,-1) {};
\node [label=left:$x_1$] at (-.5,-2) {};
\node [label=left:$\overline{x}_1$] at (-.5,-3) {};
%\node [label=right:$\omega(?)$] at (4,2) {};
%\node [label=right:$\omega(?)$] at (4,1) {};
%\node [label=right:$\omega(?)$] at (4,0) {}; 
\draw [-]
	(0,2) arc (90:270:.5);
\draw [-]
	(0,0) arc (90:270:.5);
\draw [-]
	(0,-2) arc (90:270:.5);
\draw [-] 
    (0,2) -- (2.5,2);
\draw [-] 
    (0,1) .. controls (0.5,1) and (0.5,0) .. (1,0)
    .. controls (1.5,0) and (1.5,-1) .. (2,-1)
    -- (2.5,-1); 
\draw [-]
    (0,0) .. controls (0.5,0) and (0.5,1) .. (1,1)
    -- (2.5,1);
\draw [-]
    (0,-1) .. controls (0.5,-1) and (0.5,-2) .. (1,-2)
    -- (2.5,-2);
\draw [-]
    (0,-2) .. controls (0.5,-2) and (0.5,-1) .. (1,-1)
    .. controls (1.5,-1) and (1.5,0) .. (2,0)
    -- (2.5,0);
\draw [-]
    (0,-3) -- (2.5,-3);
\draw [-,line width=1mm,red,opacity=.5,rounded corners=0pt]
    (-.5,1.5) arc (180:90:.5)
    -- (2.5,2);
\draw [-,line width=1mm,red,opacity=.5,rounded corners=0pt]
    (-.5,-.5) arc (180:90:.5)
    .. controls (0.5,0) and (0.5,1) .. (1,1)
    -- (2.5,1);
\draw [-,line width=1mm,red,opacity=.5,rounded corners=0pt]
    (-.5,-2.5) arc (180:90:.5)
    .. controls (0.5,-2) and (0.5,-1) .. (1,-1)
    .. controls (1.5,-1) and (1.5,0) .. (2,0)
    -- (2.5,0);
\filldraw[black] (-.5,1.5) circle (4pt);
\filldraw[black] (-.5,-.5) circle (4pt);
\filldraw[black] (-.5,-2.5) circle (4pt);
%\filldraw[white] (2.5,2) circle (4pt);
%\filldraw[white] (2.5,1) circle (4pt);
%\filldraw[white] (2.5,0) circle (4pt);
%\filldraw[red,opacity=.5] (2.5,2) circle (4pt);
%\filldraw[red,opacity=.5] (2.5,1) circle (4pt);
%\filldraw[red,opacity=.5] (2.5,0) circle (4pt);
%\draw[fill=white] (2.5,-1) circle (4pt);
%\draw[fill=white] (2.5,-2) circle (4pt);
%\draw[fill=white] (2.5,-3) circle (4pt);
\end{tikzpicture}\right)
=B_1B_2B_3\frac{x_1-\overline{x}_2}{x_1-q\overline{x}_2} \cdot \frac{x_1-\overline{x}_3}{x_1-q\overline{x}_3} \cdot \frac{x_2-\overline{x}_3}{x_2-q\overline{x}_3}.
\end{align*}
Additionally, we have \begin{equation}\label{eq:bcoeffs3} \prod_{\substack{k<l \\ k,l \in \{\bar{1},1,\bar{2},2,\bar{3},3\}}} b_{j_k,j_l}^{-1}(x_k,x_l) = \prod_{k,l \in \{\rho(3), \rho(\bar{3}), \rho(2)\}} \frac{x_k-q\overline{x}_l}{x_k-\overline{x}_l} \quad \quad \text{ when $\rho(3) =\pm 3,\, \rho(\bar{3}) = \pm 2,\, \rho(2)=\pm 1$}.\end{equation} Thus, for $\rho=\rho_0$,
\begin{align*}
(\bra{\mathbf{K}}\mathbf{F}^{-1})^{j_3j_{\bar{3}}j_2j_{\bar{2}}j_1j_{\bar{1}}} &= C_1(x_1)C_2(x_2)C_3(x_3)\frac{x_1-\overline{x}_2}{x_1-q\overline{x}_2} \cdot \frac{x_1-\overline{x}_3}{x_1-q\overline{x}_3} \cdot \frac{x_2-\overline{x}_3}{x_2-q\overline{x}_3} \cdot \prod_{k,l \in \{3,2,1\}} \frac{x_k-q\overline{x}_l}{x_k-\overline{x}_l}, 
\\&= C_1(x_1)C_2(x_2)C_3(x_3)\prod_{3 \geq k\geq l \geq 1} \frac{x_k-q\overline{x}_l}{x_k-\overline{x}_l}.
\end{align*}
Note that if $\rho= (3^{\epsilon_3}, 2^{\epsilon_2}, 1^{\epsilon_1}, \bar{3}^{\epsilon_3}, \bar{2}^{\epsilon_2}, \bar{1}^{\epsilon_1})$, for some $\epsilon_1,\epsilon_2,\epsilon_3 \in \{\pm 1\}$, then, after applying the Fish equation (Lemma \ref{le:fish}) and carrying out the same analysis as in the $\rho_0$ case above, we have that \[(\bra{\mathbf{K}}\mathbf{F}^{-1})^{j_3j_{\bar{3}}j_2j_{\bar{2}}j_1j_{\bar{1}}} = C_1(x_1^{\epsilon_1})C_2(x_2^{\epsilon_2})C_3(x_3^{\epsilon_3})\prod_{3 \geq k\geq l \geq 1} \frac{x_k^{\epsilon_k}-q\overline{x}_l^{\epsilon_l}}{x_k^{\epsilon_k}-\overline{x}_l^{\epsilon_l}}.\] %This is similar to the approach taken in the $(2,\overline{1}, \overline{2}, 1)$ case in the proof of Lemma \ref{le:K.rank2.generic}.

In the other twelve cases, $(\bra{\mathbf{K}}\mathbf{R}_{3\bar{3}2\bar{2}1\bar{1}}^{\rho})^{j_3j_{\bar{3}}j_2j_{\bar{2}}j_1j_{\bar{1}}} = 0$ so that $(\bra{\mathbf{K}}\mathbf{F}^{-1})^{j_3j_{\bar{3}}j_2j_{\bar{2}}j_1j_{\bar{1}}} = 0$. Let's explain why this is the case. First, recall that $D_3 = A_1 = 0$. Hence, $(\bra{\mathbf{K}}\mathbf{R}_{3\bar{3}2\bar{2}1\bar{1}}^{\rho})^{j_3j_{\bar{3}}j_2j_{\bar{2}}j_1j_{\bar{1}}} = 0$ in the following six cases: 
$$ \rho = (3, \overline{3}, 2, \overline{2}, 1, \overline{1}), (3, \overline{3}, 1, 2, \overline{2}, \overline{1}), (3, 2, \overline{2}, \overline{3}, 1, \overline{1}),(3, \overline{3}, \overline{2}, 2, 1, \overline{1}), (3, \overline{3}, \overline{1}, 2, \overline{2}, 1),(\overline{3}, 2, \overline{2}, 3, 1, \overline{1}).$$  
One application of the Caduceus relation (Lemmas \ref{le:caduceus}, \ref{le:caduceus3}, \ref{le:caduceus1}) shows us that $(\bra{\mathbf{K}}\mathbf{R}_{3\bar{3}2\bar{2}1\bar{1}}^{\rho})^{j_3j_{\bar{3}}j_2j_{\bar{2}}j_1j_{\bar{1}}} = 0$ for $\rho = (3, 1, \overline{1}, \overline{3}, 2, \overline{2})$ and $(2, \overline{2}, 1, 3, \overline{3}, \overline{1})$, and a subsequent application of the Fish relation takes care of $\rho = (\overline{3}, 1, \overline{1}, 3, 2, \overline{2})$ and $(2, \overline{2}, \overline{1}, 3, \overline{3}, 1)$. Finally, if $\rho = (2, 1, \overline{1}, 3, \overline{3}, \overline{2})$, it can be verified directly that $(\bra{\mathbf{K}}\mathbf{R}_{3\bar{3}2\bar{2}1\bar{1}}^{\rho})^{j_3j_{\bar{3}}j_2j_{\bar{2}}j_1j_{\bar{1}}} = 0$ using a method similar to that in the proof of Lemma \ref{le:caduceus}; $(\bra{\mathbf{K}}\mathbf{R}_{3\bar{3}2\bar{2}1\bar{1}}^{\rho})^{j_3j_{\bar{3}}j_2j_{\bar{2}}j_1j_{\bar{1}}} = 0$ for $\rho = (\overline{2}, 1, \overline{1}, 3, \overline{3}, 2)$ by one further application of the fish relation. See Figure \ref{fig:rank.3.examples} for a depiction of this argument. 

\begin{figure}[ht!]
$$\begin{tikzpicture}[scale=.5,baseline={([yshift=-\the\dimexpr\fontdimen22\textfont2\relax]current bounding box.center)}]
\node [label=left:$x_3$] at (-.5,2) {};
\node [label=left:$\overline{x}_3$] at (-.5,1) {};
%\node [label=left:$\vdots$] at (-.7,1.2) {};
\node [label=left:$x_2$] at (-.5,0) {};
\node [label=left:$\overline{x}_2$] at (-.5,-1) {};
\node [label=left:$x_1$] at (-.5,-2) {};
\node [label=left:$\overline{x}_1$] at (-.5,-3) {};
%\node [label=right:$\omega(?)$] at (4,2) {};
%\node [label=right:$\omega(?)$] at (4,1) {};
%\node [label=right:$\omega(?)$] at (4,0) {}; 
\draw [-]
	(0,2) arc (90:270:.5);
\draw [-]
	(0,0) arc (90:270:.5);
\draw [-]
	(0,-2) arc (90:270:.5);

\draw [-] 
    (0,2) -- (1.5,2);
    
\draw [-] 
    (0,1) -- (1.5,1);
    
\draw [-]
    (0,0) .. controls (0.5,0) and (0.5,-1) .. (1,-1)
    -- (1.5,-1);
    
\draw [-]
    (0,-1) .. controls (0.5,-1) and (0.5,0) .. (1,0)
    -- (1.5,0);
    
\draw [-]
    (0,-2) -- (1.5,-2);
    
\draw [-]
    (0,-3) -- (1.5,-3);
    
\filldraw[white] (1.5,2) circle (4pt);
\filldraw[white] (1.5,1) circle (4pt);
\filldraw[white] (1.5,0) circle (4pt);
\filldraw[red,opacity=.5] (1.5,2) circle (4pt);
\filldraw[red,opacity=.5] (1.5,1) circle (4pt);
\filldraw[red,opacity=.5] (1.5,0) circle (4pt);
\draw[fill=white] (1.5,-1) circle (4pt);
\draw[fill=white] (1.5,-2) circle (4pt);
\draw[fill=white] (1.5,-3) circle (4pt);
    
\filldraw[black] (-.5,1.5) circle (4pt);
\filldraw[black] (-.5,-.5) circle (4pt);
\filldraw[black] (-.5,-2.5) circle (4pt);

\end{tikzpicture} 
\qquad \qquad 
\begin{tikzpicture}[scale=.5,baseline={([yshift=-\the\dimexpr\fontdimen22\textfont2\relax]current bounding box.center)}]
\node [label=left:$x_3$] at (-.5,2) {};
\node [label=left:$\overline{x}_3$] at (-.5,1) {};
%\node [label=left:$\vdots$] at (-.7,1.2) {};
\node [label=left:$x_2$] at (-.5,0) {};
\node [label=left:$\overline{x}_2$] at (-.5,-1) {};
\node [label=left:$x_1$] at (-.5,-2) {};
\node [label=left:$\overline{x}_1$] at (-.5,-3) {};
%\node [label=right:$\omega(?)$] at (4,2) {};
%\node [label=right:$\omega(?)$] at (4,1) {};
%\node [label=right:$\omega(?)$] at (4,0) {}; 
\draw [-]
	(0,2) arc (90:270:.5);
\draw [-]
	(0,0) arc (90:270:.5);
\draw [-]
	(0,-2) arc (90:270:.5);

\draw [-] 
    (0,2) -- (5.5,2);

\draw [-] 
    (0,1) -- (3,1)
    .. controls (3.5,1) and (3.5,0) .. (4,0)
    .. controls (4.5,0) and (4.5,-1) .. (5,-1)
    -- (5.5,-1);
    
\draw [-]
    (0,0) -- (1,0)
    .. controls (1.5,0) and (1.5,-1) .. (2,-1)
    .. controls (2.5,-1) and (2.5,-2) .. (3,-2)
    -- (5.5,-2);
    
\draw [-]
    (0,-1) .. controls (0.5,-1) and (0.5,-2) .. (1,-2)
    .. controls (1.5,-2) and (1.5,-3) .. (2,-3)
    -- (5.5,-3);
    
\draw [-]
    (0,-2) .. controls (0.5,-2) and (0.5,-1) .. (1,-1)
    .. controls (1.5,-1) and (1.5,0) .. (2,0)
    -- (3,0)
    .. controls (3.5,0) and (3.5,1) .. (4,1)
    -- (5.5,1);
    
\draw [-]
    (0,-3) -- (1,-3)
    .. controls (1.5,-3) and (1.5,-2) .. (2,-2)
    .. controls (2.5,-2) and (2.5,-1) .. (3,-1)
    -- (4,-1)
    .. controls (4.5,-1) and (4.5,0) .. (5,0)
    -- (5.5,0);
    
\filldraw[white] (5.5,2) circle (4pt);
\filldraw[white] (5.5,1) circle (4pt);
\filldraw[white] (5.5,0) circle (4pt);
\filldraw[red,opacity=.5] (5.5,2) circle (4pt);
\filldraw[red,opacity=.5] (5.5,1) circle (4pt);
\filldraw[red,opacity=.5] (5.5,0) circle (4pt);
\draw[fill=white] (5.5,-1) circle (4pt);
\draw[fill=white] (5.5,-2) circle (4pt);
\draw[fill=white] (5.5,-3) circle (4pt);
    
\filldraw[black] (-.5,1.5) circle (4pt);
\filldraw[black] (-.5,-.5) circle (4pt);
\filldraw[black] (-.5,-2.5) circle (4pt);

\end{tikzpicture}\qquad \qquad 
\begin{tikzpicture}[scale=.5,baseline={([yshift=-\the\dimexpr\fontdimen22\textfont2\relax]current bounding box.center)}]
\node [label=left:$x_3$] at (-.5,2) {};
\node [label=left:$\overline{x}_3$] at (-.5,1) {};
%\node [label=left:$\vdots$] at (-.7,1.2) {};
\node [label=left:$x_2$] at (-.5,0) {};
\node [label=left:$\overline{x}_2$] at (-.5,-1) {};
\node [label=left:$x_1$] at (-.5,-2) {};
\node [label=left:$\overline{x}_1$] at (-.5,-3) {};
%\node [label=right:$\omega(?)$] at (4,2) {};
%\node [label=right:$\omega(?)$] at (4,1) {};
%\node [label=right:$\omega(?)$] at (4,0) {}; 
\draw [-]
	(0,2) arc (90:270:.5);
\draw [-]
	(0,0) arc (90:270:.5);
\draw [-]
	(0,-2) arc (90:270:.5);

\draw [-] 
    (0,2) -- (1,2)
    .. controls (1.5,2) and (1.5,1) .. (2,1)
    .. controls (2.5,1) and (2.5,0) .. (3,0)
    .. controls (3.5,0) and (3.5,-1) .. (4,-1)
    -- (4.5,-1);
    
\draw [-] 
    (0,1) .. controls (0.5,1) and (0.5,0) .. (1,0)
    .. controls (1.5,0) and (1.5,-1) .. (2,-1)
    .. controls (2.5,-1) and (2.5,-2) .. (3,-2)
    -- (4.5,-2);
    
\draw [-]
    (0,0) .. controls (0.5,0) and (0.5,1) .. (1,1)
    .. controls (1.5,1) and (1.5,2) .. (2,2)
    -- (4.5,2);
    
\draw [-]
    (0,-1) .. controls (0.5,-1) and (0.5,-2) .. (1,-2)
    .. controls (1.5,-2) and (1.5,-3) .. (2,-3)
    -- (4.5,-3);
    
\draw [-]
    (0,-2) .. controls (0.5,-2) and (0.5,-1) .. (1,-1)
    .. controls (1.5,-1) and (1.5,0) .. (2,0)
    .. controls (2.5,0) and (2.5,1) .. (3,1)
    -- (4.5,1);
    
\draw [-]
    (0,-3) -- (1,-3)
    .. controls (1.5,-3) and (1.5,-2) .. (2,-2)
    .. controls (2.5,-2) and (2.5,-1) .. (3,-1)
    .. controls (3.5,-1) and (3.5,0) .. (4,0)
    -- (4.5,0);
    
\filldraw[white] (4.5,2) circle (4pt);
\filldraw[white] (4.5,1) circle (4pt);
\filldraw[white] (4.5,0) circle (4pt);
\filldraw[red,opacity=.5] (4.5,2) circle (4pt);
\filldraw[red,opacity=.5] (4.5,1) circle (4pt);
\filldraw[red,opacity=.5] (4.5,0) circle (4pt);
\draw[fill=white] (4.5,-1) circle (4pt);
\draw[fill=white] (4.5,-2) circle (4pt);
\draw[fill=white] (4.5,-3) circle (4pt);
    
\filldraw[black] (-.5,1.5) circle (4pt);
\filldraw[black] (-.5,-.5) circle (4pt);
\filldraw[black] (-.5,-2.5) circle (4pt);

\end{tikzpicture} 
$$
\caption{Three diagrams relevant to the computation of $(\bra{\mathbf{K}}\mathbf{R}_{3\bar{3}2\bar{2}1\bar{1}}^{\rho})^{j_3j_{\bar{3}}j_2j_{\bar{2}}j_1j_{\bar{1}}}$, for $\rho = (3,\bar{3},\bar{2},2,1,\bar{1})$ (on the left), $\rho=(3,1,\bar{1},\bar{3},2,\bar{2})$ (in the middle), and $\rho=(2,1,\overline{1},3,\overline{3},\overline{2})$ (on the right).  Notice that for the first, the sum over all possible states is zero since any admissible filling would involve filled entries along the top two rows, and then one uses $D_3=0$.  For the middle diagram, an application of the cadaceus allows us to replace this configuration with the one corresponding to $\rho = (3,2,\bar{2},\bar{3},1,\bar{1})$ which we know vanishes since $A_1=0$.  Finally, for the last diagram, we compute the sum over all possible states manually to arrive at the value $0$.}\label{fig:rank.3.examples}
\end{figure}
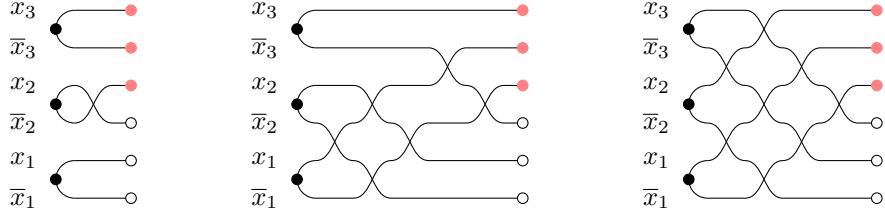

The proof of the vanishing of $(\bra{\mathbf{K}}\mathbf{R}_{3\bar{3}2\bar{2}1\bar{1}}^{\rho})^{j_3j_{\bar{3}}j_2j_{\bar{2}}j_1j_{\bar{1}}}$ in the case where $\rho=(2,1,\overline{1},3,\overline{3},\overline{2})$ is done by writing out the (nineteen) non-zero states with this boundary condition and then evaluating their sum directly.  It would be interesting to determine if there were a more systematic approach to handling the vanishing of this case, as it might translate to more elaborate configurations in higher rank cases as well.
\end{proof}

With the value of $\bra{\mathbf{K}} \mathbf{F}^{-1}$ established, Theorem \ref{thm:spherical3} follows from Lemma \ref{le:K_sph3} much as Theorem \ref{thm:rank2.generic} follows from Lemma \ref{le:K.rank2.generic}. 

Before we move on, we will make quick notes about scenarios (2), (5), (6), and (8) in Proposition \ref{prop:rank3wts}, which have been otherwise ignored in this section. For scenario (2), we again recover the model from \cite{Wheeler-Zinn-Justin}, in analogy with case (3) and case (4) of Theorem \ref{thm:spherical1}. Scenarios (5) and (8) yield symmetric functions with convoluted alternator expressions. Finally, scenario (6) leads to a model similar to the $A_j=B_j=C_j=D_j=1$ model that results from scenario (3).

\section{Higher Rank Models}

In this section we consider solvable type $B/C$ models, as defined in Definition \ref{def:solvable}, in rank $r>3$. Perhaps surprisingly, there are fewer solvable models when $r>3$ than there are solvable rank three models. In particular, for all $r>3$, there is no solvable type $B/C$ model whose partition function gives values of the rank $r$ zonal spherical function. As shown in the proof of Proposition \ref{prop:higher.rank} below, this is a consequence of the caduceus relations. In brief, our results will show that when there are a sufficient number of bivalent vertices in a non-trivial model (i.e.~when the rank is sufficiently large), satisfying the caduceus relations forces all bivalent weights $A_j$ and $D_j$ to be non-zero. Contrast this with Theorem~\ref{thm:rank2.generic} ($r=2$) and Theorem \ref{thm:spherical3} ($r=3$); in these cases, the solvable models whose partition functions match the zonal spherical function satisfy $A_iD_j=0$ for some $i$ and $j$. The next result shows that the tipping point for when this can no longer be arranged occurs when the rank $r>3$.

\begin{proposition}\label{prop:higher.rank}
The only possible choices of bend weights to complete the trigonometric six-vertex model to a solvable $B/C$ model in rank $r>3$ are as follows:
\begin{enumerate}
\item $B_j(x) = -qC_j(x)$ and $A_j(x) = D_j(x) = 0$ for all $j$.
\item $B_j(x) = -x^2C_j(x)$ and $A_j(x) = D_j(x) = 0$ for all $j$.
\item $A_j, B_j, C_j, D_j \in \mathbb{C}[q]$, with $B_j = C_j$ and \[\frac{A_i}{C_i}\cdot \frac{D_j}{C_j} \cdot \prod_{k=1}^r C_k = \prod_{k=1}^r C_k \text{ for any $i \neq j$.}\]
\item $B_j(x) = qm_jx$, $C_j(x) = m_j\bar{x}$, where $m_j \in \mathbb{C}[q]$ and $A_j, D_j \in \mathbb{C}[q]$ for all $j$ and \[ \frac{A_i}{C_i}\cdot \frac{D_j}{C_j} \cdot \prod_{k=1}^r C_k = q^{1-2(j-i)}\prod_{k=1}^r C_k \text{ for any $i \neq j$.} \]
\item $B_j = C_j = 0$, $A_j, D_j \in \mathbb{C}[q]$ with $q^2A_jD_{j+1} = A_{j+1}D_j$ for all $j$, and $r$ is even.
\end{enumerate}
\end{proposition}

\begin{proof}

As in the proof of Proposition \ref{prop:rank3wts}, we begin by noting that a type $B/C$ model of uniform regime satisfies the two particle caduceus relation if and only if it satisfies one of the four conditions listed in Lemma \ref{le:caduceus}. Note that scenarios (1) and (2) above correspond to cases (3) and (4) of Lemma \ref{le:caduceus}, respectively. Since the bivalent weights in scenarios (1) and (2) above trivially satisfy the three particle caduceus relation and the one particle caduceus relation, we see that these choices of bend weights complete the trigonometric six-vertex model to a solvable $B/C$ model in rank $r>3$.

Suppose instead that a given model has bivalent weights as described in case (1) of Lemma \ref{le:caduceus}, such that $C_j \neq 0$ for each $j$. (Later on, we will show that if $C_j = 0$ for some $j$ then the model in question must fall into scenario (5) above.) Then, in order to have a solvable B/C model in rank $r>3$, the following constraints on the bivalent weights must be satisfied:
\begin{equation}\label{eq:cad.conditionN}
\begin{array}{lll} (i) & q^2(A_jD_{j+1}-C_jC_{j+1})=(A_{j+1}D_j-C_jC_{j+1}) & \text{for $j \in [1,r-1]$.} \\ (ii) & A_jC_kD_{k+1} = A_jC_{k+1}D_k & \text{for $j \in [1,r]$, $k \in [1,r-1]$, $j\neq k,k+1$} \\ (iii) & C_kA_{k+1}D_j = A_kC_{k+1}D_j & \text{for $j \in [1,r]$, $k \in [1,r-1]$, $j\neq k,k+1$} \end{array} \end{equation}

Constraint (i) is a restatement of the result of Lemma \ref{le:caduceus}. Constraint (ii) is a consequence of Lemma \ref{le:caduceus3}; note that this constraint contains an extra factor of $A_j$ when compared with the result from Lemma \ref{le:caduceus3} due to particle conservation. For example, if $r=4$, we have $D_1C_2A_3C_4 = C_1D_2A_3C_4$ and $D_1C_2C_3A_4 = C_1D_2C_3A_4$ both by Lemma \ref{le:caduceus3}. Similarly, constraint (iii) is a consequence of Lemma \ref{le:caduceus1}. 

%We will show that the only choice of bend weights that satisfies these constraints and leads to a non-trivial B/C model satisfies the additional requirement $A_jD_j = B_j^2$ for all $j$. In particular, 

Since $C_j \in \mathbb{C}[q]$ and $C_j \neq 0$ for each $j$, (ii) and (iii) imply that, \begin{equation}\label{eq:rank.r.doubles1} A_iD_j = v_{ij}(q) A_kD_{\ell} \text{ if $1 \leq i < j \leq n$ and $1 \leq k < \ell \leq n$}\end{equation} where $v_{ij}(q)$ is some rational function in $q$. As an example when $r=4$, we can repeatedly apply (ii) and (iii) to establish the following equalities: \[A_1D_2C_3C_4 = A_1C_2D_3C_4 = A_1C_2C_3D_4 = C_1A_2C_3D_4 = C_1C_2A_3D_4.\] Similarly, \begin{equation}\label{eq:rank.r.doubles2} A_iD_j = w_{ij}(q) A_kD_{\ell} \text{ if $i>j$ and $k>\ell$} \end{equation} for some rational function $w_{ij}(q)$. 

If $A_1C_2D_3 \neq 0$ and $A_3C_2D_1 \neq 0$, then, using \eqref{eq:rank.r.doubles1} and \eqref{eq:rank.r.doubles2} and an argument analogous to the one given to address scenario (3) in Proposition \ref{prop:rank3wts} can be used to prove that \[\frac{A_i}{C_i}\cdot \frac{D_j}{C_j} \cdot \prod_{k=1}^r C_k = \prod_{k=1}^r C_k\] for any $i\neq j$, which is scenario (3) above. On the other hand, it is clear that a type $B/C$ model satisfying the condition above along with the assumption $B_j = C_j$ for all $j$ is solvable.

Now suppose towards a contradiction that $A_1D_2 = 0$. Then, by \eqref{eq:rank.r.doubles1}, $A_3D_4 = 0$ as well. From (i), we have that $A_2D_1 = (1-q^2)C_1C_2$ and $A_4D_3 = (1-q^2)C_3C_4$, and hence $A_2,A_4,D_1,D_3 \neq 0$. On the other hand, since $A_1D_2 = 0$, $A_2D_3$ must be 0 as well, which is a contradiction. Similarly, if we assume that $A_2D_1 = 0$ and conditions (i), (ii), and (iii), then we reach a contradiction. Thus, in order to have a solvable rank $r>3$ model, we must have $A_j \neq 0$ and $D_j \neq 0$ for all $j$.

If instead we consider a model with bivalent weights as described in case (2) of Lemma \ref{le:caduceus}, such that $C_j \neq 0$ for any $j$, then an analysis similar to the one given above where the bivalent weights satisfy case (1) of Lemma \ref{le:caduceus} leads us to scenario (4) in the statement of this proposition.

Finally, observe that, if $C_j = 0$ for some $j$, then one can use (ii) and (iii) to show that every state that contains a factor of $B_k$ or $C_k$ for any $k$ must evaluate to zero. For example, if $r=4$ and $C_1 = 0$, then \[A_1D_2C_3C_4 = A_1C_2D_3C_4 = C_1A_2D_3C_4 = 0.\] Hence, the only non-zero states must consist only of vertices with weight $A_j$ or $D_j$; furthermore, such states occur only if the rank $r$ is even. Thus, if $C_j=0$ for some $j$, then the model in question falls into scenario~(5). 
\end{proof}

\bibliographystyle{habbrv} 
\bibliography{blattice}

\end{document}